\documentclass[10pt,amsfonts, epsfig]{amsart}

\usepackage{amsmath, amscd, amssymb}
\usepackage{graphpap, color}
\usepackage{mathrsfs}
\usepackage{pstricks}
\usepackage{color}
\usepackage{cancel}
\usepackage[mathscr]{eucal}
\usepackage{pstricks}
\usepackage{color}
\usepackage{cancel}
\usepackage{verbatim}

\usepackage{bbm}
\usepackage{amscd}

\usepackage{amsmath, amscd, amssymb,stmaryrd,mathbbol}
\usepackage[frame,cmtip,arrow,matrix,line,graph,curve]{xy}
\usepackage{graphpap, color}
\usepackage[mathscr]{eucal}
\usepackage{mathrsfs}%,mathabx}
\usepackage{graphicx}

\numberwithin{equation}{section}

\def\basic{\mathrm{pri}}

\def\Hilb{\mathrm{Hilb}}

\def\dual{^{\vee}}

\def\reg{\mathrm{reg}}

\def\sta{^\ast}
\def\st{^{\mathrm{st}}}
\def\supp{\mathrm{Supp}}

\def\upmo{^{-1}}
\def\sta{^{\ast}}
\def\dpri{^{\prime\prime}}

\def\mm{{\mathfrak m}}
\def\sta{^*}
\def\inc{\mathrm{inc}}

\def\ind{{\mathrm{ind}}}

\def\lbe{_{\beta}}

\def\lra{\longrightarrow}

\def\lsta{_{\ast}}

\def\beq{\begin{equation}}
\def\eeq{\end{equation}}
\def\and{\quad\text{and}\quad}
\def\bl{\bigl(}
\def\br{\bigr)}
\def\defeq{:=}

\def\sub{\subset}
\def\lab#1{\label{#1}[{#1}]\  }

\def\Ao{{{\mathbb A}^{\!1}}}

\def\mapright#1{\,\smash{\mathop{\lra}\limits^{#1}}\,}

\def\Chow{\text{Chow}\,}
\def\lam{\lambda}
\def\Lam{\Lambda}
\def\nlc{{\rm {n.l.c.}}}
\def\dpri{^{\prime\prime}}
\def\lra{\longrightarrow}
\def\upmo{^{-1}}
\def\ti{\tilde}

\def\lalp{_\alpha}
\def\lbe{_\beta}

\newcommand{\CC}{\mathbb{C}}

\newcommand{\II}{\mathbb{I}}
\newcommand{\JJ}{\mathbb{J}}
\newcommand{\kk}{\mathbb{k}}

\newcommand{\NN}{\mathbb{N}}
\newcommand{\PP}{\mathbb{P}}
\newcommand{\QQ}{\mathbb{Q}}
\newcommand{\RR}{\mathbb{R}}
\newcommand{\TT}{\mathbb{T}}
\newcommand{\ZZ}{\mathbb{Z}}

\def\cC{\mathcal{C}}

\def\cH{\mathcal H}

\def\bH{\mathcal H}
\def\cK{\mathcal K}

\def\bM{\mathbf M}

\DeclareMathOperator{\diag}{diag} \DeclareMathOperator{\Aut}{Aut}

 \DeclareMathOperator{\Br}{Br}

\newtheorem{prop}{Proposition}[section]

\newtheorem{theo}[prop]{Theorem}

\newtheorem{clai}[prop]{Claim}
\newtheorem{lemm}[prop]{Lemma}
\newtheorem{coro}[prop]{Corollary}
\newtheorem{rema}[prop]{Remark}
\newtheorem{exam}[prop]{Example}
\newtheorem{defi}[prop]{Definition}

\newtheorem{subl}[prop]{Sublemma}

\newtheorem{pf-thm}[prop]{proof of theorem}
\def\fx{\mathfrak x}

\def\sO{{\mathscr O}}

\def\sL{{\mathscr L}}
\def\sI{{\mathscr I}}
\def\sO{\mathscr{O}}
\def\sI{\mathscr{I}}

\def\sE{\mathscr{E}}
\def\sF{\mathscr{F}}
\def\cM{\mathcal{M}}

\def\sR{\mathscr{R}}
\def\sX{\mathscr{X}}

\def\bs{{\mathbf s}}

\def\underx{\underline{x}}
\def\undera{\underline{a}}

\def\primary{^{\mathrm{prim}}}
\def\kk{\Bbbk}
\def\len{\mathrm{length}}

\def\bax{\ba\cdot\bx}

\def\xAo{{X\times\Ao}}
\def\Ycomp{Y^\complement}

\def\vsp{\vskip3pt}
\let\lab=\label
\def\Bl{\Bigl(}
\def\Br{\Bigr)}

\def\codeg{\mathrm{codeg}}
\def\comp{^{\complement}}
\def\prim{^{\mathrm{pri}}}
\let\primary=\prim

\begin{document}
\title{Hilbert-Mumford criterion for nodal curves}

\author{Jun Li}
\address{Stanford University, Stanford, CA94305, USA}
\email{jli@stanford.edu}

\author{Xiaowei Wang}
\address{Rutgers University, Newark, NJ07102, USA}
\email{xiaowwan@rutgers.edu}

\date{\today}

\maketitle
\begin{abstract}
We prove by Hilbert-Mumford criterion that a slope stable polarized weighted pointed nodal curve is Chow asymptotic
stable. This generalizes the result of Caporaso on stability of polarized nodal curves, and of Hasset on weighted pointed
stable curves polarized by the weighted dualizing sheaves. It also solved a question raised by Mumford and Gieseker to
prove the Chow asymptotic stability of stable nodal curves by Hilbert-Mumford criterion.
\end{abstract}

\def\bmg{\overline{\mathcal{M}}_{g}}
\def\CC{\mathbb{C}}

\section{Introduction and summary of main result}\label{intro}

\def\bx{\mathbf x}
\def\ba{\mathbf a}
\let\underx=\bx
\let\undera=\ba
\def\txAo{{\ti X\times\Ao}}
\def\tiX{{\ti X}}
\def\lra{\longrightarrow}
\def\bolds{{\mathbf s}}
\def\ordq{w(\ti \sI,q)}

In late seventy, Mumford \cite{Mum} and Gieseker \cite{Gie} constructed
the coarse moduli space $\overline{\mathcal{M}}_{g}$ of stable curves
% and surfaces of general type by applying
using Mumford's Geometric Invariant Theory (GIT).
{They proved the stability of smooth curves by verifying Hilbert-Mumford stability criterion;
for nodal curves, they proved the stability indirectly by using semi-stable replacement
and the numerical criterion to
rule out curves with worse than nodal singularities.}
This construction has been  very successful and  is widely adopted subsequently for studying related to stability of curves, for instance,
Caporaso's proof of asymptotic stablility of nodal curves \cite{Cap}.

%Despite various generalizations and improvement, so far the stability of nodal
%curve and its analogue are are settled by semistable replacement.
In this paper,
we will prove the Chow asymptotic stability of weighted pointed nodal curves by verifying
Hilbert-Mumford criterion directly. As an application, we provide a GIT construction of the moduli of
weighted pointed stable curves. An interesting consequence of this construction is that the
GIT closure of the moduli of weighted pointed smooth curves, using Chow asymptotic stability,
is identical to Hassett's coarse moduli of  weighted pointed stable curves; nevertheless, its universal
family includes strictly semistable weighted pointed nodal curves.

Another application of our stability study is showing that a polarized nodal curve is
$K$-stable (c.f. Section \ref{K-stable}) if and only if the polarization is numerically equivalent to a multiple of its
dualizing sheaf. This generalizes a theorem of Odaka that a stable nodal curve polarized with dualizing sheaf is
$K$-stable.

The primary goal of this work is to understand the GIT compactification of moduli of
canonically polarized varieties. The recent work on the relation between various notions of $K$-stabilities and
the existence of constant scalar metrics suggests that some deep and interesting geometry
are yet to be uncovered in this area. This work is a first step toward this direction.
We hope this study will help us understand the stability
of high dimensional singular varieties.

\vsp
We now outline the results proved in this paper.

\begin{defi}[Hassett \cite{Hass}]\label{weight-pt-curve}
A {\em weighted pointed nodal curve} $(X, \bx, \ba)$
consists of a reduced, connected curve $X$, $n$ ordered (not necessarily distinct) smooth points
$\bx=(x_1,\cdots,x_n)$ of $X$, and weights
$\ba=(a_1,\cdots,a_n)$, $a_i\in \QQ_{\ge 0}$, of $\bx$, such that the total weight at any point is no more than one
(i.e. for any $p\in X$, $\sum_{x_i=p} a_i\leq 1$). A polarized weighted pointed curve is a weighted pointed
curve together with a polarization $\sO_X(1)$.
\end{defi}

In this paper, we will use $(X,\sO_X(1), \bx, \ba)$ to denote such a polarized weighted pointed curve.
In case $\sO_X(1)$ is {\em very ample}, we let
\beq\lab{emb-X}
\imath: X\mapright{\sub} \PP W,\quad W=H^0(\sO_X(1))\dual,
\eeq
be the tautological embedding; let %the Chow point
$$
\Chow(X)\in \mathrm{Div}^{d,d}[ ( \PP W\dual ) ^2] := \PP H^0\bl % (\PP W\dual)^2,
 \sO_{(\PP W\dual)^2}(d,d)\br
$$
%where $\sO_{(\PP W)^2}(d,d)$ is the ample line bundle on $(\PP W)^2$ of bi-degree $(d,d)$,
be the Chow point of $X$, which is the bi-degree $(d,d)$ hypersurface in $(\PP W\dual)^2$  consisting
of points $(V_1,V_2)\in (\PP W\dual)^2$ such that $V_i\sub \PP W$ are hyperplanes satisfying $V_1\cap V_2\cap \imath(X)\ne \emptyset$.
We abbreviate
\beq\label{Xi}
\Xi:= \mathrm{Div}^{d,d}[ ( \PP W\dual ) ^2]\times (\PP W)^n,
\eeq
and let the Chow point of $(X,\sO_X(1),\bx)$ be
$$\Chow(X,\bx)=(\Chow(X),\bx)\in \Xi . %:= \mathrm{Div}^{d,d}[ ( \PP W\dual ) ^2]\times (\PP W)^n.
$$
%Here the Chow point is independent of the choice of weights $\ba$; the weights will play a role in
%the stability of this Chow point.

%We begin with the Chow-weight of a polarized nodal curve. Let $(X,\sO_X(1))$ be a polarized
%nodal curve. The Chow point $\Chow(X)$ of $(X,\sO_X(1))$ is the Chow point of the tautological embedding

The stability of this Chow point is tested by the positivity of the $\ba$-weight of any one parameter subgroup
$\lam:\CC^\times \to SL(W)$. (A one parameter subgroup, abbreviated to 1-PS, is always non-trivial.)
Given a 1-PS $\lam$, its action on $W$ induces an action on $\Xi$. Since $\mathrm{Div}^{d,d}[ ( \PP W\dual ) ^2]$ is a projective space, it
has a canonical polarization $\sO(1)$. We let
$$\sO_\Xi(1,\ba)
$$
be the $\QQ$-ample line bundle on $\Xi$ that has degree $1$ on
$\mathrm{Div}^{d,d}[ ( \PP W\dual ) ^2]$ and has degree $a_i$ on the
$i$-th copy of the $\PP W$ in $(\PP W)^n$.  Integral multiple of this line bundle is canonically linearized by
$SL(W)$.

\begin{defi}With $(X,\sO_X(1), \bx,\ba)$ understood, we define the $\ba$-$\lam$-weight of $\Chow(X,\bx)\in \Xi$ be
the weight of the $\lam$-action on the fiber $ \sO_{\Xi}(1,\ba)
\vert_{\zeta }$, where $ \zeta=\lim_{t\rightarrow 0}
\lam(t)\cdot\Chow(X,\bx)\in \Xi$; we denote this weight to be $\omega_\ba(\lam)$.

We define $\omega(\lam)$ be the $\lam$-weight of $\Chow(X)\in \mathrm{Div}^{d,d}[ ( \PP W\dual ) ^2]$
defined with $\Chow(X,\bx)$ (resp. $\sO_{\Xi}(1,\ba)$) replaced by $\Chow(X)$ (resp. $\sO(1)$).
\end{defi}

\begin{defi}\label{chow-ss}
Given $(X,\sO_X(1),\underx,\undera)$, we say that it is {\em stable} (resp. {\em semistable}) if for any 1-PS $\lam$ of
$SL (W)$, the $\ba$-$\lam$-weight of $\Chow(X,\underx)$ is positive (resp. non-negative).
\end{defi}
\black
To make an analogy with the slope stability of vector bundle, we introduce the notion of slope stable
by testing on proper closed subcurves $Y\sub X$. First, with $\sO_X(1)$ understood, we
denote $\deg X=\deg\sO_X(1)$, and for any subcurve $Y\sub X$, we denote $\deg Y=\deg\sO_X(1)|_Y$.
For any proper subcurve $Y\sub X$, we define the number of {\em linking nodes} of $Y$ to be
\beq\label{ell-Y}
\ell_Y=\bigl\vert Y\cap Y\comp\bigr\vert,\quad Y\comp=\overline{X\setminus Y}\,.
\eeq

\begin{defi} Given $(X,\sO_X(1),\bx,\ba)$,
We say that it is {\em slope (semi-)stable} if $X$ is nodal and if for any proper subcurve $Y\subsetneq X$ we have
\beq\label{basic1}
\frac{\deg Y+\frac{\ell_Y}{2}+\sum_{x_j\in Y}\frac{a_j}{2}}{h^0(\sO_X(1)|_Y)}< \frac{\deg X+\sum_{j=1}^n \frac{a_j}{2}}{h^0(\sO_X(1))}\, ,
\quad (resp. \leq)\ .
\eeq
\end{defi}

In this paper, we will prove by verifying the Hilbert-Mumford criterion the following theorem. For the weight $\ba$ and
$g(X)=g$, we denote
\beq\label{aa}
\chi_{\ba}(X):=g-1+(a_1+\cdots+a_n).
\eeq

\begin{theo}\label{main}
Given $g$ and $\ba$ such that $\chi_\ba(X)>0$, there is an $N$ so that a genus $g$
polarized weighed pointed curve $( X,\sO_X(1),\underx,\undera)$
such that
$\deg X\geq N$  is (semi-)stable if
and only if it is slope (semi-)stable.
\end{theo}

For $(X,\bx)$, we abbreviate the $\QQ$-line bundle $\omega_X(\sum a_i x_i)$ to
$\omega_X(\bax)$. For integer $k$ so that $k\cdot a_i\in\ZZ$
for all $i$, then $\omega_X(\bax)^{\otimes k}=\omega_X^{\otimes k}(\sum ka_i x_i)$ is a line bundle.
In Section 5, we will show that in case $\deg X$ is sufficiently large, the slope stability is equivalent to the criterion:
%stated in Theorem \ref{main}. This way, Theorem \ref{main} can be restated as

\begin{prop}
\lab{main-c}
Given $g$ and $\ba$ such that $\chi_{\ba}(X)>0$, there is an $N$ so that a genus $g$ polarized weighed pointed
curve $( X,\sO_X(1),\underx,\undera)$ such that
$\deg X\geq N$  is slope (semi-)stable if and only if for any
proper subcurve $Y\subsetneq X$ satisfying $h^0(\sO_X(1)|_Y)<h^0(\sO_X(1))$, we have
\beq\label{basic}
 \Bigl\vert \Bl\deg Y+\sum_{x_{j}\in Y}\frac{a_{j}}{2}\Br-\frac{\deg_Y\omega_X(
\bax)}{\deg\omega_X(\bax)}\Bl \deg X+\sum_{j=1}^n\frac{a_{j}}{2}\Br
\Bigr\vert
<\frac{\ell_Y}{2}\, ,\ (\text{resp.} \leq)\, .
\eeq
\end{prop}

The case $\bx=\emptyset$ is a theorem of Caporaso \cite{Cap} on the
stability of polarized nodal curves.  The case of  the
asymptotic Hilbert stability of smooth\footnote{Notice if $X$
is smooth then \eqref{basic} becomes vacuous. } weighted pointed
curves is a theorem of David Swinarski \cite{Swin} (see also
\cite{Mor}).

We now sketch the main ingredients of our proof.
Our starting point is a theorem of Mumford that expresses the
$\ba$-$\lam$-weight of $\Chow(X,\bx)$ in terms of the leading coefficient of the Hilbert-Samuel polynomial of
an ideal $\sI\sub\sO_{X\times\Ao}(1)$ (cf. Prop. \ref{chow-wt}). Our first observation is that this leading coefficient
can be evaluated by the leading coefficient of the Hilber-Samuel polynomial of the pull back $\ti\sI$ of $\sI$ to the normalization
$\ti X$ of $X$.  This transforms the evaluation of the $\ba$-$\lam$-weight to the calculation of the areas of
a class of Newton polygons
associated to the pull back sheaf $\ti\sI$.
By dividing the Newton polygons into two kinds and studying them seperately, we obtain
an effective  bound of the areas, thus a bound of the $\ba$-$\lam$-weight of $\Chow(X,\bx)$.
This bound is linear in the weights of $\lam$.  We then apply linear
programing to complete our proof of Theorem \ref{main}.

\vsp

Our GIT construction of the moduli of  weighted pointed stable curves goes as follows. We form
the Hilbert scheme $\cH$ of pointed $1$-dimensional subscheme of $\PP^m$ of fixed degree. Let
$\psi: \cH\to \cC$ be the Hilbert-Chow morphism (map) to the Chow variety of pointed 1-dimensional
cycles in $\PP^m$ of the same degree, equivariant under $SL(m+1)$. Applying our main theorem, we conclude that
in case the degree is sufficiently large, the preimage under $\psi$ of the set $\cC^{ss}\sub \cC$ of
GIT-semistable points is the set of semistable polarized weighted pointed nodal curves. Let $\cK\sub \cH$ be the
subset of canoncially polarized weighted pointed smooth curves. We prove that
the GIT-quotient of the closure $\overline\cK$ is isomorphic to the Hassett's moduli of
weighted pointed stable curves. An interesting observation is that the complement $\overline\cK-\cK$
contains polarized semistable but not canonically polarized weighted pointed curves. Thus though
GIT gives the same compactification as that of Hassett of the moduli of canonically polarized weighted pointed
smooth curves, the geometric objects added to obtain the compactification in the mentioned two
constructions are markedly different. It is worth pursuing to see how this extends in the high dimensional case.

\vsp
In the end, using that the Donaldson-Futaki invariants can be expressed as the limit of certain Chow weights
under a 1-PS, we apply our main theorem to prove that a polarized nodal curve $(X,\sO_X(1))$ is $K$-stable if and only if
$\sO_X(1)$ is numerically equivalent to a multiple of $\omega_X$.
This implies that GIT compactification is same as the compactification of smooth curve using $K$-stability. This
is analogous to that the Uhlenbeck compactification coincides with the GIT compactification of the moduli of vector bundles over curves.
%This is closely related to recent work of Odaka \cite{Od}.

\vsp
The paper is organized as follows. In Section two, we show that the weights can be evaluated via
the leading coefficients of the Hilbert-Samuel polynomial of a sheaf on the normalization $\ti X$.
In Section three, we reduce our study to a particular class of 1-PS: the staircase 1-PS.
We will derive a sharp bound for each irreducible component in Section four. We complete the proof
of our main theorems in Section five. The last two sections include the application of our stability study
to constructing moduli of weighted pointed curves and to study the $K$-stability of polarized curves.

\vsp
\noindent
{\bf Acknowledgement}. The first author is partially supported by an NSF grant NSF0601002. The second author is partially supported
by RGC grants CUHK 403106 and 403709 from the Hong Kong government.  Part of the writing was completed when the second author was visiting the
Mathematics Department of Stanford University in summer 2009 and the Mathematics Department of Harvard University in fall 2009.
He thanks the faculty members there for the hospitality and the fine research environment they provided.

\vsp
\noindent
{\bf List of notations}

\vsp
%For a polarized curve $(X,\sO_X(1))$, we define $\deg X=\deg \sO_X(1)$. This applies to
%when $X\sub \PP^m$ is a curve in the projective space, where its polarization is the pull back of $\sO_{\PP^m}(1)$.
%For subcurve $Y\sub X$, we define $\deg Y=\deg \sO_X(1)|_Y$.  We let  $W=H^0(\sO_X(1))^\vee$.

\begin{tabular}{lll}
$\sI(\lam)$; $\ti\sI(\lam)$ &  $(t^{\rho _0}s_0,\cdots)\sub \sO_{X\times\Ao}(1)$; similarly defined on $\ti X$ & \eqref{I}\\

$e(\sI(\lam))$; $e(\ti \sI)$ & $ \nlc\, \chi ( \sO_{X\times \Ao}(k)/\sI(\lam)^k)$; similarly defined on $\ti X$ & \eqref{e-def}; after \eqref{ti-I} \\

$e(\ti \sI)_q$; $e(\ti \sI\lalp)$ & contribution of $e(\ti \sI)$ at $q\in X$; along $\ti X\lalp$ &\eqref{e-q}\\

$\omega(\lam)$ & $\lam$-Chow weight & Prop.\,\ref{chow-wt}\\

$v(\ti s_i,q)$&  the vanishing order of $\ti s_i$ at $q$ &\eqref{v}\\

$\hbar(q)$& $\max\{i\mid v(\ti s_i,q)\ne \infty\}$ &\eqref{hbar}\\

$\hbar\lalp$& $\min_i\{i\mid \ti s_j|_{\ti X\lalp}=0,\ \text{for}\ j \ge i+1\}.$&\eqref{hbar-alp}\\

$\Delta_q$ & Newton polygon supported at  $q\in \ti X$ & Def.\,\ref{Delta-p}\\

% & contribution of $e(\ti \sI)$ from  $X\lalp\sub X$  & Proof of Lemma \ref{sum-q}\\

$\sE_i=\sE(\lam)_i$ &$ ( s_i, s_{i+1},\ldots, s_m)\sub \sO_{X}(1)$ & \eqref{E}\\

%$\ti\sE_i=\ti\sE(\lam)_i$ &$(\ti s_i,\ti s_{i+1},\ldots,\ti s_m)\sub \sO_{\ti X}(1).$ & after \eqref{E}\\

$\Lam\lalp(\lam)$; $\Lam(\lam)$ &$\{ q\in X\lalp\mid s_{\hbar\lalp}(q)=0\}$; $\Lam(\lam)=\cup_{\alpha=1}^r \Lam\lalp(\lam)$
&Def.\,\ref{index}\\
%
% &$\cup_{\alpha=1}^r \Lam\lalp(\lam)$ &Def.\,\ref{index}\\
%
%$\ti\Lam\lalp(\lam)$ & $\{p\in\ti X\lalp\mid \ti s_{\hbar\lalp}(p)=0\}$ &Def.\,\ref{index}\\
%
%$\ti\Lam=\ti\Lam(\lam)$ & $\cup_{\alpha=1}^m \ti\Lam\lalp(\lam)$ &Def.\,\ref{index}\\

$\delta(\ti s_i,p)$ &$\len(\ti\sE_{i}/\ti\sE_{i+1})_p \text{ or}=0 $ & Def.\,\ref{del-p}\\

$\inc\lalp(\ti s_i)$&$\sum_{p\in\ti X\lalp} \delta(\ti s_i,p) p\quad\text{and}\quad
\inc(\ti s_i)=\sum_\alpha \inc\lalp(\ti s_i);$ & Def.\,\ref{E}\\

$\delta\lalp(\ti s_i)$; $\delta(\ti s_i)$ &$\sum_{p\in\ti X\lalp} \delta(\ti s_i,p)$; $\delta(\ti s_i)=
\sum\lalp\delta\lalp(\ti s_i)$ & Def.\,\ref{E}\\

%$\delta(\ti s_i)$&$\sum\lalp\delta\lalp(\ti s_i)$ & Def.\,\ref{E}\\

$w(\ti\sE_i,p)$; $ w\lalp(\ti\sE_i)$ &$\len(\sO_\tiX(1)/\ti \sE_i)_p $; $w\lalp(\ti\sE_i)=\sum_{p\in \ti X\lalp}w(\ti\sE_i,p)$
& Def.\,\ref{w}\\

% & &\eqref{w}\\

$\II\lalp=\II\lalp(\lam)$ &$\{i\in\II\mid \inc(\ti s_i)\cap \ti X\lalp\ne\emptyset\ \text{or}\ i=\hbar\lalp\};$
&\eqref{I-lalp}\\

$L_Y$; $L\lalp$; $\ti L_Y$; $\ti L\lalp$ &$Y\cap Y\comp$; $L_{X\lalp}$; $\pi\upmo(L_Y)\cap\ti Y$;
$\ti L_{X\lalp}$ &\eqref{link-node} and \eqref{Z-p}\\

$\ti N_Y$; $N\lalp$; $\ti N\lalp$ & $\pi\upmo(N_Y)\cap \ti Y$; $N_{X\lalp}$; $\ti N_{X\lalp}$ &\eqref{N-Y}\\

%$N\lalp\text{ and }\ti N\lalp,$& $N_{X\lalp}\text{ and } \ti N_{X\lalp}$  &\eqref{Z-p}\\
%
%$L\lalp \text{ and } \ti L\lalp $&  $L_{X\lalp}\text{ and } \ti L_{X\lalp}$&\eqref{Z-p}\\

$ \ell\lalp$; $\ell_{\alpha,\beta}$; $\ell_{\alpha,\alpha}$ & $|L\lalp|$; $ |X\lalp\cap X\lbe|$; $-|L\lalp|$ &\eqref{Z-p}; \eqref {ell_ab}\\

$\II_{\alpha}\prim$ &$\{i\in \II\lalp\mid w\lalp(\ti \sE_{i+1})\le \deg X\lalp-2g(X\lalp)-\ell\lalp -1 \} $& Def.
\ref{prim-ind}\\

$E\lalp(\rho)$ &upper bound of $e(\sI)\lalp$ & \eqref{E-alp}\\

$ W_i=W_i(\lam)$ &$\{v\in W\mid s_i(v)=\cdots=s_m(v)=0\}\subset W$& \eqref{W-i}\\

$\omega_\ba(\lam)$ &$ \omega(\lam)+\mu_{\mathbf{a}}(\lam)\ . $&\eqref{omega-mu}\\

%$\cH$ & $\{(X,\iota, \bx)\mid [\iota: X\to \PP^m]\in \Hilb_{\PP^m}^P, \ \bx\in X^n\}.
%$ & after \eqref{spe}\\
%
%$\cC$ & $\{(Z,\bx) \in \Chow_{\PP^m}^d\times (\PP^m)^n\mid  \bx \in \text{supp}(Z)^n\}.$ & after \eqref{univ}\\

$\Phi: \cH\to \cC$ & Hilbert-Chow map & before Lem.\,\ref{weight-nodal}\\

$\cK, \bar\cK\sub \cH$ & slice polarized by $\omega_{\mathcal{X}/\cH}(\ba\cdot \bx)$ & before \eqref{bq}\\

$\vec\delta(\sL)$ & degree class for the line bundle $\sL$ & after \eqref{eq6.6}\\

%$\mathrm{DF}(\mathcal{X},\sL)$ & Donaldson-Futaki invariant for $(\mathcal{X},\sL)$ & Definition \ref{DF}\\
%
%$\varrho_{l,k}$ & weight for $l$-th re-embedding & \eqref{varrho} \\

%$X\lra Y$ & a function to a space with morphism & Theorem \ref{aa}
\end{tabular}

%For the weights $\ba=(a_1,\cdots,a_n)\in \QQ_{\ge 0}^n$ of $n$-ordered points $(x_1,\cdots,x_n)$ (possibly duplicating) of $X$,
%we define $|\ba|=a_1+\cdots+a_n$.
%For any subset $A\sub X$, we define $|\ba|_A=\sum_{x_i\in A} a_i$.
%Thus the requirement on weighted pointed nodal curve can be rephrased as for any closed point $p$, we have
%$|\ba|_p\le 1$.

\section{Chow stability, Chow weight and Newton polygon} \lab{newton-poly}

In this section, we  first recall some basic facts from \cite{Mum} on
stability of a polarized curve; we then
 localize the calculation of the weight of $\Chow(X)$ to a divisor on the normalization of $X$, and interpret the contribution from each point of the divisor as the area of a generalized Newton polytope.

%We will postpone the proofs of the results stated to the later part of this section.

Throughout the paper, we fix a polarized (connected) curve
$(X,\sO_X(1))$, its associated embedding $\imath: X\to \PP W$ (cf. \eqref{emb-X}),
and denote by $\Chow(X)$ the Chow point of $\imath$
once and for all. We also assume that $X$ is nodal unless otherwise is mentioned.

We will reserve the symbol $\lam$ for a 1-PS of $SL(W)$;
for such $\lam$, we diagonalize its action by choosing
\beq\lab{basis-s}
\bolds=\left\{ s_0,\cdots,s_m\right\}\quad \text{a basis of}\quad W\dual % H^{0}(\sO_{X}(1) )
\eeq
so that under its dual bases the action $\lambda$ is given by
\beq
\lambda(t):=\mathrm{diag}[t^{\rho _0},\cdots,t^{\rho_m}]\cdot t^{-\rho_{\text{ave}}}, %\subset SL(W),
\quad
\rho_0\geq\rho_1\ge\cdots\ge \rho_m=0,
\eeq
and $\rho_{\text{ave}}=\frac{1}{m+1}\sum\rho_i$.
We will call $\bolds$ a diagonalizing basis of $\lam$.

In \cite{Mum},  Mumford introduced a subsheaf
\begin{equation}\label{I}
\sI(\lam)=(t^{\rho _0}s_0,\cdots,t^{\rho_m}s_m) \sub \sO_{X\times \Ao}(1):=p_X\sta \sO_X(1)
\end{equation}%
generated by sections in the paranthesis, where
$p_X:\xAo\to X$ is the projection.
Let $e( \sI(\lam)) $ be the normalized leading coefficient
(abbreviate to n.l.c.) of the Hilbert-Samuel polynomial:
\begin{equation}\label{e-def}
e( \sI(\lam))= \nlc\, \chi ( \sO_{X\times \Ao}(k)/\sI(\lam)^k)
 \end{equation}

\begin{prop}[Mumford]
\lab{chow-wt}The $\lam$-weight  $\Chow(X)$ is
\begin{equation*}
\omega(\lam)=\frac{2\deg X}{m+1}\sum_{i=0}^m\rho _{i}-e( \sI(\lam)).
\end{equation*}
\end{prop}

In the following, when the 1-PS $\lam$ and its diagonalizing basis $\bolds$
are understood, we will drop $\lam$ from $\sI(\lam)$ and abbreviate $\sI(\lam)$ to $\sI$.
Our first step is to lift the calculation of $e(\sI)$ ($=e(\sI(\lam))$) to the normalization of $X$:
$$
\pi: \ti X\lra X.
$$We let
\beq\label{si}\ti s_i=\pi\sta s_i\in \sO_\tiX(1):=\sO_X(1)\otimes_{\sO_X}\sO_\tiX,
\eeq
and let $\ti\sI$ be the pull-back of $\sI$:
\beq\lab{ti-I}
\ti \sI=(t^{\rho_0}\ti s_0,\cdots, t^{\rho_m}\ti s_m)\sub \sO_\txAo(1)=\sO_{\ti X}(1)\otimes_{\sO_{\ti X}}\sO_\txAo\ .
\eeq

Like $e(\sI)$, we define
$e(\ti \sI)=\nlc \, \chi(\sO_\txAo(k)/\ti \sI^k)$. We have the following
proposition whose proof will be given at the end of this section.
\begin{prop}\lab{lift}We have
 $e(\sI)=e(\ti \sI)$.
\end{prop}

This Proposition enables us to  lift the evaluation of $e(\sI)$ to $\ti X$.  Our next step is to localize the evaluation of $e(\ti\sI)$ to individual $q\in \ti X$.
In order to do that,  let $z$ be a uniformizing parameter of $\ti X$ at $q$; let  $t$ be the standard coordinates of $\Ao$.
We denote by $\hat \sO_{\tiX,q}$ the formal completion of the local ring $\sO_{\tiX,q}$ at its
maximal ideal. We fix an isomorphism of $\hat\sO_{\tiX, q}$-modules (the first isomorphism below):\beq\lab{phi}
\varphi_q: \sO_{\ti X}(1)\otimes_{\sO_{\ti X}}\hat \sO_{\ti X, q} \,\cong\, \hat\sO_{\tiX, q}\, \cong\, \CC[\![z]\!],
\eeq
where the second isomorphism is
induced by the choice of $z$.

\begin{defi}\lab{def-q}
Let $\ti s_i\in H^0(\sO_{\ti X}(1))$ be as in \eqref{si}. We define
\beq\lab{v}
v(\ti s_i,q)=\text{the vanishing order of $\ti s_i$ at $q$};
\eeq
in case $\ti s_i\equiv0$ near $q$,  we define $v(\ti s_i,q)=\infty$.
We set
\beq\label{hbar}
\hbar(q)=\max\{i\mid v(\ti s_i,q)\ne \infty\}\and
w(\ti\sI,q)= v(\ti s_{\hbar(q)},q).
\eeq
\end{defi}

The quantity $w(\ti\sI,q)$ is the width of the polygon $\Delta_q$ associated to $\ti\sI$ (at $q$)  to be defined later.
\vsp

We now look at  the image of $\ti\sI$ under $\sO_{\ti X\times\Ao}(1)\to \hat \sO_{\ti X\times\Ao, (q,0)}$.
We let
\beq\lab{I_q}
I_q=( z^{v(\ti s_m, q) },z^{v(\ti s_{m-1}, q) }t^{\rho
_{m-1}},\cdots ,z^{v(\ti s_0,q) }t^{\rho _{0}})\sub R=\CC[\![z,t]\!].
\eeq
By construction,  $\varphi_q$ induces an isomorphism
\beq\lab{truncate}
\bl\sO_{\txAo}(k)/\ti\sI^{k}\br\otimes_{\sO_\txAo}\hat{\sO}_{\txAo,(q,0)}\cong R/I_q^{k}.
\eeq

Notice that the right hand side is not a finite module when $\hbar(q)<m$.
Since $t^{\rho_i}\cdot\varphi_q(\ti s_i)\in t^{\rho_{\hbar(q)}} R$ for all $i$, the map
$$ t^{k\cdot\rho_{\hbar(q)}} R/\bl I_q\cap t^{\rho_{\hbar(q)}}R\br^k\lra R/I_q^k
$$
induced by the inclusion $ t^{k\cdot\rho_{\hbar(q)}}R\sub R$ is injective.
This time the $R$-module on the left hand side is a finite module.
We define
\beq\label{e-q}
e(\ti\sI)_q=\nlc\dim  t^{k\cdot\rho_{\hbar(q)}}R/\bl I_q\cap t^{\rho_{\hbar(q)}}R\br^k+2\rho_{\hbar(q)}\cdot\ordq.
\eeq

\begin{lemm}\lab{sum-q}
We have the summation formula\,\footnote{\ We were informed that similar formula was obtained by Swinarski in 2008.}
$e(\ti \sI)=\sum_{q\in\tiX} e(\ti \sI)_q$.
\end{lemm}

We need some preparation to prove this Lemma. We begin with a
geometric interpretation of the quantity $e(\ti\sI)_q$.
Let $I\subset \CC[ z_1,z_2] $
be a monomial ideal and let $\Gamma $ be the set of exponents of
monomials in $I$; namely, $I$ is the linear span of the monomials $\{x^\gamma\mid\gamma \in \Gamma\}$,
where $\Gamma$ is a subset of
$(\NN\cup\{0\})^{2}\subset \mathbb{R}_{\ge 0}^{2}$. ($\RR_{\ge 0}^2$ is the first quadrant of $\RR^2$---the
$xy$-plane.)

We then form the {\sl closed convex hull} $\mathrm{Conv}( \mathbb{R}_{\ge 0}^2+\Gamma )$
 of $\mathbb{R}_{\ge 0}^2+\Gamma $. We let
$\bar{\Gamma}=\mathrm{Conv}( \mathbb{R}_{\ge 0}^2+\Gamma )\cap \NN^2$;
 the integral closure $%
\bar{I}$ of $I$ is the ideal generated by $\{x^\gamma\mid \gamma\in \bar{\Gamma}\}$ \cite[Ex4.23]{Eis}.

We let $\Delta(I)$ be the Newton polygon of $I$:
$$\Delta ( I) =\mathbb{R}_{\ge 0}^2-\mathrm{%
Conv}( \mathbb{R}_{\ge 0}^2+\Gamma )\sub \RR^2_{\geq 0} .
$$

\begin{lemm}
\lab{newton} Let $\left\vert \Delta ( I) \right\vert $ be the area of the $\Delta(I)$.
Then%
\begin{equation*}
\dim \CC[z_1,z_2] /I^{k}=\left\vert \Delta
( I) \right\vert \cdot k^2+O( k^{}).
\end{equation*}%
\end{lemm}

\begin{proof}
Since $\bar{I}$ is the integral closure of $I,$ by Briancon-Skoda theorem
\cite[Thm 9.6.26]{Lar},
$I^{k}\subset \bar{I}^{k}\subset I^{k-1}$ for $k$ sufficiently large. Since $\dim I^{k-1}/I^k$ is bounded from above by
a linear function
in $k$, $\dim \CC[z_1,z_2] /I^{k}=\dim \CC[z_1,z_2] /\bar{I}^{k}+O( k)$.

Further,  $\dim \CC[z_1,z_2] /%
\bar{I}^{k}$ is precisely the number of lattice points in
$k\Delta (
\bar{I}) =k\Delta ( I)$. From the work of Kantor and
Khovanski \cite{KK,Do}, the number of lattice points
inside the polygon is given by $\left\vert \Delta ( I)
\right\vert\cdot k^2+O(1) .$ This proves the Lemma.
\end{proof}

We now come back to the 1-PS $\lam$ and its diagonalizing basis $\bolds=\{ s_i\}$.

\begin{defi}\lab{Delta-p}
For any $q\in\ti X$, we define
$\Gamma_q=\left\{( v(\ti s_i, q) ,\rho_{i}) \right\} _{0\leq i\leq m}\subset (\NN\cup\{0\})^{2}$
and define the {Newton polygon} (of $\ti\sI=\ti \sI(\lam)$) at $q$ to be
$$
\Delta _{q}(\lam):=( \mathbb{R}_{\ge 0}^{2}-\mathrm{Conv}( \mathbb{R}_{\ge 0}^{2}+\Gamma_q ) )
\cap ( [ 0,\ordq ] \times \mathbb{R}_{\ge 0}).
$$
\end{defi}

We will abbreviate $\Delta_q(\lam)$ to $\Delta_q$ when the
choice of the basis $\bolds$ is understood.
Let $|\Delta_q|$ be the area of $\Delta_q$.

\begin{coro}\lab{sum-Delta}
We have $e(\ti\sI)_q=2|\Delta_q|$; henceforth,
$e(\ti\sI)=2\sum_{q\in\tiX} |\Delta_q|$. %In particular, if $\ordq=0$ then $e(\ti\sI)_q=2|\Delta_q|=0$.
\end{coro}

\begin{proof}
Since $\Delta_q$ is the union of $\Delta_q\cap [0,\ordq]\times[\rho_{\hbar(q)},\infty)$ with
$[0,\ordq]\times [0,\rho_{\hbar(q)}]$,
by \eqref{e-def}, \eqref{e-q} and Lemma \ref{newton},
$$e(\ti \sI)_q=2\cdot |\Delta_q\cap [0,\ordq]\times[\rho_{\hbar(q)},\infty)|+2\cdot \rho_{\hbar(q)}\cdot \ordq=2\, |\Delta_q|.
$$
The second identity follows from Lemma \ref{sum-q}.
\end{proof}

%\begin{exam}\lab{cubic}
%Let $X=\displaystyle\bigcup_{i=1}^3 L_i\in \PP^3$, where $L_i\sub\PP^3$ are lines given by
%$L_1=(s_1=s_2=0)$, $L_2=(s_2=s_3=0)$ and $L_2=(s_0=s_3=0)$. $e(\sI)=e(\sI)_p+e(\sI)_q$ with $p=L_1\cap L_2$ and
%$q=L_2\cap L_3$. Then we have the following figure???
%\end{exam}

This formula will be used to estimate the quantity $e(\sI)=e(\ti \sI)$ in the next section. For now,
 we prove Proposition \ref{lift}.

\begin{proof}[Proof of Proposition \ref{lift}]
Let $p_1,\cdots, p_l$ be the nodes of $X$; let $\xi=\pi\times 1_\Ao:\ti X\times\Ao\to X\times\Ao$ be the
projection. Tensoring the exact sequence
$$0\lra \sO_{X\times\Ao}\lra \xi\lsta\sO_{\txAo}\lra \oplus_{j=1}^l \sO_{p_j\times\Ao}\lra 0
$$
with $\sO_\xAo(k)/\sI^k$, we obtain an exact sequence
$$\sO_\xAo(k)/\sI^k\mapright{f_k} (\sO_\xAo(k)/\sI^k)\otimes_{\sO_{X\times\Ao}}\xi\lsta \sO_{\txAo}
\lra \bigoplus_{\alpha=1}^r (\sO_\xAo(k)/\sI^k)|_{p_j\times\Ao}\lra 0.
$$
By projection formula, we have
\begin{eqnarray*}
%\lab{push-for}
\xi\lsta\bl  \sO_\txAo(k)/\ti \sI^k\br
=\xi\lsta\bl  \xi\sta(\sO_\xAo(k)/ \sI^k)\br=
(\sO_\xAo(k)/\sI^k)\otimes_{\sO_{X\times \Ao}}\xi\lsta \sO_{\txAo}.
\end{eqnarray*}
Thus
$$
%e(\sI)=
e(\ti\sI)
=\nlc \chi\bl\xi \lsta\bl  \sO_\txAo(k)/\ti \sI^k\br\br
=\nlc \chi\bl(\sO_\xAo(k)/\sI^k)\otimes_{\sO_{X\times\Ao}}\xi\lsta \sO_{\txAo}\br,
$$
which equals
$$
\nlc\bl \chi\bl\sO_\xAo(k)/\sI^k\br-\dim \ker f_k+ \sum_{i=1}^l \chi\bl(\sO_\xAo(k)/\sI^k)|_{p_j\times\Ao})\br.
$$
%Since the support of both $\ker(f_k)$ and $\sO_\xAo(k)/\sI^k\otimes_{\sO_X}\sO_{p_i}$ are of one dimensional, this implies

We claim that both
\beq\lab{est-1}
\chi\bl(\sO_\xAo(k)/\sI^k)\otimes_{\sO_{X\times\Ao}}\sO_{p_j\times\Ao}\br \and  \dim \ker f_k
\eeq
are linear in $k$. % and $f_k$ is injective.
This will prove the Proposition.

We begin with the first claim. We let $q$ be one of the node of $X$;
let $q^+$ and $q^-$ be the preimages $\pi\upmo(q)\sub\ti X$,
and let $x$ and $y$ be uniformizing parameters of $\ti X$ at $q^+$ and $q^-$, respectively. Then
after fixing an isomorphism $\sO_X(1)\otimes_{\sO_X}\sO_{X,q}\cong\sO_{X,q}$ near $q$
and denoting $R=\kk[\![x,y]\!]/(xy)$, we have isomorphism
\beq\label{R}(\sO_\xAo(k)/\sI^k)\otimes_{\sO_{X\times\Ao}}\sO_{q{}\times\Ao}
\cong (R[t]/I^k)\otimes_{R[t]} R[t]/(x,y),
\eeq
where $I\sub R[t]$ is the ideal generated by
$t^{\rho_i} \hat s_{i}$, $i=0,\cdots,m$, and $\hat s_i$ are formal germs of $s_i$ at $q$ as elements in $R$.
Since for some $i$ the value $s_i(q)\ne 0$,  $i_q=\max\{i \mid s_i(q)\ne 0\}$
is finite. Thus the right hand side of \eqref{R} is isomorphic to
$R[t]/(I^k,x,y)=\kk[t]/(t^{k\cdot i_q})$ whose dimension is linear in $k$. This proves the first claim.

For the second claim, since the kernel of $f_k$ consists of torsion elements supported on the union of
$p_1\times\Ao,\cdots, p_l\times\Ao$. Hence to prove the claim, we only need to study the kernel of
an analogue homomorphism
$$\bar f_k: R[t]/I^k\lra (R[t]/I^k)\otimes_{R[t]} (\kk[\![x]\!][t]\oplus \kk[\![y]\!][t]),
$$
where $I$ is as in the previous paragraph, and
$R[t]\to \kk[\![x]\!][t]\oplus \kk[\![y]\!][t]$ is the normalization homomorphism $g(x,y,t)\mapsto (g(x,0,t),g(0,y,t))$.
Since the domain and the target of $\bar f_k$ are $t$-graded rings and $\bar f_k$ is a homomorphism
of graded rings, as vector spaces
$$\ker \bar f_k=\bigoplus_{j\geq 0} \ker\bigl\{ (\bar f_k)_j: t^j R/(I^k\cap t^jR)\to (t^j R/(I^k\cap t^jR))\otimes_R (\kk[\![x]\!]\oplus \kk[\![y]\!])\,
\bigr\}.
$$
Because $R=\kk[\![x,y]\!]/(xy)$, as $R$-modules, $t^j R/(I^k\cap t^jR)$ is isomorphic to
$R/J$ for $J$ one of the ideals in the list:
$$R,\ (0), \ (x^e),\ (y^e),\ (x^e,y^{e'}),\ (x^e+y^{e'}), \ \text{where}\ e,e'\in \NN.
$$
One checks that for $J$ of the first five kinds, $\ker (\bar f_k)_j=0$; for $J$ of the last kind, $\ker(\bar f_k)_j\cong \kk$.
Thus we always have $\dim \ker (\bar f_k)_j\leq 1$. On the other hand, since $s_{i_q}(q{})\ne 0$,
$t^{\rho_{i_q}}\in I$ and $t^{k\rho_{i_q}}\in I^k$. Thus $\ker (\bar f_k)_j=0$ for $j\geq ki_q$.
This proves that $\dim \ker f_k$ is at most linear in $k$.
This proves the Proposition.
\end{proof}

Because of this Proposition, we will work over the normalization $\ti X$ of $X$ subsequently. To
avoid possible confusion, we will reserve ``$\ti{\quad}$''
to denote the associated objects lifted to $\ti X$. For instance, we will denote by $X_1,\cdots,X_r$ the
irreducible components of $X$, and denote by $\ti X_1,\cdots,\ti X_r$ their respective normalizations.
For the sections $t^{\rho_i}s_i$
in $\sI$, $t^{\rho_i}\ti s_i$ are their lifts in $\ti\sI=\sI\otimes_{\sO_{X\times\Ao}}\sO_{\ti X\times\Ao}$.
For consistence, we reserve subindex $i$ for the sections $s_i$, and reserve
the greek $\alpha$ for the index of the irreducible components $\{X\lalp\}_{1\leq \alpha \le r}$.

%\begin{proof}
%By Proposition \ref{lift}, we only need to prove that for any connected component $\ti X\lalp\sub\ti X$,
%$e(\ti \sI|_{\ti X\lalp})=\sum_{q\in\tiX\lalp} e(\ti \sI)_q$.  [Continue]
%\end{proof}

\begin{proof}[Proof of Lemma \ref{sum-q}]
For each irreducible component $X\lalp\sub X$, we let $\ti\sI\lalp=\ti\sI|_{\ti X\lalp\times\Ao}
\sub\sO_{\ti X\lalp\times\Ao}(1)$. Then
$$e(\ti\sI)=\sum_{\alpha=1}^r \nlc \, \chi \bl \sO_{\ti X\lalp\times \Ao}(k)/\ti\sI\lalp^k\br=\sum_{\alpha=1}^r e(\ti\sI\lalp).
$$
Thus to prove the Lemma, we only need to show that for each $X\lalp$,
$$e(\ti \sI\lalp)=\sum_{q\in\tiX\lalp} e(\ti \sI\lalp)_q,
$$
where $e(\ti\sI\lalp)_q=e(\ti\sI)_q$ when $q\in\ti X\lalp$.
To proceed, we notice that $\hbar(q)$(cf.\eqref{hbar}) is a locally constant function on $\tiX\lalp$; we let
$\hbar\lalp$ be the values of $\hbar(q)$ for $q\in\ti X\lalp$. Then we have
\beq\label{hbar-alp}
\hbar\lalp=\min_i\{i\mid \ti s_j|_{\ti X\lalp}=0,\ \text{for}\ j \ge i+1\}.
\eeq
Thus $t^{\rho_{\hbar\lalp}}$ divides $t^{\rho_i}\ti s_i$ for all $i>\hbar\lalp$.
Since $\rho_{i}\ge \rho_{i+1}$, the same division holds for all
%thus  $t^{\rho_{\hbar\lalp}}\mid t^{\rho_i}\ti s_i$ for all
$i$. We let $\bar\rho_i=\rho_i-\rho_{\hbar\lalp}$, and introduce ideal
$$\ti\sR\lalp=(t^{\bar\rho_0}\ti s_0, t^{\bar\rho_1}\ti s_1,\cdots, t^{\bar\rho_{\hbar\lalp}}\ti s_{\hbar\lalp})
\sub \sO_{\ti X\lalp\times \Ao}(1).
$$
This way,
$\ti\sI\lalp=t^{\rho_{\hbar\lalp}}\ti\sR\lalp\subset t^{\rho_{\hbar\lalp}} \sO_{\tiX\lalp\times \Ao}(1) $.

We let $(t^{k\rho_{\hbar\lalp}})= t^{k\rho_{\hbar\lalp}}\sO_{\tiX\lalp\times \Ao}(k)$; it belongs to the exact sequence
$$0\lra (t^{k\cdot\rho_{\hbar\lalp}})/\ti \sI^k\lalp \lra \sO_{\ti X\lalp\times\Ao}(k)/\ti\sI^k\lalp\lra
\sO_{\tiX\lalp\times \Ao}(k)/(t^{k\cdot\rho_{\hbar\lalp}})\lra 0.
$$
%with $(t^{k\cdot\rho_{\hbar\lalp}})=t^{k\cdot\rho_{\hbar\lalp}}\sO_{\tiX\lalp\times \Ao}(k)$ .
Since
$(t^{k\cdot\rho_{\hbar\lalp}})/\ti \sI^k\lalp=
t^{k\cdot\rho_{\hbar\lalp}}\cdot\bl\sO_{\ti X\lalp\times\Ao}(k)/\ti\sR\lalp^k\br$
and $\sO_{\ti X\lalp\times\Ao}(k)/\ti\sR\lalp^k$ is a finite module, we have
$$\chi( \sO_\txAo(k)/\ti\sI^k\lalp)=\chi(\sO_{\tiX\lalp\times\Ao}(k)/\ti\sR\lalp^k)+
\chi(\sO_{\ti X\lalp\times\Ao}(k)/(t^{k\cdot\rho_{\hbar\lalp}})).
$$
Taking the $\nlc$ of individual term, and using
$$\chi(\sO_{\ti X\lalp\times\Ao}(k)/(t^{k\cdot \rho_{\hbar\lalp}}))=
k\,\rho_{\hbar\lalp}\cdot \chi(\sO_{\tiX\lalp}(k))=k^2\,\rho_{\hbar\lalp}\cdot \deg X\lalp+O(k),
$$
we obtain
\beq
\lab{sum-2}e(\ti\sI\lalp)=
\nlc\, \chi(\sO_\txAo(k)/\ti\sI^k\lalp)=
\nlc\, \chi(\sO_{\tiX\lalp\times\Ao}(k)/\ti\sR\lalp^k)+2\rho_{\hbar\lalp}\cdot\deg \tiX\lalp.
\eeq

Next let $\{q_1,\cdots,q_l\}$ be the support of $(\ti s_{\hbar\lalp}=0)\cap \tiX\lalp$. Following the convention in \eqref{truncate},
we have an isomorphism
$$\sO_{\tiX\lalp\times\Ao}(k)/\ti\sR\lalp^k\mapright{\cong} \oplus_{a=1}^l t^{k\cdot \rho_{\hbar\lalp}}R/(I_{q_a}\cap t^{\rho_{\hbar\lalp}} R)^k,
$$
induced by restricting to germs at $q_a$ after multiplying $t^{k\cdot \rho_{\hbar\lalp}}$.
Adding that
$$\deg\tiX\lalp=\dim \sO_{\tiX\lalp}(1)/(\ti s_{\hbar\lalp})=\sum_{a=1}^l w(\ti\sI,q_a),
$$
\eqref{sum-2} gives us
$$e(\ti\sI\lalp)=\sum_{a=1}^l \Bl \nlc\, h^0( t^{k\cdot \rho_{\hbar\lalp}}R/(I_{q_a}\cap t^{\rho_{\hbar\lalp}} R)^k)+2\rho_{\hbar\lalp}\cdot w(\ti\sI, q_a)\Br
=\sum_{q\in \ti X\lalp} e(\ti \sI)_{q}.
$$
This proves the Lemma.
\end{proof}

Finally, we give one example that will be used later.

\begin{exam}\label{weighted-pt}
Let $\bs=\{s_i\}$ be a basis of $H^0(\sO_X(1))$; using weights $\rho_0=1>\rho_1=\cdots=\rho_m=0$ we form
a 1-PS with diagonalizing basis $\bs=\{s_i\}$:
$$\lam=\mathrm{diag}[t,1,\cdots,1]\cdot t^{-\frac{1}{m+1}}.
$$
Suppose $p=\{s_1=\cdots=s_m=0\}\in X$ is a single point. Then $e(\sI(\lam))=1$ (resp. $=2$) when $p$ is
a smooth point (resp. nodal point) of $X$. Hence
\begin{displaymath}
   \omega(\lam)= \left\{
     \begin{array}{llr}
      \frac{2\deg X}{m+1}-2 , &q \text{ is a nodal point;}\\
      \frac{2\deg X}{m+1}-1,   &q \text{ is a regular point.}
     \end{array}
   \right.
\end{displaymath}
\end{exam}

\section{Staircase One-parameter subgroups}\label{stair}

We begin with some conventions attached to a fixed 1-PS $\lam$ and its diagonalizing basis
$\{s_0,\cdots, s_m\}$. For simplicity, we denote
$$\II=\{0,1,\ldots,m\}.
$$
For each $i\in \II$, we introduce
subsheaves %of $\sO_X(1)$ and $\sO_\tiX(1)$:
\beq\lab{E}
\sE_i=\sE(\lam)_i:=( s_i, s_{i+1},\ldots, s_m)\sub \sO_{X}(1);
\eeq
they form a decreasing sequence of subsheaves.
Similarly, we introduce
$\sO_\tiX$-submodules
$$\ti\sE_i=\ti\sE(\lam)_i:=(\ti s_i,\ti s_{i+1},\ldots,\ti s_m)\sub \sO_{\ti X}(1).
$$

\begin{defi}\lab{index}
We call $i\in\II$ a {\em base index} if $i=\hbar\lalp$ (cf. \eqref{hbar-alp}) for some irreducible component $X_\alpha$.
For each $X\lalp$, we define $\Lam\lalp(\lam)=\{ q\in X\lalp\mid s_{\hbar\lalp}(q)=0\}$;
define $\Lam(\lam)=\cup_{\alpha=1}^r \Lam\lalp(\lam)$; define $\ti\Lam\lalp(\lam)=\{p\in\ti X\lalp\mid \ti s_{\hbar\lalp}(p)=0\}$,
and define $\ti\Lam=\ti\Lam(\lam)=\cup_{\alpha=1}^m \ti\Lam\lalp(\lam)$.
\end{defi}

In the following, for any sheaf of $\sO_{\ti X}$-modules $\sF$ and $p\in\ti X$, we denote $\sF_p:=
\sF\otimes_{\sO_\tiX}\sO_{\tiX,p}$, the localization of $\sF$ at $p$. We remark that for any $p\in\tiX\lalp$,
$\hbar(p)=\hbar\lalp$ is the largest index $i$ so that $(\ti\sE_i)_p\ne 0$.

\begin{defi}\label{del-p}
For closed point $p\in \ti X\lalp \sub \ti X$, we define
$$\delta(\ti s_i,p)=\len(\ti\sE_{i}/\ti\sE_{i+1})_p\ \text{when}\ i+1\le \hbar\lalp(p); \quad %(\ti\sE_{i+1})_p\ne 0;\quad
\delta(\ti s_i,p)=0 \ \text{otherwise}.
%\footnote{\red We define it to be $0$ instead of $\infty$ because  we want $\delta (\ti s_i)<\infty$.}
$$
We define the {\em increments} of $\ti s_i$, along $\ti X\lalp$ and $\ti X$,  be (cycles)
$$\inc\lalp(\ti s_i)=\sum_{p\in\ti X\lalp} \delta(\ti s_i,p) p\quad\text{and}\quad
\inc(\ti s_i)=\sum_\alpha \inc\lalp(\ti s_i);
$$
we define their degrees be
$\delta\lalp(\ti s_i)=\sum_{p\in\ti X\lalp} \delta(\ti s_i,p)$ and $\delta(\ti s_i)=\sum\lalp\delta\lalp(\ti s_i)$.
We also define the width of $\ti\sE_i$ at $p\in \ti X\lalp$ and at $\ti X\lalp$ for $i\le\hbar\lalp$ be
\beq\label{w}
w(\ti\sE_i,p):=\len(\sO_\tiX(1)/\ti \sE_i)_p\quad\text{and}\quad w\lalp(\ti\sE_i)
:=\sum_{p\in \ti X\lalp}w(\ti\sE_i,p)\ .
\eeq
\end{defi}

We remark that for $p\in \ti X\lalp$, $i+1\leq \hbar(p)$ is equivalent to $(\ti\sE_{i+1})_p\ne 0$.

\begin{defi}
For  any irreducible component $X\lalp\sub X$ we introduce
\beq\label{I-lalp}
\II\lalp=\II\lalp(\lam)=\{i\in\II\mid \inc(\ti s_i)\cap \ti X\lalp\ne\emptyset\ \text{or}\ i=\hbar\lalp\};
\eeq
for $m\lalp+1=|\II\lalp|$, the order of $\II\lalp$, we introduce a re-indexing map
\beq\label{ind-a}
\ind\lalp: \II\lalp\lra [0,m\lalp]\cap \ZZ, \quad\text{order preserving and bijective}.
\eeq
Similarly, for $p\in\ti X$, we introduce
$$\II_p=\{i\in\II\mid p\in \inc(\ti s_i)\}.
$$
For $m_p+1=|\II_p|$; we define similarly
$$\ind_p: \II_p\lra [0,m_p]\cap \ZZ, \quad\text{order preserving and bijective}.
$$
\end{defi}
%\begin{rema} Note that $\ind\lalp(\hbar\lalp)=m\lalp$.
%\end{rema}

To define the {\em staircase} 1-PS, we need the following
\begin{defi}
For each $\sE_i$, we define its codegree
\beq\lab{codeg}
\codeg(\sE_i)=\len \bl \sO_Y(1)/\sE_i|_Y\br+\deg\sO_{Y\comp}(1),\quad Y=\supp (\sE_i),
\eeq
where $\supp(\sE_i)$ is the smallest
closed subscheme $Y\sub X$ so that the tautological $\sE_i\to\sE_i|_Y:= \sE_i\otimes_{\sO_X}\sO_Y$
is injective, and $Y\comp=\overline{X\setminus Y}$. Since
$\sE_i$ is decreasing, $\codeg(\sE_i)$ is increasing.
\end{defi}

\begin{defi}\lab{def-stair}
We say a 1-PS $\lam$ is a {\em semi-staircase after index} $i$ if for any $i< j\leq m$, either
$\codeg(\sE_{j-1}) <\codeg(\sE_{j})$,
or $j=\hbar\lalp+1$ (cf. \eqref{hbar-alp})
%\footnote{\red this seems never happen}
 for some irreducible component $X\lalp\sub X$.
We say $\lam$ is a {\em semi-staircase}
when $\lam$ is a semi-staircase after index $1$.
\end{defi}

%We have the following existence result.
\begin{prop}\lab{stair-case}
Given a 1-PS $\lam$, there is a semi-staircase 1-PS $\lam'$ so that $\omega(\lam)\geq \omega(\lam')$.
\end{prop}

\begin{proof}
Suppose $\lam$ is a semi-staircase at index $i$  but not at $i-1$ then
%By definition, $i\ne \hbar\lalp$ for all $\alpha$ and
\beq\lab{ineq-Z}
\codeg(\sE_{i-1})=\codeg(\sE_{i})<\codeg(\sE_{i+1}).
\eeq
We claim that $\sE_{i-1}=\sE_i$. Since $\supp(\sE_i)$ is always a subcurve of $X$
and $i\ne \hbar\lalp +1$ for all
$\alpha$ by the assumption, $Y$ is also the support of $\sE_{i-1}$.  Consequently, $\sE_{i}|_Y\sub \sE_{i-1}|_Y$ and
$\sE_{i-1}|_Y/ \sE_{i}|_Y$ is a finite module. Then $\codeg(\sE_{i-1})=\codeg(\sE_i)$ implies that
$\len( \sE_{i-1}|_Y/ \sE_{i}|_Y)=0$. This proves that $\sE_{i-1}=\sE_i$.

As a consequence, we have $\sE_{i-1}=\sE_i\supsetneq \sE_{i+1}$. Therefore, there is a
point $p\in X$ such that if we denote by $\hat s_j\in \hat\sO_{X,p}(1)$ the formal germ of
$s_j$ at $p$, then as $\hat\sO_{X,p}$-modules
\beq\label{formal}
\hat\sO_{X,p}(1)\supset (\hat s_{i-1},\cdots,\hat s_m)=(\hat s_{i},\cdots,\hat s_m)\supsetneq (\hat s_{i+1},\cdots,\hat s_m).
\eeq
By the middle equality, we can find $\hat c_j\in\hat\sO_{X,p}$ such that $\hat s_{i-1}=\sum_{j=i}^m \hat c_j \hat s_j$.

We now construct a new basis $\bs'$. Let $c=\hat c_i(p)\in\kk$.
We define
\beq\label{new}
s_j'=s_j \quad \text{for}\ j\ne i,i-1;\quad s_i'=s_{i-1}- c s_i;\quad s'_{i-1}=s_i.
\eeq
Clearly, $\bs'=\{s_i'\}$ is a basis of $H^0(\sO_X(1))$. For $j\ne i$,
because the linear span of $\{s_j,\cdots,s_m\}$ equals the linear span of $\{s_j',\cdots,s_m'\}$, we have
$\sE_j=\sE_j'$, where $\sE_j'$ is the $\sE_i$ in \eqref{E} with $s_i$ replaced by $s_i'$.

For $i$, we claim that $\sE_i'\subsetneq\sE_i$. The inclusion $\sE_i'\sub\sE_i$ follows from
$\sE_i'\sub\sE_{i-1}=\sE_i$. For the inequality, we claim that
$$(\hat s_{i-1}-c\hat s_i,\hat s_{i+1},\cdots,\hat s_m)
\ne (\hat s_i,\hat s_{i+1},\cdots,\hat s_m).
$$
Suppose instead the identity holds, then there are constants $a_j\in\kk$ such that
$$\hat s_i= a_i(\hat s_{i-1}-c\hat s_i)+ \sum_{j=i+1}^m a_{j}\hat s_{j}=
\bl a_i(\hat s_{i-1}-\hat c_i \hat s_i)+\sum_{j=i+1}^m a_{j}\hat s_{j} \br+ a_i(\hat c_i-c)\hat s_i.
$$
Combined with $\hat s_{i-1}=\sum_{j=i}^m \hat c_j \hat s_j$, we conclude that
$\hat s_i\in (\hat s_{i+1},\cdots,\hat s_m)+\hat s_i\mm$, where $\mm\sub\hat \sO_{X,p}$ is the maximal ideal.
By Nakayama Lemma, $\hat s_i\in (\hat s_{i+1},\cdots,\hat s_m)$,
contradicting to \eqref{formal}.
This proves the claim.

Finally, we claim that if we define $\lam'$ be the 1-PS with diagonalizing basis $\bs'$ and associated weights $\{\rho\}_{i\in\II}$,
then $\omega(\lam')\leq \omega(\lam)$. By Mumford's formula (cf. Prop.\,\ref{chow-wt}),
this is equivalent to $e(\sI(\lam'))\ge e(\sI(\lam))$. By our construction, $\sE_i'\sub \sE_i$ for all $i\in\II$;
hence since $\rho_{i-1}\geq \rho_i$,
$\sI(\lam')\sub \sI(\lam)$. Thus $\sO_{X\times\Ao}(k)/\sI(\lam')^k$ surjects onto $\sO_{X\times\Ao}(k)/\sI(\lam)^k$.
This proves $e(\sI(\lam'))\ge e(\sI(\lam))$.

So far, for any $\lam$ that is not a semi-staircase, we have constructed a new $\lam'$ so that $\sI(\lam')\sub \sI(\lam)$.
We now claim that  by continuing this process, we eventually arrive at a semi-staircase $\lam'$. Suppose not, then we can
constructed an infinite sequence of 1-PS
$$\lam=\lam_0,\lam_1,\cdots,\lam_l,\cdots
$$
so that $\sE(\lam_{l+1})_i\sub \sE(\lam_{l})_i$ for all $i\in\II$,
 and for some $i$, $\sE(\lam_{l+1})_{i}\ne \sE(\lam_{l})_{i}$.
 (Here $\sE(\lam_l)_i$ is the sheaf $\sE_i$ in \eqref{E} with
$\lam$ replaced by $\lam_l$.) Because $\codeg (\sE(\lam_{l})_i)\leq
\deg\sO_X(1)$, for each $i$, the sequence
$$\sE(\lam_{0})_i\sub \sE(\lam_{1})_i\sub \cdots\sub \sE(\lam_{l})_i\sub\sE(\lam_{l+1})_i\sub\cdots
$$
stabilize at finite places. In particular, after finite place, we will have $\sE(\lam_{l})_i=\sE(\lam_{l+1})_i$ for all $i$;
or equivalently, $\sI(\lam_l)=\sI(\lam_{l+1})$, a contradiction.
This proves that this process eventually provides us a semi-staircase $\lam'$ such that
$\omega(\lam)\geq \omega(\lam')$.
\end{proof}

%We now prove some useful properties of a semi-staircase 1-PS.
\begin{rema}
We remark that for a {\em semi-staircase} $\lam$, the inclusions
$\sO_{\ti X}(1)=\ti\sE_0\supsetneq\ti\sE_1\supsetneq\cdots\supsetneq\ti\sE_m\ne 0$ are proper.
\end{rema}

\begin{defi}\label{staircase}
We say a semi-staircase 1-PS $\lam$ is a {\em staircase} if for any
$p\in\ti\Lam$, $v(\ti s_i,p)\le v(\ti s_{i+1},p)$ for all $i$ (cf. Definition \ref{def-q}).
%\footnote{Notice that in general it is possible that $v(\ti s_i,p)>v(\ti s_{i+1},p)$ for a semi-staircase.  E.g. $\Lam=\{p\}$ and
%near $p\in X$, $s_0=z^2+O(z^3)$ and $s_1=z+O(z^2)$, in this case we replace $s_0,s_1$ by $s'_0=s_0+\epsilon s_1$ and $s'_1=s_1$.}
\end{defi}

\begin{coro}
Proposition \ref{stair-case} holds with semi-staircase replaced by staircase.
\end{coro}

\begin{proof}
By Proposition \ref{chow-wt}, the $\lam$-weight $\omega(\lam)$ (of $\Chow(X)$) depends only
the sheaf $\sI(\lam)$ and the weights $\{\rho_i\}$.
Thus, for any 1-PS $\lam'$ with $\sI(\lam)=\sI(\lam')$ and having identical weights $\{\rho'_i\}$
as that of $\lam$, we have $\omega(\lam)=\omega(\lam')$.

Given any 1-PS, we let $\lam$ be the corresponding semi-staircase constructed in Proposition \ref{stair-case}.
Let $\ti\Lam$ and $\{s_i\}$ be the associated objects of $\lam$. Since $\ti \Lam$ is a finite set, if we
replace $s_i$ by $s_i'=s_i+\sum_{j>i} c_{ij}s_j$  for a general choice of
$c_{ij}\in\CC$, the new 1-PS with the same $\{\rho_i\}$ but new
basis $\{s_i'\}$ will be a desired {\em staircase} 1-PS.
\end{proof}

\begin{lemm}\label{w-delta}
Suppose $\lam$ is a  staircase 1-PS, then for $p\in\ti X\lalp$ and $ i< \hbar\lalp$,
$w(\ti\sE_i,p)=v(\ti s_i,p)$, and
$\delta(\ti s_{i-1},p)=v(\ti s_{i},p)-v(\ti s_{i-1},p)$.
\end{lemm}

\begin{proof}
The proof is a direct consequence of the definition of  staircase 1-PS.
\end{proof}

As we will see, if $\lam$ is a {\em staircase} 1-PS then for most of $i$, $\delta(\ti s_i)=1$. For those $i$
with $\delta(\ti s_i)>1$, we will give a detailed characterization (cf. Prop.\,\ref{vir-ind}). To
this purpose, for any
subcurve $Y\sub X$, we denote by $N_Y$ to be the set of {\em nodes} of $X$ in $Y$; namely,
$N_Y=X_{\text{node}}\cap Y$.  We denote (cf. \eqref{ell-Y})
\beq\label{link-node}
L_Y:=Y\cap Y\comp\ ,
\eeq
and call it the {\em
linking nodes} of $Y$. Moreover, let
\beq\label{N-Y}
\ti N_Y:=\pi\upmo(N_Y)\cap \ti Y\quad\text{ and }\quad\ti L_Y:= \pi\upmo(L_Y)\cap\ti Y
\sub\ti N_Y.
\eeq
Since we reserve $\alpha$ for the index of the components $X\lalp$, we abbreviate
\beq\label{Z-p}
N\lalp:=N_{X\lalp},\quad
\ti N\lalp:=\ti N_{X\lalp}, \quad
L\lalp:=L_{X\lalp}, \quad  \ti
L\lalp:=\ti L_{X\lalp},\quad
\ell\lalp:=|L\lalp|\ .
\eeq
We now state a characterization of those indices with $\delta(\ti s_i)>1$.

\begin{prop}\label{vir-ind}
Suppose $\lam$ is a {\em staircase} 1-PS. Let $i\in\II\lalp$ be a non-base index (cf. Definition \ref{index}) and
let $p\in\inc(\ti s_i)\cap \ti X\lalp$. Suppose
$\delta(\ti s_i)\geq 2$, and suppose further that either $\deg X\lalp=1$ or
\beq\label{prim-ineq}
w\lalp(\ti \sE_i)+1\leq \deg X\lalp-2g(X\lalp)-\ell\lalp,
\eeq
then $q=\pi(p)\in X$ is a node of $X$, $\ind_p(i)=0$ and $\delta(\ti s_i,p)=1$.
When this happens, let $\{p,p'\}=\pi\upmo(q)$ and let $\ti X\lbe$ be the component satisfying $p'\in \inc(\ti s_i)\cap \ti X\lbe$
(possibly $\ti X\lalp=\ti X\lbe$), and suppose further $\deg X\lbe>1$ and
\beq\label{prim-ineq-0}
w\lbe(\ti \sE_i)+1\leq \deg X\lbe-2g(X\lbe)-\ell\lbe\ ,\eeq
then $\inc(\ti s_i)=p+p'$.
\end{prop}

Before  its proof, we introduce a few notations. Since $\ti X\lalp$ is smooth, we can view a zero-subscheme
of $\ti X\lalp$ as a divisor as well.
This way, the union of two effective divisors
is the union as zero subschemes, and the sum is as sum of divisors.
For example, $(\sum n_p p)\cup(\sum n'_p p)=\sum \max\{n_p,n'_p\} p$
and $(\sum n_p p)+(\sum n'_p p)=\sum (n_p+ n'_p) p$.

\begin{proof}[Proof of Proposition \ref{vir-ind}.]
We will prove each part of the statement by repeatedly applying the following strategy.
Suppose $i$ satisfies \eqref{prim-ineq} and $\delta(\ti s_i)\geq 2$,
we will construct a section $\zeta\in H^0(\sO_X(1))$ so that the $\sO_X$-modules $\sF_j=(\zeta, s_j,\cdots, s_m)$
fits into a strict filtration
\beq\label{strict-F}
\sF_0\supsetneq %\sE_1\supsetneq
\cdots\supsetneq\sF_i\supsetneq \sF_{i+1}\supsetneq\sE_{i+1}\supsetneq
\cdots\supsetneq \sE_m\ne 0.
\eeq
Since $\sE_j$ and $\sF_j$ are generated by global sections of $H^0(\sO_X(1))$, this implies
$h^0(\sO_X(1))>m+2$, a contradiction.

Let us assume $\deg X\lalp>1$ first, since for the case $\deg X\lalp=1$ the proof is rather easy. So $w_i(\ti\sE_i)$ satisfies \eqref{prim-ineq}.
We recall an easy consequence of a vanishing result. Let $B \sub \ti X\lalp$ be a closed zero-subscheme
such that
\beq\label{Z-ineq}
\deg B\leq  \deg X\lalp-2g(X\lalp)-\ell\lalp+1.
\eeq
Let $\ti N\lalp$ be as defined in \eqref{Z-p}. We claim that the $\gamma$ in the exact sequence
\beq\label{H1}H^0(\sO_{\ti X\lalp}(1))\mapright{\gamma} H^0(\sO_{\ti N\lalp\cup B}(1))\lra
H^1(\sO_{\ti X\lalp}(1)(-\ti N\lalp\cup B))
\eeq
is surjective.
%Following \eqref{ind-a}, we let $j_0\le j_1\le\cdots\le j_{m\lalp}$ be the set $\II\lalp$.
Indeed, using $\deg \ti N\lalp= 2g(X\lalp)-2g(\ti X\lalp)+\ell\lalp$ and \eqref{Z-ineq}, we obtain
\begin{eqnarray*}
\deg\sO_{\ti X\lalp}(1)(-\ti N\lalp\cup B)
\geq \deg\tiX\lalp-\deg\ti N\lalp-\deg B \ge 2g(\ti X\lalp)-1.
%&=&\deg X\lalp-2g(X\lalp)-\deg B-\ell\lalp+2g(\ti X\lalp)\\ %+\deg(\ti N\lalp\cap B)\\
%&\geq&2g(\ti X\lalp) +\deg(\ti N\lalp\cap B)-1\ge 2g(\ti X\lalp)-1\ .
\end{eqnarray*}
Therefore, the last term in \eqref{H1} vanishes, which shows that the $\gamma$ in \eqref{H1} is surjective.

The section $\zeta$ mentioned before \eqref{strict-F} will be chosen by picking an appropriate $B$ and
$v\in H^0(\sO_{\ti N\lalp\cup B}(1))$ so that any element $\ti \zeta\lalp\in
\gamma\upmo(v)$ descends to a section in $H^0(\sO_{X\lalp}(1))$ and the descent
glue with $s_{i+1}|_{X\lalp\comp}$ to form a desired section $\zeta$.

%We now prove the first part of the Proposition.
We let
\beq\label{def-Z}
\ti Z_{\alpha,j}:=(\ti s_j=\cdots=\ti s_m=0)\cap \ti X\lalp\sub\ti X\lalp.
%\cI_{ Z_{j}\sub X}(1)=\sE_{j}\quad\text{and}\quad
%\cI_{\ti Z_{j}\sub \tiX\lalp}(1)=\ti\sE_{\alpha,j}:=\ti\sE_{j}\otimes_{\sO_X}\sO_{\tiX\lalp}.
\eeq
%we notice that $\delta(\ti s_i)\geq 2$ and
Since $p\in \inc(\ti s_i)\cap \ti X\lalp$, $\delta\lalp(\ti s_i)\geq 1$. In case $\delta\lalp(\ti s_i)=1$,
we choose $B=\ti Z_{\alpha,i}+p$,
% \footnote{Note that the slight advantage of defining $B$ as above instead of  $\ti Z_{\alpha,i+1}-p$ will be seen in the proof of  Corollary \ref{bdd-tail}.}
which is a subscheme of $\ti Z_{\alpha,i+1}$.
In case $\delta\lalp(\ti s_i)\geq 2$ and $\delta(\ti s_i,p)=1$, then there exists a
$p'\ne p\in \ti X\lalp$ such that $p+p'\leq \inc(\ti s_i)\cap \ti X\lalp$,
(which is equivalent to $\ti Z_{\alpha,i}+p+p'\sub\ti Z_{\alpha,i+1}$). In case
$\delta(\ti s_i, p)\geq 2$, we choose $p'=p$. Combined, we let
%\beq\label{ti-Z}
$B=\ti Z_{\alpha,i}+p+p'$.

We then let
$$v_1=\ti s_{i+1}|_{\ti N\lalp}\in H^0(\sO_{\ti N\lalp}(1))
\and v_2\ne 0\in   H^0(\sO_{{B}}(1)) \quad \text{s.t.}\  v_2|_{{B}-p}=0 .
$$
We claim that when $p\not\in \ti N\lalp$, or $\ind_p(i)\geq 1$, or $\delta(\ti s_i,p)\ge 2$,
then both $v_1|_{\ti N\lalp\cap {B}}$ and $v_2|_{\ti N\lalp\cap {B}}$ are zero.

%the image of $v_1$ (resp.  $v_2$)
%under the restriction homomorphisms from $H^0(\sO_{{B}}(1))$ (resp.  $H^0(\sO_{\ti N\lalp}(1))$)
%to $H^0(\sO_{\ti N\lalp\cap{B}}(1))$ is zero.

Indeed, since $\ti N\lalp\cap B\sub\ti Z_{\alpha,i+1}$ and $\ti s_{i+1}|_{\ti Z_{\alpha,i+1}}=0$, we have
$v_1|_{\ti N\lalp\cap {B}}=\ti s_{i+1}|_{\ti N\lalp\cap {B}}=0$.
For $v_2$, we prove case by case. Suppose
$p\not\in \ti N\lalp$, then $\ti N\lalp\cap B=\ti N\lalp\cap (B-p)$; therefore since $v_2|_{B-p}=0$,
$v_2|_{\ti N\lalp\cap B}=0$. Now suppose $p\in\ti N\lalp$. Since $v_2|_{B-p}=0$, $v_2(\bar p)=0$ for all
$\bar p\in (\ti N\lalp\cap B)-\{p\}$. We remain to show that $v_2(p)=0$.
We write ${B}=\sum_{k=0}^l n_k p_k$, $p_k$ distinct, as an effective
divisor. Since $p\in B$, we can arrange $p_0=p$. In case
$\ind_p(i)\geq 1$, we have $n_0\ge 2$; in case $\delta(\ti s_i,p)\ge 2$,
since $p'=p$ we still have $n_0\ge 2$. Thus $p\in B-p$ and $v_2(p)=0$.
This proves that $v_1$ and $v_2$ have identical images in $H^0(\sO_{\ti N\lalp\cap{B}}(1))$.
Consequently, $(v_1,v_2)$ lifts to
a section $ v\in H^0(\sO_{\ti N\lalp\cup {B}}(1))$ using the exact sequence
$$H^0(\sO_{\ti N\lalp\cup {B}}(1))\lra
H^0(\sO_{\ti N\lalp}(1))\oplus H^0(\sO_{{B}}(1))\lra
H^0(\sO_{\ti N\lalp\cap{B}}(1))\ .
$$

Since $\deg {B}\leq w\lalp(\ti\sE_i)+2$ and $i$ satisfies \eqref{prim-ineq} (, because we assume $\deg X\lalp>1$),
$\deg B$ satisfies the inequality \eqref{Z-ineq}.
Therefore, the
$\gamma$ in \eqref{H1} is surjective. We let
$\ti\zeta\lalp\in \gamma\upmo(v)\sub H^0(\sO_{\tiX\lalp}(1))$ be any lift.
Because it is a lift of $v_1$, $\ti \zeta\lalp|_{\ti N\lalp}=\ti s_{i+1}|_{\ti N\lalp}$. This implies that $\ti \zeta\lalp$
descends to a section $\zeta\lalp\in H^0(\sO_{X\lalp}(1))$, and the descent $\zeta\lalp$ glues
with $s_{i+1}|_{X\lalp\comp}$ to form a new section $ \zeta\in H^0(\sO_X(1))$.

We now prove the first part of the Proposition. We let $Z_{\alpha,j}\sub X\lalp$ be the subscheme
$Z_{\alpha,j}=(s_j=\cdots=s_m=0)\cap X\lalp$. We decompose $Z_{\alpha,j}$ into disjoint union
$Z_{\alpha,j}=R_j\cup R_j'$ so that
$R_j$ is supported at $q=\pi(p)$ and $R_j'$ is disjoint from $q$.
We let $\bar Z\lalp=(\zeta=s_{i+1}=\cdots=s_m=0)\cap X\lalp$ and decompose
$\bar Z\lalp=\bar R\cup \bar R'$ accordingly.

Suppose $q$ is a smooth point of $X$. Then $R_j$ and $\bar R$ are divisors, and can be written as
$R_j=n_j q$ and $\bar R=\bar n q$. In case $\delta\lalp(\ti s_i)=1$, the choice of $B$
ensures that $n_i=\bar n=n_{i+1}-1$ and $R_i'\sub \bar R'\subsetneq R_{i+1}'$.
Thus
$$(s_i,\cdots,s_m)\otimes_{\sO_X}\sO_{X\lalp}\subseteq
(\zeta,s_{i+1},\cdots,s_m)\otimes_{\sO_X}\sO_{X\lalp}\subsetneq
(s_{i+1},\cdots,s_m)\otimes_{\sO_X}\sO_{X\lalp}.
$$
Further, since $\delta(\ti s_i)\ge 2$ and $\zeta|_{X\lalp\comp}=s_{i+1}|_{X\lalp\comp}$, we have
$$(s_i,\cdots,s_m)\otimes_{\sO_X}\sO_{X\lalp\comp}\subsetneq
(\zeta,s_{i+1},\cdots,s_m)\otimes_{\sO_X}\sO_{X\lalp\comp}\subseteq
(s_{i+1},\cdots,s_m)\otimes_{\sO_X}\sO_{X\lalp\comp}.
$$
Thus we have
\beq\label{not-inc}
\sE_i\subsetneq\sF_{i+1}\subsetneq\sE_{i+1}.
\eeq
In case $\delta\lalp(\ti s_i)\geq 2$, the the choice of
$B$ ensures that $R_i\subsetneq \bar R\subsetneq R_{i+1}$.
Thus
$$(s_i,\cdots,s_m)\otimes_{\sO_X}\sO_{X\lalp}\subsetneq
(\zeta,s_{i+1},\cdots,s_m)\otimes_{\sO_X}\sO_{X\lalp}\subsetneq
(s_{i+1},\cdots,s_m)\otimes_{\sO_X}\sO_{X\lalp}.
$$
This implies \eqref{not-inc} as well.
In summary, by the argument at the beginning of the proof, \eqref{not-inc} leads to a contradiction
which proves that $q$ must be a node of $X$.

It remains to study the case where $q$ is a node of $X$. A careful case by case study shows that
when either $\ind_p(i)\geq 1$ or $\delta(\ti s_i,p)\geq 2$, then $Z_{\alpha,i}\subsetneq \bar Z\lalp\subsetneq Z_{\alpha,i+1}$.
Thus \eqref{not-inc} holds, which leads to a contradiction. This proves that
$q$ is a node, $\ind_p(i)=0$ and $\delta(\ti s_i,p)=1$.

We complete the proof of the first part by looking at the case $\deg X\lalp=1$.
In this case $\ind_p(i)=0$ and $\delta(\ti s_i,p)=1$, since otherwise $\deg X\lalp=1$ implies that $i=\hbar\lalp$,
contradicting to the assumption that $i$ is not a base index. We next show that $p\in L\lalp$.
But this is parallel to the proof of the case $\deg X\lalp>1$ by letting $B=p$ because $\delta\lalp(\ti s_i)=1$.
This completes the proof of the first part.

\vsp
We now prove the further part. Let $\pi\upmo(q)=\{p,p'\}$ with $p'\in\inc(\ti s_i)\cap \ti X\lbe$ so that
\eqref{prim-ineq-0} holds.
Then by the first part of the Proposition, we have $\ind_p(i)=\ind_{p'}(i)=0$; hence $s_i(q)\ne 0$.
Thus for $Z_j=(s_j=\cdots=s_m=0)\sub X$, we have $p\not\in Z_i$ and $Z_{i+1}=p\cup S$,
where $S$ is a zero-subscheme disjoint from $p$. Since $Z_i\subsetneq Z_{i+1}$
and $p\not\in Z_i$, we have $Z_i\sub S$. In case $Z_i=S$, then the further part of the Proposition
holds. Suppose $Z_i\subsetneq S$, then repeating the proof of the first part of the
Proposition, we can find a section $\zeta\in H^0(\sO_X(1))$ so that
$p\not\in (\zeta=0)$ and $S\sub (\zeta=0)$. This way, we will have \eqref{not-inc} again, which
leads to a contradiction. This proves the further part of the Proposition.
\end{proof}

The Proposition above motivates the following

\begin{defi}\label{prim-ind}
For $\deg X\lalp>1$,
we define the {\em primary} indices of $X\lalp$ be
\begin{equation*}
\II_{\alpha}\prim=\{i\in \II\lalp\mid w\lalp(\ti \sE_{i+1})
\le \deg X\lalp-2g(X\lalp)-\ell\lalp -1 \},
%\footnote{ We put $i+1$ because of $\delta(\ti s_{i},p)=v(\ti s_{i+1},p)-v(\ti s_{i},p)$ by Lemma \ref{w-delta}. };
\end{equation*}
for $\deg X\lalp=1$, we define $\II\lalp\prim=\ind\lalp\upmo(0) \sub \II\lalp$.
We say $i\in \II\lalp$ is {\em primary} at $p\in \inc(\ti{s}_i)\cap \tiX\lalp$ if
$i\in \II_{\alpha}\primary$; otherwise we say it is {\em secondary}. We define
$ \bar\jmath\lalp:=\max\{i \mid i\in \II\lalp\prim\}\ $.
\end{defi}

Note that in the proof above, the assumption $\delta(\ti s_i)\ge 2$ is used only to show that \eqref{strict-F} is strict. If $i=\hbar\lalp$ for
some $\alpha$, then $\len(\sE_i/\sE_{i+1})=\infty$. This time we  choose
$\zeta$ so that $\sE_{i}/\sF_{i+1}$ is finite. Since $\sE_{i}/\sE_{i+1}$ is infinite, \eqref{strict-F} remains strict. Hence we have

\begin{prop}\label{vir-ind-2}
Let $i=\hbar\lalp$ be a base index for some $X\lalp$, and let $p\in \inc(\ti s_i)\cap \ti X\lalp$.
Suppose $\delta(\ti s_i)\geq 1$ and $\deg X\lalp=1$, or  $w\lalp(\ti \sE_i)$ satisfies the inequality \eqref{prim-ineq}.
Then $\ind_p(i)=0$, $\delta(\ti s_i,p)=1$, and $q=\ti\pi(p)\in X\lalp$ is a
linking node of $X\lalp$. Further, let $\{p,p'\}=\pi\upmo(q)$, then $i$ must be secondary at $p'$ (cf. Definition \ref{prim-ind}),  and  there is a  component $\ti X\lbe$ so that  $p'\in \ti X\lbe$ and $i=\hbar\lbe$.
\end{prop}

\begin{proof}
The proof is parallel to the proof of the previous Proposition. We will omit it here.
\end{proof}

\begin{coro}\label{bdd-tail}
Denoting $w\prim\lalp:= w\lalp(\ti \sE_{\bar\jmath\lalp+1})$, suppose $X\lalp \subsetneq X$, then
\beq\label{cor-ineq}
0 \le \deg X\lalp-w\prim\lalp \leq 2(g(X\lalp)+\ell\lalp+1).
\eeq
\end{coro}

\begin{proof} The first inequality is trivial. We now prove the second one.
If $\deg X\lalp=1$ we obtain $\deg X\lalp-w\lalp\prim=0$, from which the second inequality trivially follows. So from now on we assume $\deg X\lalp>1$.
We let $\bar i \in \II\lalp$ be the index succeeding $\bar\jmath\lalp$; namely, $\bar i$ is the smallest index
$>\bar\jmath\lalp$ so that  $\delta\lalp(\ti s_{\bar i})\geq 1$.  In particular, this implies that
\beq\label{bar-i}
\delta\lalp(\ti s_{\bar\jmath\lalp})=\cdots=\delta\lalp(\ti s_{\bar i-1})=0.
\eeq
Since $\bar i \not\in \II\lalp\prim$,
\beq\label{esti-3}
w\prim\lalp=w\lalp(\ti\sE_{\bar \jmath\lalp+1})=w\lalp(\ti\sE_{\bar i+1})-\delta\lalp(\ti s_{\bar i})>
\deg X\lalp-2g(X\lalp)-\ell\lalp-1-\delta\lalp(\ti s_{\bar i}).
\eeq
Thus when $\delta\lalp(\ti s_{\bar i})\leq 2 $, the second inequality follows from $\ell\lalp\geq 1$ (,  since $X\lalp\subsetneq X$ ).

Suppose $\delta\lalp(\ti s_{\bar i})>2$. By our assumption $\bar i$ is the index in $\II\lalp$ immediately succeeding $\bar\jmath\lalp$, we have
$w\lalp(\ti\sE_{\bar i})=w\lalp(\ti\sE_{\bar\jmath\lalp+1})$ because of \eqref{bar-i}.
By Definition \ref{prim-ind}, $w\lalp(\ti\sE_{\bar i})$ satisfies \eqref{prim-ineq}. So we can
apply Proposition \ref{vir-ind} to the index $\bar i$ to conclude that every
$p\in \inc(\ti s_{\bar i})\cap \ti X\lalp $
lies in $\ti N\lalp$ and has $\delta(\ti s_{\bar i},p)=1$.

We claim that $ \inc(\ti s_{i})\cap \ti X\lalp\sub \ti L\lalp$. Indeed, let
$p\in\inc(\ti s_{\bar i})\cap(\ti N\lalp\setminus\ti L\lalp)$,
then the second part of Proposition \ref{vir-ind} implies that $\inc(\ti s_{\bar i})=p+p'$ and $\delta(\ti s_{\bar i})=2$, contradicting to the assumption
$\delta\lalp(\ti s_{\bar i})>2$.
This proves that $\inc(\ti s_{\bar i})\cap \ti X\lalp\sub
\ti L\lalp$. Adding that $\delta(\ti s_{\bar i},p)=1$ for $p\in \inc(\ti s_{\bar i})\cap \ti X\lalp$, we conclude that
$\delta\lalp(\ti s_{\bar i})\leq \ell\lalp$. These and \eqref{esti-3} proves the second inequality in  \eqref{cor-ineq}.
\end{proof}

\section{Main estimate for irreducible curves}\label{main-estimate}
Throughout this section, we fix a staircase 1-PS $\lam$, and an irreducible $X\lalp$.  We will
derive a sharp estimate of $e(\ti\sI\lalp(\lam))$ for the  $X\lalp\subset X$.

We let  $ g\lalp$ be the genus of $
X\lalp$; we define the set of {\em special points}
\beq\label{S}
{\ti S\lalp}=(\pi\upmo(\bx)\cap \tiX\lalp))\cup \ti N \lalp \sub\ti X\lalp\ ,
\eeq
where $\bx=(x_1,\cdots,x_n)\sub X$ is the set of weighted points.
We continue to
denote by $\bar\rho_i=\rho_i-\rho_{\hbar\lalp}$. For each $p\in\ti \Lam\lalp$, we define
the {\em initial index}
\beq\label{i0} i_0(p):=\min\{i\ |\ i\in \II_p\}.
\eeq
{For $\deg X\lalp>1$} and a fixed $\epsilon>0$, we define
\beq\label{E-alp}
E\lalp(\rho):=
\Bl2+\frac{2\epsilon}{\deg X_{\alpha}}\Br\sum_{i\in
\II_{\alpha}^{\basic}} \delta\lalp(\ti s_i)\bar{\rho}_{i} -\Bl1
+\frac{2\epsilon}{\deg X\lalp}\Br\sum_{q\in {\ti S\lalp}\cap\ti
\Lam\lalp}\bar{\rho}_{i_0(q)}+2\deg X_{\alpha}\cdot
\rho_{\hbar\lalp};
\eeq
for $\deg X\lalp=1$ satisfying $\bx\cap X\lalp=\emptyset$, we define
\beq\label{E-alp-0}
E\lalp(\rho):=\delta\lalp(\ti s_{i_0})\bar{\rho}_{i_0}+2\cdot \rho_{\hbar\lalp};\quad
i_0=\ind\lalp\upmo(0).
\eeq
It is clear that in both cases $E\lalp(\rho)$ are linear in $\rho\in\RR^{m+1}_+$. Our main result of this section is the following

\begin{theo}\label{main-est}
For any $1\geq\epsilon>0$ there is a constant $M$ depending only on
$g\lalp$, $\ell\lalp$ and $\epsilon$ such that whenever $\deg
X_{\alpha}\geq M$, then
$$
e(\ti\sI\lalp(\lam))\leq E\lalp(\rho).
$$
In case $\deg X\lalp=1$ and $\bx\cap X\lalp=\emptyset$, the same inequality  holds for  $E\lalp(\rho)$ defined in \eqref{E-alp-0}.
%\footnote{Notice that in this case, the estimate is better.}
%If $X\lalp$ is a smooth rational curve,  then the inequality holds
%for $\epsilon=0$ and $\II\lalp^{\basic}$ replaced by $\II\lalp$ in
%the definition of $E\lalp(\rho)$.
\end{theo}
\black
Note that the theorem implies that we can bounded
$e(\ti\sI(\lam)))$ in terms of the primary $\rho_i$'s only. And for
primary indices, we have a complete understanding of the
multiplicity $\delta\lalp(\ti s_i)$ due to the detailed study in the previous section.

We begin with the following bound on the area of $\Delta_p$ in terms of $\{\rho_i\}$.
\begin{lemm}\label{trapezoid}
Let $\lam$ be a staircase. Then for each $p\in \ti \Lam\lalp$ and
any $0\leq l\leq k\leq \hbar\lalp$,  we have
\beq
|\Delta_p\cap([w(\ti \sE_l,p),w(\ti \sE_k,p)]\times\RR)|
-\rho_{\hbar\lalp}\!\!\!\cdot (w(\ti\sE_k,p)-w(\ti\sE_l,p))\leq
\eeq
$$
\qquad\qquad\leq \sum_{i\in \II_p\cap[l,k-1]}\!\!\!\!\!\delta(\ti
s_i,p)\bar{\rho}_i
-\frac{(\bar{\rho}_{i_{\min}(p)}+\bar{\rho}_{i_{\max}(p)})}{2}\ ,
$$
where $i_{\min}:=\min(\II_p\cap[l,k-1])$ and
$i_{\max}:=\max(\II_p\cap[l,k-1])$.
\end{lemm}

Note that by letting  $l=0$ and $k=\hbar\lalp$, we obtain
\beq\label{half-rho}
|\Delta_p|
-\rho_{\hbar\lalp}\!\!\!\cdot w(\ti\sE_{\hbar\lalp},p)\leq
\sum_{i\in \II_p}\delta(\ti s_i,p)\bar{\rho}_i
-\frac{\bar{\rho}_{i_0(p)}}{2}.
\eeq

\begin{proof}
First, we notice that the above inequality is invariant when varying
$\rho_{\hbar\lalp}$, thus  to prove the Lemma we can and do assume from
now on that $\rho_{\hbar\lalp}=0$; hence $\bar{\rho_{i}}=\rho_{i}$.

Let $\Gamma_p:=\{(w(\ti \sE_i,p),\rho_i)\}_{0\leq i\leq m}$; it
follows from Definition  \ref{Delta-p} and \ref{staircase} that
\beq\label{Delta-p-1}
\Delta _{p}=(
\mathbb{R}_{+}^{2}-\mathrm{Conv}( \mathbb{R} _{+}^{2}+\Gamma_p ) )
\cap ( [ 0,w(\ti\sI,p) ] \times \mathbb{R})\ . \eeq
Fixing an indexing
\beq\label{index-p} \II_p=\{i_0(p),\cdots,i_d(p)\}\sub \II,\quad
i_j(p)\ \text{increasing and}\ d+1=|\II_p|, \eeq we let $\TT$ be the
continuous piecewise linear function on $[0,w(\ti \sI,p)]$ defined
by linear interpolating the points
$$\{(0,\rho_{i_0}),\cdots,(w(\ti \sE_{i_k},p),\rho_{i_k}),
\cdots, %(w(\ti \sE_{i_k+1},p)-1,\rho_{i_k}),(w(\ti \sE_{i_{k+1}},p),\rho_{i_{k+1}}),
(w(\ti \sE_{i_{d}},p),\rho_{\hbar\lalp})
 \}\sub \RR^2\ ,$$
and let $\Delta_\TT$ be the polygon bounded on two sides by $x=0$ and
$x=w(\ti \sE_k,p)$, from below by $y=0$ and from above by the graph of $y=\TT$.
By the convexity of $\Delta_p$, we have
$$\Delta_p\cap([w(\ti \sE_l,p),w(\ti \sE_k,p)]\times\RR)\subset\Delta_\TT\cap([w(\ti \sE_l,p),w(\ti\sE_k,p)]\times\RR)\subset \RR^2\ .$$
By Lemma \ref{w-delta}, $w(\ti \sE_i,p)=\sum_{j=0}^{ i-1}\delta(\ti
s_j,p)$; hence
\begin{eqnarray*}
&&|\Delta_p\cap([w(\ti \sE_l,p),w(\ti \sE_k,p)]\times\RR)| \leq
|\Delta_\TT\cap([w(\ti \sE_l,p),w(\ti \sE_k,p)]\times\RR)|
\\
& \leq& \displaystyle\sum_{i\in \II_p\cap[l,k-1]}\delta(\ti
s_i,p)\rho_{i}- \frac{1}{2}(\rho_{i_{\min}(p)}+\rho_{i_{\max}(p)}).
\end{eqnarray*}
This proves the Lemma.
\end{proof}

With this lemma in hand, we now explain the key ingredient in the
proof of the theorem. We will divide our estimates into two cases
according to the size of $|\ti\Lam\lalp|$ (cf. Definition \ref{index}). When $|\ti\Lam\lalp|$ is large,
applying Lemma \ref{trapezoid}, we will
gain a sizable multiple of $\frac{1}{2}\rho_{i_0(p)}$'s (cf. \eqref{half-rho}) in the estimate of $\Delta_p$;
these extra gains will take care of the
contributions from non-primary $\rho_i$'s. When $|\ti\Lam\lalp|$ is
small, one large $\Delta_p$ is sufficient to cancel the contribution
from the non-primary $\rho_i$'s.

We need a few more notions. For any $p\in
\ti \Lam\lalp$, we let $ \II\prim_p:=\II\prim\lalp\cap \II_p, $ and
define
\beq\label{Wp}
 \bar\jmath_p:=\max\{i\in \II_p\prim\},\quad
w\prim(p):=w(\ti\sE_{\bar\jmath\lalp+1},p),\quad\text{and}\quad
 w(p):=w(\ti\sI,p) \text{ (cf. (\ref{hbar}))} .
\eeq
Note that $w(p)$ is the base-width of the Newton polygon $\Delta_p$.
Using $\bar\jmath_p$, we truncate the Newton polygon $\Delta_p$ by intersecting
it with the strip $[0,w^{\basic}(p)]\times\RR$:
$$ \Delta^{\basic}_{p}:=\Delta_p\cap [0,w^{\basic}(p)]\times\RR.
$$

Our next Lemma says that if one $\Delta_p$ is big enough, the contribution from non-primary
$\rho_i$'s will be absorbed by the difference between
$E\lalp(\rho)$ and $e(\ti\sI\lalp(\rho))$.
\begin{lemm}\label{eps}
For any $1>\epsilon>0$, there is an $M$ depending only on $\epsilon$, $g\lalp$ and $\ell\lalp$  such that
whenever $w(p)\geq M$ (cf.\eqref{Wp}),
\begin{equation*}\label{one-tower}
|\Delta_{p}^{\basic}|+2(\deg X\lalp-w\lalp\prim)\bar{\rho}_{\bar\jmath_{\alpha}}
\leq
\Bigl(1+\frac{\epsilon}{w(p)}\Bigr)\!\!\sum_{i\in \II^{\basic}_{p}} \delta(\ti s_i,p)\bar{\rho}_{i}
+\rho_{\hbar\lalp}\cdot w\prim(p)-\Bl\frac{1}{2}+\frac{\epsilon}{w(p)}\Br\bar\rho_{i_0(p)},
\end{equation*}
where $w\lalp\prim$ is defined in Corollary \ref{bdd-tail}.
%\beq
%|\Delta_{q,K}^{\basic}|-\frac{1}{2}\rho_{i_0(q)}+(2\ti g\lalp+K)\bar{\rho}_{\basic_{X_{\alpha},K}}
%\leq\Bl1+\frac{\epsilon}{n}\Br\!\!(\sum_{i\in \II^{\basic}_{q,K}}\bar{\rho}_{i}-\rho_{i_0(q)})+\rho_{b(X_{\alpha})}\cdot \bar\jmath_K(q).
%\eeq
\end{lemm}

\begin{proof}
By the same reason as in the proof of Lemma \ref{trapezoid}, we only
need to treat the case that $\rho_{\hbar\lalp}=0$; hence
$\bar{\rho_{i}}=\rho_{i}$.

Our proof relies on the proximity of $\partial^+\Delta_{p}$ ($\partial^+\Delta_p$ is the
 boundary component of $\Delta_{p}$ lying in the (open) first quadrant)
with the lattice points $(w(\ti\sE_i,{p}), \rho_i)$ (cf. \eqref{w}). In case they differ slightly, then
the term $\frac{\epsilon}{w({p})}\sum_{i\in \II^{\basic}_{{p}}} \delta(\ti s_i,{p})\bar{\rho}_{i}$ is sufficient
to absorb the term $2(\deg X\lalp-w\lalp\prim)\rho_{\bar\jmath\lalp}$
in the inequality (note $\bar\rho_i=\rho_i$ by assumption).
Otherwise, the difference between $\sum_{i\in \II\prim_p\cap[c,\bar\jmath_p]} \delta(\ti s_i,{p})\bar{\rho}_{i}$
(for some $c$ that will be specified below )
and $|\Delta_{p}|$ is sufficient to imply the desired estimate.

Let us assume $M>4$, then $w(p)-\sqrt{w(p)}\geq2 $ whenever $w(p)\geq M$. We introduce
$$c=\max\{i\in \II\prim_{p}\mid (w(\ti\sE_i,{p}),\frac{\rho_i}{2})\in \Delta_p\sub\RR^2\}
\quad\text{and}\quad w^c(p):=w(\ti\sE_c,p)\ ,$$ and let
$$\Delta_{p}^{\le c}=\Delta_{p}\cap [0,w^c(p)]\times\RR\ .
%\and\Delta_{p}^{\ge c}=\Delta_{p}\cap [w^c(p), w(\ti\sI,{p})]\times\RR\ .
$$

We divide our study into two cases. The first is when
$w(p)-w^c(p)\leq \sqrt{w({p})}$. We introduce trapezoid $\Theta$ to be the region
between $x$-axis and the line passing through the points
$(w(p),0)$ and $(w^c(p),\frac{\rho_c}{2})$ intersecting with the strip  $[1,w^c(p)]\times \RR$.
Then by our assumption
$$w^c(p)-1\geq w(p)-\sqrt{w(p)}-1\geq \frac{w(p)-\sqrt{w(p)}}{2}\ .$$
Since the length of the two vertical edges of $\Theta$ are of  $\frac{\rho_c}{2}$ and
$\frac{w(p)-1}{w(p)-w^c(p)}\cdot \frac{\rho_c}{2}$, we deduce
\begin{eqnarray*}
|\Theta|&=&\Bl\frac{w(p)-1}{w(p)-w^c(p)}+1\Br(w^c(p)-1)\frac{\rho_c}{4}\\
&\geq&\Bl\frac{\sqrt{w(p)}}{2}+1\Br
\cdot \frac{w(p)-\sqrt{w(p)}}{2}\cdot\frac{\rho_c}{4}
\geq\frac{w(p)^{3/2}\cdot \rho_c}{32}\ .
\end{eqnarray*}
Since the piecewise linear $\partial^+\Delta_{p}$ is
convex, $\Theta$ lies inside $\Delta_{p}$, hence
$$|\Delta_p|-\frac{\rho_{i_0(p)}}{2}>|\Theta|> \frac{w(p)^{3/2}}{32}\rho_{c}\ .$$
By the definition of $\Delta_p^{\basic}$, the difference between the base-width of $\Delta_p^{\basic}$ and of $\Delta_p$ is bounded
by $w(p)-w^{\basic}(p)$; therefore by Lemma \ref{trapezoid} we have
\begin{eqnarray*}
|\Delta_p^{\basic}|-\frac{\rho_{i_0(p)}}{2}+2(\deg X\lalp-w\lalp\prim)\rho_{\bar\jmath\lalp}
&\geq&|\Delta_p^{\basic}|-\frac{\rho_{i_0(p)}}{2}+(w(p)-w^{\basic}(p))\rho_{\bar\jmath\lalp}\\
&\geq&|\Delta_p|-\frac{\rho_{i_0(p)}}{2}\geq\frac{w(p)^{3/2}}{32}\rho_{c}\ .
\end{eqnarray*}
Since $\rho_{\bar\jmath\lalp}\leq\rho_{c}$, this implies
\beq\label{64}
|\Delta_p^{\basic}|-\frac{\rho_{i_0(p)}}{2}>\Bl\frac{w(p)^{3/2}}{32}-2(\deg X\lalp-w\lalp\prim)\Br\rho_{c}
 \ .
\eeq

We now choose $M$ so that
$M^{3/2}\geq 2^8(g\lalp+\ell\lalp+1)$. By Corollary \ref{bdd-tail}, we have $\deg X\lalp -w\lalp\prim\leq2(g\lalp+\ell\lalp+1)$. Therefore when
$w(p)\geq M$, we have
$$2(\deg X\lalp-w\lalp\prim)\leq 4(g\lalp+\ell\lalp+1)\leq\frac{w(p)^{3/2}}{64} \ .$$
Pluging this into  \eqref{64}, we obtain
$$|\Delta_p^{\basic}|-\frac{\rho_{i_0(p)}}{2}>\frac{w(p)^{3/2}}{64}\rho_c
\quad\text{, equivalently, }\quad \rho_c\leq\frac{2^6}{w(p)^{3/2}}\bl|\Delta_p^{\basic}|-\frac{\rho_{i_0(p)}}{2}\br \ ,$$
hence
$$2(\deg X\lalp-w\lalp\prim)\rho_{\bar\jmath\lalp}\leq2(\deg X\lalp-w\lalp\prim)\rho_{c}
\leq\frac{2^6(\deg X\lalp-w\lalp\prim)}{w(p)^{3/2}}\Bl|\Delta_p^{\basic}|-\frac{\rho_{i_0(p)}}{2}\Br.
$$
So if we assume further $M\geq 2^{14}(g\lalp+\ell\lalp+1)^2/\epsilon^2$, then whenever $w(p)\geq M$ we have
$2^6(\deg X\lalp-w\lalp\prim)w(p)^{-3/2}\leq \epsilon/w(p)$, thus
\begin{eqnarray*}
|\Delta_p^{\basic}|-\frac{1}{2}\rho_{i_0(p)}+2(\deg X\lalp-w\lalp\prim)\rho_{\bar\jmath\lalp}
&\leq&  \Bl1+\frac{\epsilon}{w(p)}\Br\Bl|\Delta_p^{\basic}|-\frac{\rho_{i_0(p)}}{2}\Br\\
&\leq& \displaystyle\Bl1+\frac{\epsilon}{w(p)}\Br
\bl\sum_{i\in \II^{\basic}_p}\delta(\ti s_i,p)\rho_{i}-\rho_{i_0(p)}\br,
\end{eqnarray*}
where the last inequality follows from Lemma \ref{trapezoid}. %with $l=0$ and $k=\bar\jmath\lalp$.
This proves the Lemma in this case.

The second case is when
$w(p)-w^c(p)>\sqrt{w(p)}$.
By the definition of $c$,  for $\ j\in \JJ:=\II_p\cap(c,\bar\jmath\lalp]$,
$(w(\ti\sE_j , p),\rho_j/2)\notin\Delta_p $.
Since $\partial^+\Delta_p$ is convex, by Lemma \ref{trapezoid}, we have
$$\sum_{i \in \JJ }\delta(\ti s_i,p)\rho_{i}-|\Delta_p^{\basic}\setminus\Delta_p^{\leq c}|
\geq\sum_{i \in \JJ}\delta(\ti s_i,p)\rho_{i}/2.$$
Notice that
\begin{eqnarray*}
\sum_{i\in\JJ}\delta(\ti s_i,p)
&=&w^{\basic}(p)-w^c(p)=w(p)-w^c(p)-(w(p)-w\prim (p))\\
&>&\sqrt{w(p)}-(\deg X\lalp-w\lalp\prim)\ ,
\end{eqnarray*}
since $\deg X\lalp-w\lalp\prim\geq w(p)-w\prim (p)$ and $w(p)-w^c(p)>\sqrt{w(p)}$ by our assumption. If we choose
$$M\geq 10^2(g\lalp+\ell\lalp+1)^2\geq 5^2(\deg X\lalp-w\lalp\prim)^2$$
and require $w(p)\geq M$, then
$\sum_{i\in\JJ}\delta(\ti s_i,p)\geq  4(\deg X\lalp-w\lalp\prim)$.
This implies
$$\sum_{i \in \JJ }\delta(\ti s_i,p)\rho_{i}-|\Delta_p^{\basic}\setminus\Delta_p^{\leq c}|\geq
\sum_{i \in \JJ}\delta(\ti s_i,p)\rho_{i}/2\geq 2(\deg X\lalp-w\lalp\prim)\rho_{\bar\jmath\lalp},
$$
and combined with Lemma \ref{trapezoid}, we obtain
\begin{eqnarray*}
|\Delta_p^{\basic}|
&+&2(\deg X\lalp-w\lalp\prim)\rho_{\bar\jmath\lalp}\\
&\leq& |\Delta_p^{\leq c}|+|\Delta_p^{\basic}\setminus \Delta_p^{\leq c}|+2(\deg X\lalp-w\lalp\prim)\rho_{\bar\jmath\lalp}\\
&\leq&  |\Delta_p^{\leq c}|+|\Delta_p^{\basic}\setminus \Delta_p^{\leq c}|-\sum_{i \in \JJ }\delta(\ti s_i,p)\rho_{i}+
\sum_{i \in \JJ }\delta(\ti s_i,p)\rho_{i}+2(\deg X\lalp-w\lalp\prim)\rho_{\bar\jmath\lalp}\\
&\leq&  |\Delta_p^{\leq c}|+\sum_{i \in \JJ}\delta(\ti s_i,p)\rho_{i}
\leq  \sum_{i\in \II^{\basic}_p}\delta(\ti s_i,p)\rho_{i}-\frac{\rho_{i_0(p)}}{2}\\
%&\leq&  \sum_{i\in \II^{\basic}_p}\delta(\ti s_i,p)\rho_{i}
&< & \bigl(1+\frac{\epsilon}{w(p)}\bigr)\bl\sum_{i\in \II^{\basic}_p}\delta(\ti s_i,p)\rho_{i}-\frac{\rho_{i_0(p)}}{2}\br
+\frac{\rho_{i_0(p)}}{2}.
\end{eqnarray*}

%Finally, by Corollary \ref{bdd-tail} we have
%$$|w(p)-w^{\basic}(p)|\leq|\deg X\lalp-w\prim\lalp|\leq 2\ti g\lalp+2K$$

In the end, since $\epsilon<1$ we choose
$M:=  2^{14}(g\lalp+\ell\lalp+1)^2/\epsilon^2$. Then for $w(p)>M$, (\ref{one-tower}) holds.
This proves the Lemma.
\end{proof}

\begin{proof}[Proof of Theorem \ref{main-est}]
First, by the same reason as in the proof of Lemma \ref{one-tower} we only need to deal with the case
$\rho_{\hbar\lalp}=0$ and
$\bar{\rho}_{i}=\rho_{i}$.
Also, when $\deg X\lalp=1$, then the statement is a direct consequence of Lemma \ref{trapezoid}. So from now on
we assume that $\deg X\lalp>1$. Let $1>\epsilon>0$ be any constant.  %we will determine the corresponding $M$.
Since $\epsilon< 1$, we have $\frac{\epsilon}{\deg X\lalp}\leq 1/2$ whenever $\deg X\geq M\geq 2$.
% and this will be used in the rest of the proof.
We define $\sigma$ to be the number of Newton polytopes supported on $\ti X\lalp$.
 We divide our study into two cases.

The first case is when $\sigma>10(g\lalp+\ell\lalp+1)+|{\ti S\lalp}|$.
Since Corollary \ref{bdd-tail} implies
$$%\sum_{\substack{p\in \ti \Lam\lalp\cap {\ti S\lalp}\\ i_0\in \II\lalp\setminus \II\lalp^{\basic}}}\rho_{i_0(p)}\leq
|\{p\in \ti \Lam\lalp\cap \ti S\lalp\mid i_0(p)>\bar\jmath\lalp\}|\leq
\sum_{i\in \II\lalp\setminus \II\lalp^{\basic}}\delta\lalp(\ti s_i)\leq(\deg X\lalp-w\prim\lalp)\leq 2( g\lalp+\ell\lalp+1)
\ ,$$
the number of $p\in \ti \Lam\lalp\setminus {\ti S\lalp}$ satisfying $\rho_{i_0(p)}\geq \rho_{\bar\jmath\lalp}$ is
at least $8(\ti g\lalp+\ell\lalp+1)$. By
Lemma \ref{trapezoid}, for each $p\in \ti \Lam\lalp$,  we gain an extra
$\rho_{i_0(p)}/2$ on the right hand side in  the estimate $\Delta_p$ in terms of $\{\rho_i\}_{i=0}^m$.
This implies
\begin{eqnarray}\label{rho-tail}
&&\sum_{i\in \II\lalp\setminus \II\lalp^{\basic}}\delta\lalp(\ti s_i)\rho_{i}\leq(\deg X\lalp-w\prim\lalp)\rho_{\bar\jmath\lalp}
\leq2(g\lalp+\ell\lalp+1)\rho_{\bar\jmath\lalp}\leq
\frac{1}{4}\sum_{p\in \ti \Lam\lalp\setminus {\ti S\lalp}}\rho_{i_0(p)}\ .
\end{eqnarray}
So we obtain, via using Lemma \ref{trapezoid} and summing over $p\in \ti \Lam\lalp$,
\begin{eqnarray}
\sum_{p\in \ti \Lam\lalp}|\Delta_p|
&\leq &\sum_{i\in \II\lalp^{\basic}}\delta\lalp(\ti s_i)\rho_{i}
+\sum_{i\in \II_{\alpha}\setminus\II\lalp^{\basic}}\delta\lalp(\ti s_i)\rho_{i}
-\frac{1}{2}\sum_{p\in \ti \Lam\lalp}\rho_{i_0(p)}\nonumber\\
&=&\Bl\sum_{i\in \II\lalp^{\basic}}\delta\lalp(\ti s_i)\rho_{i}
-\frac{\epsilon}{\deg X\lalp}
\sum_{\substack{p\in \ti \Lam\lalp\cap {\ti S\lalp}\nonumber\\
i_0(p)\in \II\lalp\setminus\II\lalp\prim}}\rho_{i_0(p)}
-\frac{1}{2}\sum_{p\in \ti \Lam\lalp\cap {\ti S\lalp}}\rho_{i_0(p)}\Br+\
\nonumber\\
&&+\,\Bl\sum_{i\in \II_{\alpha}\setminus\II\lalp^{\basic}}\delta\lalp(\ti s_i)\rho_{i}
+\frac{\epsilon}{\deg X\lalp}\sum_{\substack{p\in \ti \Lam\lalp\cap {\ti S\lalp}\\ i_0(p)\in \II\lalp\setminus\II\lalp\prim}}\rho_{i_0(p)}
-\frac{1}{2}\sum_{p\in \ti \Lam\lalp\setminus\ti S\lalp}\rho_{i_0(p)}\label{last}\Br.
\end{eqnarray}
Using \eqref{rho-tail} and
$$
\frac{\epsilon}{\deg X\lalp}
\sum_{\substack{p\in \ti \Lam\lalp\cap {\ti S\lalp}\\ i_0(p)\in \II\lalp\setminus\II\lalp\prim}}\rho_{i_0(p)}
\leq\sum_{i\in \II\lalp\setminus\II\lalp^{\basic}}\delta\lalp(\ti s_i)\rho_i
\leq\frac{1}{4}\sum_{p\in \ti \Lam\lalp\setminus {\ti S\lalp}}\rho_{i_0(p)}\ ,
$$
the sum in the line  of \eqref{last} is non-positive. Therefore, for any $0<\epsilon <1$ we have
$$\sum_{p\in \ti \Lam\lalp}|\Delta_p|\ \leq \
\sum_{i\in \II\lalp^{\basic}}\delta\lalp(\ti s_i)\rho_i
-\frac{\epsilon}{\deg X\lalp}
\sum_{\substack{p\in \ti \Lam\lalp\cap {\ti S\lalp}\\ i_0(p)\in \II\lalp\setminus\II\lalp\prim}}\rho_{i_0(p)}
-\frac{1}{2}\sum_{p\in \ti \Lam\lalp\cap {\ti S\lalp}}\rho_{i_0(p)}\qquad\qquad
$$
$$\qquad\qquad\qquad
\leq
\Bl1+\frac{\epsilon}{\deg X\lalp}\Br\sum_{i\in \II\lalp^{\basic}}\delta\lalp(\ti s_i)\rho_i
-\Bl\frac{1}{2}
+\frac{\epsilon}{\deg X\lalp}\Br\sum_{p\in \ti \Lam\lalp\cap\ti S\lalp}\rho_{i_0(p)}=\frac{E\lalp(\rho)}{2}\ ,
$$
since
$$
\frac{\epsilon}{\deg X\lalp}\sum_{\substack{p\in \ti \Lam\lalp\cap {\ti S\lalp}\\ i_0(p)\in \II\lalp\prim}}\rho_{i_0(p)}
\leq
\frac{\epsilon}{\deg X\lalp}
\sum_{i\in \II\lalp^{\basic}}\delta\lalp(\ti s_i)\rho_i\ .
$$
This verifies the Theorem in this case.

The other case is when $\sigma\leq 10(g\lalp+\ell\lalp+1)+|{\ti S\lalp}|$. By the pigeon hole principle,
there exists  at least one $p_0\in \ti \Lam\lalp$ such that
\beq\label{w-deg-X}
w(\ti \sI,p_0)\geq \frac{\deg X_{\alpha}}{\sigma}\geq \frac{\deg X_{\alpha}}
{10( g\lalp+\ell\lalp+1)+|{\ti S\lalp}|}\ .
\eeq
By Corollary \ref{sum-Delta}, we have
$$\frac{e_{X\lalp}(\sI(\lam))}{2}=\sum_{p\in\ti \Lam\lalp}|\Delta_p|.
$$
Our assumption $\epsilon\leq 1,\ 1/\deg X\leq 1/2$ and Corollary \ref{bdd-tail} imply
\beq\label{tail-rho0}
\bl\frac{1}{2}+\frac{\epsilon}{\deg X\lalp}\br
\sum_{\substack{p\in {\ti S\lalp}\cap\ti \Lam\lalp\\ i_0(p)\in\II\lalp\setminus\II\lalp\prim}}\rho_{i_0(p)}
\leq\sum_{\substack{p\in {\ti S\lalp}\cap\ti \Lam\lalp\\ i_0(p)\in\II\lalp\setminus\II\lalp\prim}}\rho_{i_0(p)}
\leq(\deg X\lalp-w\lalp\prim)\rho_{\bar\jmath\lalp}\ .
\eeq
So we obtain
\begin{eqnarray*}
&&\frac{e_{X\lalp}(\sI(\lam))}{2}
-\sum_{p\in {\ti S\lalp}\cap\ti \Lam\lalp}\frac{\rho_{i_0(p)}}{2}\nonumber\\
&=&|\Delta\prim_{p_0}|+|\Delta_{p_0}\setminus\Delta\prim_{p_0}|
+\sum_{p_0\ne p\in \ti \Lam\lalp}(|\Delta\prim_p|+|\Delta_p\setminus\Delta\prim_p|)
-\sum_{p\in {\ti S\lalp}\cap\ti \Lam\lalp}\frac{\rho_{i_0(p)}}{2}\ .\nonumber\\
\end{eqnarray*}
By Lemma \ref{trapezoid} and the first inequality of \eqref{rho-tail}, we have
$$
|\Delta_{p_0}\setminus\Delta\prim_{p_0}|
+\sum_{p_0\ne p\in \ti \Lam\lalp}|\Delta_p\setminus\Delta\prim_p|
=\sum_{p\in \ti \Lam\lalp}|\Delta_p\setminus\Delta\prim_p|\leq
(\deg X\lalp-w\prim\lalp)\rho_{\bar\jmath\lalp}\ .
$$
So
\begin{eqnarray}
&&\frac{e_{X\lalp}(\sI(\lam))}{2}
-\sum_{p\in {\ti S\lalp}\cap\ti \Lam\lalp}\frac{\rho_{i_0(p)}}{2}\nonumber\\
&\leq&|\Delta\prim_{p_0}|+(\deg X\lalp-w\prim\lalp)\rho_{\bar\jmath\lalp}+\sum_{p_0\ne p\in \ti \Lam\lalp}|\Delta\prim_p|
-\sum_{p\in {\ti S\lalp}\cap\ti \Lam\lalp}\frac{\rho_{i_0(p)}}{2}\nonumber
\end{eqnarray}

\begin{eqnarray}\label{last-term}
\qquad
&\leq&|\Delta\prim_{p_0}|+2(\deg
X\lalp-w\prim\lalp)\rho_{\bar\jmath\lalp}-\frac{\rho_{i_0(p_0)}}{2}|\{p_0\}\cap
{\ti S\lalp}|
+\sum_{p_0\ne p\in \ti \Lam\lalp}|\Delta\prim_p|\\
&&-\sum_{p_0\ne p\in {\ti S\lalp}\cap\ti \Lam\lalp}\frac{\rho_{i_0(p)}}{2}
-\bl\frac{1}{2}+\frac{\epsilon}{\deg X\lalp}\br\sum_{\substack{p\in {\ti S\lalp}\cap\ti \Lam\lalp\\ i_0(p)\in\II\lalp\setminus \II\lalp\prim}}
\rho_{i_0(p)}\nonumber\\
\nonumber
\end{eqnarray}
where we have used \eqref{tail-rho0} in \eqref{last-term}.
By definition, $|{\ti S\lalp}|\leq n+\ell\lalp+g\lalp$. Let
$$\epsilon_0= \frac{\epsilon}{11(g\lalp+\ell\lalp+1)+n} \leq\frac{\epsilon}{10(g\lalp+\ell\lalp+1)+|\ti S\lalp|}\ ,$$
by \eqref{w-deg-X} we obtain
\beq\label{ep-0}
\frac{\epsilon_0}{w(\ti\sI,p_0)}\leq
\frac{\epsilon}{w(\ti\sI,p_0)(10(g\lalp+\ell\lalp+1)+|\ti S\lalp|)}\leq
\frac{\epsilon}{\deg X\lalp}
\eeq
If we let $M=M(\epsilon_0)$ be the constant fixed in
Lemma \ref{one-tower} for $\epsilon_0$ and choose
$$
M' \ge (11( g\lalp+\ell\lalp+1)+n)M\ge (10( g\lalp+\ell\lalp+1)+|{\ti S\lalp}|)M,
$$
then $\deg X\lalp\geq M'$ implies $w(\ti \sI,p_0)> M$.  In particular, we have
$i_0(p_0)\in \II\lalp\prim$. The whole term after \eqref{last-term} is equal to
\begin{eqnarray*}
&=&|\Delta\prim_{p_0}|+2(\deg
X\lalp-w\prim\lalp)\rho_{\bar\jmath\lalp}-\frac{\rho_{i_0(p_0)}}{2}|\{p_0\}\cap
{\ti S\lalp}|
+\sum_{p_0\ne p\in \ti \Lam\lalp}|\Delta\prim_p|\\
&&-\sum_{\substack{p_0\ne p\in {\ti S\lalp}\cap\ti \Lam\lalp\\ i_0(p)\in \II\lalp\prim}}\frac{\rho_{i_0(p)}}{2}
-\bl 1+\frac{\epsilon}{\deg X\lalp}\br\sum_{\substack{p\in {\ti S\lalp}\cap\ti \Lam\lalp\\ i_0(p)\in\II\lalp\setminus \II\lalp\prim}}
\rho_{i_0(p)}\nonumber.\\
\end{eqnarray*}
 Applying Lemma \ref{trapezoid} to the term
 $|\Delta\prim_{p_0}|+2(\deg
X\lalp-w\prim\lalp)\rho_{\bar\jmath\lalp}$,  Lemma \ref{one-tower} to the term
$\sum_{p_0\ne p\in \ti \Lam\lalp}|\Delta\prim_p|
-\sum_{p_0\ne p\in {\ti S\lalp}\cap\ti \Lam\lalp, i_0(p)\in \II\lalp\prim}\frac{\rho_{i_0(p)}}{2}$ in
the above identity and using  \eqref{ep-0} we obtain
\begin{eqnarray*}%\label{3-term}
\qquad\qquad\nonumber
&&\frac{e_{X\lalp}(\sI(\lam))}{2}
-\sum_{p\in {\ti S\lalp}\cap\ti \Lam\lalp}\frac{\rho_{i_0(p)}}{2}\nonumber\\
&\leq&\Bl1+\frac{\epsilon_0}{w(\ti\sI,p_0)}\Br(\sum_{i\in \II_{p_0}\prim} \delta(\ti
s_i,p_0)\rho_{i}-\rho_{i_0(p_0)})+\frac{\rho_{i_0(p_0)}}{2}(1-|\{p_0\}\cap
{\ti S\lalp}|)
+\\
&&+\bl\sum_{p_0\ne p\in\ti \Lam\lalp}\sum_{i\in \II_{p}^{\basic}} \delta(\ti s_i,p)\rho_{i}
-\!\!\!\!\!\!\sum_{\substack{p_0\ne p\in {\ti S\lalp}\cap\ti \Lam\lalp\\ i_0(p)\in \II\lalp\prim}}\rho_{i_0(p)}\br
-\bl1+\frac{\epsilon}{\deg X\lalp}\br\!\!\!\!\!\!\sum_{\substack{p\in {\ti S\lalp}\cap\ti \Lam\lalp\\ i_0(p)\in \II\lalp\setminus\II\lalp\prim}}
\!\!\!\!\!\!\frac{\rho_{i_0(p)}}{2}\nonumber\\
&\leq&\Bl1+\frac{\epsilon}{\deg X_{\alpha}}\Br\bl\sum_{ p\in \ti \Lam\lalp}\sum_{i\in \II_{p}^{\basic}} \delta(\ti s_i,p)\rho_{i}
-\sum_{p\in {\ti S\lalp}\cap\ti \Lam\lalp}\rho_{i_0(p)}\br\\
&=&\Bl1+\frac{\epsilon}{\deg X_{\alpha}}\Br\bl\sum_{i\in \II\lalp^{\basic}} \delta\lalp(\ti s_i)\rho_{i}
-\sum_{p\in {\ti S\lalp}\cap\ti \Lam\lalp}\rho_{i_0(p)}\br\\
&=&\frac{E\lalp(\rho)}{2}-\sum_{p\in {\ti S\lalp}\cap\ti \Lam\lalp}\frac{\rho_{i_0(p)}}{2}\,.
\end{eqnarray*}
This completes the proof the theorem.
\end{proof}

\section{Stability of weighted pointed nodal curve}\label{proof}

We prove Theorem \ref{main} in this section. By Proposition \ref{stair-case},
it suffices to verify the positivity  of the $\ba$-$\lam$-weight $\omega_\ba(\lam)$ of
$\mathrm{Chow}( X,\mathbf{x})\in \Xi$ for any staircase
$\lam$.   Let $\bs$ be a diagonalizing basis of $\lam$:
$$
\lam( t) :=\mathrm{diag}[ t^{\rho _{0}},\cdots ,t^{\rho
_{m}}] \cdot t^{-\rho_{ave}}\ ,\quad\text{with}\quad \rho _{0}\geq \rho _{1}\geq \cdots \geq \rho _{m}=0.
$$
The $\ba$-$\lam$-weight of $\mathrm{Chow}( X,\mathbf{x})$ is the sum of the contributions
from $\mathrm{Div}^{d,d}[(\PP W\dual) ^{2}]$ and $(\PP W) ^{n}$.
By Proposition \ref{chow-wt}, the contribution from $\mathrm{Div}^{d,d}[(\PP W\dual) ^{2}]$
is $\omega(\lam)$.

For the contribution from $(\PP W) ^{n}$, we introduce subspaces
\beq\label{W-i}
W_i=W_i(\lam):=\{v\in W\mid s_i(v)=\cdots=s_m(v)=0\}\subset W=H^0(\sO_X(1))^\vee.
\eeq
They form a strictly increasing filtration of $W$. Also, for any closed subscheme $\Sigma\sub X$, we denote by
\beq\label{W-sigma}
W_\Sigma:=\{v\in W\mid s(v)=0\text{ for all } s\in H^0(\sO_X(1)\otimes \sI_\Sigma)\}\sub W
\eeq
the {\em linear subspace spanned by  $\Sigma\sub X$}.
For instance, for a marked point $x_i$, $W_{x_i}$ is the line in $W$ spanned by $x_i\in\PP W$; for any $i$ and
\beq\label{defZ}
Z_i=\{s_i=\cdots=s_m=0\}\sub X,
\eeq
we let $W_i=W_{Z_i}$.

By \cite[Prop 4.3]{MFK}, the $\ba$-$\lambda $-weight of
$\mathbf{x}=( x_{1},\cdots ,x_{n}) \in (\PP W) ^{n}$ is
\begin{equation}\label{M_A}
\mu_{\mathbf{a}}(\lam) :=\sum_{j=1}^{n}a_{j}\Bigl( \frac{\sum_{i=0}^{m}\rho _{i}}{%
m+1}+\sum_{i=0}^{m-1}( \rho _{i+1}-\rho _{i})\dim (W_{x_{j}}\cap W_{i+1}(\lam))\Bigr ).
\end{equation}%
%Before we move on one should notice that, similar to $\frak{e}(\lam,\rho)$, $\mu_{\mathbf{a}}( \rho;\ s)$ is also linear in $\rho$.
($\mu_\ba(\lam)$ implicitly depends on $\rho_i$, which we fix for the moment.)
Therefore, the $\ba$-$\lambda $-weight $\omega_\ba(\lam)$
of $\mathrm{Chow}( X,\bx)\in \Xi$
is
\beq\label{omega-mu}
\omega_\ba(\lam)= \omega(\lam)+\mu_{\mathbf{a}}(\lam)\ .
\eeq

We now argue that for the staircase $\lam'$ constructed from $\lam$ by applying
Proposition \ref{stair-case}, we have
\beq\label{om-om}
\omega_\ba(\lam)\geq \omega_\ba(\lam').
\eeq
Indeed, since $\omega(\lam)\geq\omega(\lam')$, if suffices to show that
$\mu_{\ba}(\lam)\geq \mu_{\ba}(\lam')$.  To see that, we first notice that
\beq\label{qqq}
\dim (W_{x_j}\cap W_{i+1}(\lam))=\#\bl x_i\cap \supp (\sO_X(1)/\sE(\lam)_{i+1})\br.
\eeq
(Here $\sE(\lam)_i=( s_i, s_{i+1},\ldots, s_m)\sub \sO_{X}(1)$.)
On the other hand, by the proof of Proposition \ref{stair-case}, we conclude
$$\supp(\sO_X(1)/\sE(\lam)_i)\sub \supp(\sO_X(1)/\sE(\lam')_i).
$$
%where $\sE(\lam')_i=( s_i', s'_{i+1},\ldots, s'_m)\sub \sE(\lam)_i\sub \sO_{X}(1)$.
This together with \eqref{qqq} proves
$$\dim (W_{x_j}\cap W_{i+1}(\lam))
%=|x_j\cap \supp (\sO_X(1)/\sE(\lam)_{i+1})|\\
%&\leq& |x_j\cap \supp (\sO_X(1)/\sE(\lam')_{i+1})|
\le \dim (W_{x_j}\cap W_{i+1}(\lam')).
$$
The inequality $\mu_{\ba}(\lam)\geq \mu_{\ba}(\lam')$ then follows from $\rho_i\geq \rho_{i+1}$.
Therefore, to prove Theorem \ref{main}, we suffices to show $\omega_\ba(\lam)>0$ for all staircase 1-PS's $\lam$.
From now on we assume $\lam$ is a {\em staircase}.

Before we proceed, we collect a few boundness results that are needed to pass from the estimates on single component in Section \ref{main-estimate} to the entire curve.

\begin{lemm}\label{deg-X-a}
Suppose  $(X,\sO_X(1),\bx,\ba)$ is  slope {\em stable} and  $\chi_{\ba}(X)>0$ (cf. Theorem \ref{main}).
Then there are positive constants
$M$ and $C$ depending only on $g$,$n$ and $\ba\in \QQ^n$ such that whenever $\deg X> M$ we have
$\deg Y\geq C \deg X$ for any connected proper subcurve $Y\sub X$.

In case $(X,\sO_X(1),\bx,\ba)$ is  slope semi-stable, then the same conclusion $\deg Y\ge C\deg X$ holds except when
$Y$ is a  line, $Y\cap \bx=\emptyset$, $|Y\cap Y^\complement|=2$, and the inequality \eqref{basic} is an equality.
\end{lemm}

\begin{proof}
We continue to denote $\ell_Y=|Y\cap Y^\complement|$. Let $g_Y$ be the arithmetic genus of $Y$. Suppose $2g_Y+\ell_Y\geq 3$, then
since $a_i\geq 0$, the %$g_Y -1+\sum_{x_{j}\in Y}a_j/2+\ell_Y/2\geq 1/2$.
inequality \eqref{basic} implies
%$$\red \frac{\deg Y+ |\ba|_Y/2}{g_Y -1+ |\ba|_Y/2+\ell_Y/2}\geq
%\frac{\deg X+ |\ba|/2}{g-1+|\ba|/2}
%-\frac{\ell_Y}{2(g_Y -1)+|\ba|_Y+\ell_Y}
%$$
$$\frac{\deg Y+ \sum_{x_{j}\in Y}a_{j}/2}{g_Y -1+\sum_{x_{j}\in Y}a_j/2+\ell_Y/2}\geq
\frac{\deg X+\sum_{j=1}^{n}a_j/2}{g-1+\sum_{j=1}^{n}a_j/2}
-\frac{\ell_Y}{2(g_Y -1)+\sum_{x_{j}\in Y}a_j+\ell_Y}
$$
$$\qquad \geq\frac{\deg X+\sum_{j=1}^{n}a_j/2}{g-1+\sum_{j=1}^{n}a_j/2}-3.
$$
Denoting
$\chi_\ba=\chi_\ba(X)$, this inequality implies
\begin{eqnarray*}
\deg Y\geq\Bl\frac{\deg X}{2\chi_{\ba}}-6\Br-\frac{n}{2}.
\end{eqnarray*}
Therefore, if we choose $M'=4\chi_\ba\cdot(6+n/2)$, choose $C'=1/4\chi_\ba$ and require $\deg X\geq M'$, we obtain
$$\deg Y\geq \deg X/4\chi=C'\cdot\deg X.
$$

Now suppose $2g_Y+\ell_Y\leq 2$. Since $Y\sub X$ is a proper connected subcurve, $\ell_Y\geq 1$ and $g_Y\geq 0$.
Thus $g_Y=0$ and $\ell_Y\leq 2 $. In this case, \eqref{basic} becomes
\beq
\Big\vert \deg Y+\sum_{x_{j}\in Y}\frac{a_{j}}{2}-\frac{\deg
X+\sum_{j=1}^{n}a_j/2}{g-1+\sum_{j=1}^{n}a_j/2}\cdot \bigl( -1
+\sum_{x_{j}\in Y}a_j/2+\ell_Y/2\bigr)
\Big\vert %\text{(resp.}\leq \text{)}
\leq1.
\eeq
Let $A:= -1+\sum_{x_{j}\in Y}a_j/2+\ell_Y/2$. In case $A\leq 0$, then we have
$\deg Y=1$, $\sum_{x_i\in Y}a_i=0$ and $\ell_Y=2$. This is precisely the second case in the statement of the Lemma.

In case
$A>0$. then
$$\deg Y>\frac{\deg
X+\sum_{j=1}^{n}a_j/2}{\chi_\ba}\cdot 2A-1- \sum_{x_{j}\in Y}\frac{a_j}{2}>\frac{\deg X}{\chi_\ba}A,
$$
provided $\deg X>(2+n)\chi_{\ba}/2A $. To obtain a universal constant, we intorduce
$$C_k:=\min_{I\sub \{1,\cdots,n\}}\{\sum_{i\in I}a_i/2+k/2\mid \sum_{i\in I}a_i/2+k/2>0\},
$$
and define $ C_{\min}:=\min_{k\geq0}\{C_k\}$. Clearly, $C_{\text{min}}>0$.
By our construction, $A\geq C_{\min}$ when $A>0$.
We choose $C\dpri=\min\{C_{\min}/\chi_{\ba}, 1/2\chi_{\ba}\}$, and choose
$$M\dpri:=\max\{6g+n/2-6,\  \chi_\ba(2+n)/2C_{\min}\}.
$$
Our discussion shows that the statement (1) in the Lemma holds in the case under study with this
choices of $M\dpri$ and $C\dpri$. Finally, we let $C=\min\{ C',C\dpri\}$ and $M=\max\{ M',M\dpri\}$.
The Lemma holds with this choice of $C$ and $M$ in all cases. This completes the proof of the Lemma.
\end{proof}

%A quick consequence is the following boundness result.

\begin{coro}\label{bdd-link} Let $M$ be as in  Lemma \ref{deg-X-a} and suppose $\deg X\geq M$. Then
the number of irreducible components of $X$ is at most $9g+4n-5$;  the number of nodes of $X$ is
at most $10g+4n-5$.
\end{coro}
\begin{proof}
We divide the irreducible components of $X$ into three categories: the first (resp. second; resp. third)
consists of irreducible components of $X\in\PP W$ that are
not lines (resp. are lines that contains marked points; resp. are lines in $\PP W$ that contains no marked points).

Applying the previous Lemma, for $X\lalp$ in the first category,
$\deg X\lalp >\deg X/4\chi_\ba$; thus the first category contains no more than $4\chi_\ba$ elements.
Since every component in the second category contains at least one marked point, there are at most
$n$ elements in this category.

We now bound the element in the third category. We let $B_1$ be a maximal subcollection of the third category
so that $X-\cup_{\alpha \in B_1} X\lalp$ remains connected. We let $Y=\overline{X-\cup_{\alpha \in B_1} X\lalp}$.
Then $g(Y)=g-|B_1|\ge 0$ since $Y$ is connected. We let $B_2$ be the
complement of $B_1$ in the third category. By Lemma \ref{deg-X-a}, $X\lalp$ and $X_{\alpha'}$ are disjoint
for $\alpha\ne\alpha'\in B_2$. Thus if we let $Y_1,\cdots, Y_k$ be the connected components
of $\overline{Y-\cup_{\alpha\in B_2} X\lalp}$, then $|B_2|\le k-1$.
On the other hand, applying Lemma \ref{deg-X-a} we know each $Y_i$ has degree at least $\deg X/4\chi_\ba$.
Thus $k\leq 4\chi_\ba$.

Combined, the total number $r$ of irreducible components of $X$ is bounded by
$$r\, \le \, 4\chi_\ba+n+(g+4\chi_\ba-1)\leq 9g+4n-5.
$$

Next we bound the total nodes of $X$. We first pick a maximal set $A_1\sub X_{\text{node}}$
so that $X-A_1$ is connected; then $|A_1|\leq g$. We let $A_2=X_{\text{node}}-A_1$.
Then $|A_2|$ is less than the number of irreducible components of $X$. By the bound we just derived,
we obtain $|X_{\text{node}}|\leq g+(9g+4n-5)$.
This proves the Corollary.
\end{proof}
We need the next consequence in our proof of Theorem \ref{main} to replace inequality \eqref{basic} by \eqref{basic1}.
\begin{coro}\label{basic=basic1}
Let $X\subset \PP H^0(\sO_X(1))\dual$ be a connected % nondegenerate,
nodal curve of arithmetic genus $g>0$. Suppose $\chi_\ba(X)>0$, then there is an $M$ depending on $g$ such that whenever $\deg X\geq M$,
$X$ satisfies \eqref{basic} for any subcurve $Y\sub X$ if and only if it
satisfies \eqref{basic1} for any subcurve $Y\sub X$ ,
\end{coro}
\begin{proof}
First, we claim that there is an $M$ depending only on $g,n$ such that for $\deg X\geq M$, we have $h^1(\sO_X(1)|_Y)=0$
for any subcurve $Y\sub X$ (not necessary proper)
provided $X$ satisfies either \eqref{basic} or \eqref{basic1} for any proper subcurve $Y\subsetneq X$.
To to that, let $Y\sub X$ be any subcurve with  arithmetic genus $g_Y$. Suppose $h^1(\sO_X(1)|_Y)\geq 1$, then by the vanishing theorem we must have $g_Y>1$ and  $2<\deg Y<2g_Y-1\leq 2g-1$.

Suppose $X$ satisfies \eqref{basic}, then the claim is
an easy consequence of Lemma \ref{deg-X-a} by letting $M=\max\{(2g+\nu)/C,2g+\nu\}$ with $C$ being chosen in Lemma \ref{deg-X-a}
and $\nu$ being the total number of nodes of $X$. Suppose
$X$ satisfies \eqref{basic1}, then the claim was proved in \cite[Proposition 1.0.7]{Gie}.

Finally, when $h^1(\sO_X(1)|_Y)=0$ for any subcurve $Y\sub X$ the
equivalence of \eqref{basic} and \eqref{basic1} follows from
an argument parallel to the one given in \cite[Proposition 3.1]{Cap} and \cite[Proposition 1.0.7]{Gie}.
So we omit it.
\end{proof}

Before we state our key estimate of this section, we first notice that  $\lam$ being a staircase 1-PS (cf. Definition \ref{staircase}) implies
that $\bigcup_{\alpha=1}^r\II\lalp=\{0,\cdots,m\}$, where $\II\lalp$ is the index set of the component $X\lalp$ defined in \eqref{I-lalp}.
We define the shifted weights $\{\hat \rho_i\}$ by
%This allows us define a map  $\hat{}:\RR^{m+1}_+\rightarrow\RR^{m+1}_+$ by letting
%$\rho=\{\rho_i\}_{i=0}^m\mapsto \hat{\rho}=\{\hat{\rho}_i\}_{i=0}^m$ with
\beq\label{hat-rho}
\hat{\rho}_i:=\min_{\alpha}\{\rho_i-\rho_{\hbar\lalp}\mid i\in \II\lalp\}\geq 0.
\eeq
%It is clear that the map $\,\hat{}\,$ is well defined and it depends only on the diagonalizing basis $\{s_i\}$ of  $\lam$.
($\hat\rho_i$ may not be monotone. And $\hat \rho_i$ are only defined for staircase 1-PS.) We state out main estimate.

\begin{prop}\label{new-bdd}
Suppose $(X,\sO_X(1),\bx,\ba)$ is slope stable (cf. \eqref{basic1}), and $\omega_{X}(\bax)$ is ample.
Then for any $0<\epsilon<1$ there exists an $M$  depending only on $\chi_\ba(X)$ (cf. Theorem \ref{main}) and $\epsilon$ such that whenever $\deg X>M$, then for any staircase 1-PS $\lam$  we have
\beq\label{E_X-new}
\frac{E_X(\lam,\rho)}{2}\leq\sum_{i=0}^m \rho_{i}-\!\!\!\sum_{q\in \ti S_{\reg}}\frac{\hat\rho_{i_0(q)}}{2}
+\sum_{\alpha=1}^r\Bl \deg X_{\alpha}+\frac{\ell\lalp}{2}
-m\lalp-1\Br\cdot\rho_{\hbar\lalp}+\frac{2C\upmo\epsilon}{ m+1}\sum_{i=0}^m\hat{\rho}_{i}\ .
\eeq
where
%\beq\label{S-reg}
$\ti S_{\reg}:=\bigcup_{\alpha=1}^r (\pi\upmo(\bx)\cap \ti X\lalp\cap\ti\Lam)$
is the support of weighted points and $C>0$ is the constant fixed in Lemma \ref{deg-X-a}.
\end{prop}

\begin{proof}
By   the definition of $E\lalp(\rho)$ (cf.\eqref{E-alp}), $E_X(\lam,\rho)=\sum_{\alpha=1}^r E\lalp(\rho)$
is linear in $\rho=(\rho_i)$. By linear programming,  \eqref{E_X-new} holds on
$$ \RR^{m+1}_+:=\{(\rho_0,\cdots,\rho_m)\in \RR^{m+1}\mid \rho_0\geq\rho_1\geq\cdots\geq\rho_m=0\,\}
$$
if and only if it holds on every edge of $\RR^{m+1}_+$; that is on
\beq\label{10}
\rho=(\overset{m_0}{\overbrace{1,\cdots,1}},0,\cdots,0)
\eeq
for every $0< m_0< m$.

We now fix a $0< m_0<m$. By  possibly reindexing the irreducible components of $X$, we can assume that for an
$\bar r\leq r$, $\hbar_1\le\cdots\le\hbar_{\bar r}< m_0\le \hbar_{\bar r+1}\le\cdots\le \hbar_r$. In other words,
\beq\label{rrr}
 \rho_{\hbar_1}=\cdots=\rho_{\hbar_{\bar r}}=1,\and \rho_{\hbar_{\bar r+1}}=\cdots=\rho_{\hbar_r}=0.
\eeq
We let $Y:=\displaystyle\bigcup_{\alpha>\bar r} X\lalp$; let its complement
$Y^{\complement}=\displaystyle\bigcup_{\alpha\le \bar r} X\lalp$.

We claim that $Y^\complement$ is the maximal subcurve
of $X$ contained in the linear subspace $\PP W_{m_0}$ (cf. \eqref{W-i}). By definition, for any $\alpha$, $\hbar\lalp$ is the largest index
$0<i\leq m$ of which $s_i|_{X\lalp}\ne 0$. On the other hand, because $\PP W_{m_0}=\{s_{m_0}=\cdots=s_m=0\}$,
$X\lalp\sub\PP W_{m_0}$ if
and only if $s_{i}|_{X\lalp}=0$ for all $i\ge m_0$, which is equivalent to $\hbar\lalp< m_0$. This proves the claim.

Let $X\lalp$ be a component in $Y^{\complement}$. Since
%all $\rho_i\leq 1$ and
$\rho_{\hbar\lalp}=1$, $\rho_i=1$ for $i\in \II\lalp$. Using the explicit expression of
$E\lalp(\lam,\rho)$, we obtain $ E\lalp(\lam,\rho)=2\deg X\lalp$.
Thus
$$\sum_{\alpha\le \bar{r}} E\lalp(\lam,\rho) =\sum_{\alpha\leq \bar r} 2\deg X\lalp= 2\deg Y^{\complement}\ .
$$

We now look at $Y$. Following \eqref{link-node} and \eqref{N-Y}, $\ell_Y:=|Y\cap Y^\complement|$, and % let us introduce $\hat{\rho}\lalp:=\rho_{\hbar\lalp}$ and define
%$$L_Y:=Y\cap Y^{\complement},\quad
$\ti L_Y:=\pi^{-1}(Y\cap Y^{\complement})\cap \ti Y$. %\quad\text{and}\quad\ell_Y:=|L_Y|\ .$$
We claim that
%\beq\label{LL}
$\ti L_Y\sub \ti \Lambda_Y:=\bigcup_{\alpha>\bar r} \ti \Lam\lalp$.
Indeed, for any
$\alpha> \bar r$, there is an $i\ge m_0$ so that $s_i|_{X\lalp}\ne 0$. However for any $\beta\le \bar r$,
$i\ge m_0$ implies $s_i|_{X\lbe}=0$. Thus $s_i|_{X\lalp\cap X_\beta}=0$, and consequently,
$\pi\upmo(X\lalp\cap X\lbe)\cap \ti X\lalp\sub \ti \Lam\lalp$.
Summing over all $\alpha>\bar r$ and $\beta\le \bar r$, we obtain $\ti L_Y\sub \ti\Lam_Y$.
As a consequence,
\beq\label{sum-L}
%\sum_{\alpha=\bar{r}+1}^r
\sum_{p\in  \ti L_Y}\rho_{i_0(p)}=\ell_Y\ .
\eeq

To simplify the notation, in the remaining part of this section, we will abbreviate
$$\sum_{p\in A}\rho_{i_0(p)}:=\sum_{p\in A\cap \ti\Lam}\rho_{i_0(p)},
$$
with the understanding that for any $A\sub \ti X$, the
summation $\sum_{p\in A}$ only sums over $p\in A\cap \ti \Lam$.

\begin{subl}\label{E-Y}
Let the notation be as before. Then
$$
\sum_{\alpha>\bar r}\frac{ E\lalp(\lam,\rho)}{2}-\frac{\ell_Y}{2}
\leq \Bl1+\frac{\epsilon}{\deg X}\Br\Bl %\sum_{\alpha>\bar r}
\sum_{\substack{ \alpha>\bar r\\
i\in \II\lalp^{\basic}}} \delta\lalp(\ti s_i)\rho_{i}
- \sum_{p\in  \ti L_Y}\rho_{i_0(p)}-
\sum_{p\in \ti N_Y\setminus \ti L_Y} \frac{{\rho}_{i_0(p)}}{2}\Br
-\sum_{\pi(p)\in \bx \cap Y}\frac{{\rho}_{i_0(p)}}{2}
$$
\end{subl}

\begin{proof}
We let $X\lalp\sub Y$ be an irreducible component, then
$\alpha >\bar{r}$ and $\rho_{\hbar\lalp}=0$.
Since $(X,\sO_X(1),\bx,\ba)$ is slope semi-stable, by Corollary \ref{deg-X-a},  there are positive
constants $C$ and $N$ such that whenever $\deg X\geq N$, either
$\deg X\lalp >C\deg X$ or $\deg X\lalp=1$.
If $\deg X\lalp >C\deg X$, from the definition of $E\lalp(\rho)$ (cf. \eqref{E-alp} ) and $\rho_{\hbar\lalp}=0$,
we have
\beq\label{E-1}
\frac{ E\lalp(\lam,\rho))}{2} \,\le\, \Bl1+\frac{C^{-1}\cdot\epsilon
}{\deg X}\Br\sum_{i\in \II\lalp^{\basic}}\delta\lalp(\ti s_i)
\rho_{i}
-\Bl\frac{1}{2}+\frac{C^{-1}\cdot\epsilon}{\deg X}\Br\sum_{p\in  \ti S\lalp} {\rho}_{i_0(p)}.
\eeq
If $\deg X\lalp=1$, \eqref{E-1} remain holds since by \eqref{E-alp-0} and Definition \ref{prim-ind}, we have
\begin{eqnarray*}
\frac{ E\lalp(\lam,\rho))}{2} \,
&=&\delta\lalp(\ti s_{\ind\lalp\upmo(0)})\cdot\frac{\bar{\rho}_{\ind\lalp\upmo(0)}}{2}+ \rho_{\hbar\lalp}
=\sum_{i\in \II\lalp^{\basic}}\delta\lalp(\ti s_i)\cdot\frac{\rho_i}{2}\\
&\leq&\, \Bl1+\frac{C^{-1}\cdot\epsilon
}{\deg X}\Br\sum_{i\in \II\lalp^{\basic}}\delta\lalp(\ti s_i)
\rho_{i}
-\Bl\frac{1}{2}+\frac{C^{-1}\cdot\epsilon}{\deg X}\Br\sum_{p\in  \ti S\lalp} {\rho}_{i_0(p)}.
\end{eqnarray*}

Next we split
$$
-\sum_{p\in  \ti S\lalp}{\rho}_{i_0(p)}=
- \sum_{p\in \pi\upmo(\bx)\cap \ti X\lalp} {\rho}_{i_0(p)}
-\sum_{p\in \ti N\lalp\setminus\ti L_Y} {\rho}_{i_0(p)}-\sum_{p\in\ti L_Y\cap\ti X\lalp}{\rho}_{i_0(p)}.
$$
Then using $\rho_i\geq 0$, we get
\begin{eqnarray*}
-\Bigl(\frac{1}{2}+\frac{C^{-1}\cdot\epsilon}{\deg X}\Bigr)
\sum_{p\in  \ti S\lalp }{\rho}_{i_0(p)}
&\leq& -\!\!\!\!\! \sum_{p\in \pi\upmo(\bx)\cap \ti X\lalp}\!\!\!\!\! \frac{{\rho}_{i_0(p)}}{2}
-\Bigl(\frac 1 2+\frac{C^{-1}\cdot\epsilon}{\deg X}\Bigr)
\sum_{p\in \ti N\lalp\setminus\ti L_Y} {\rho}_{i_0(p)}-\\
&&- \Bigl(1+\frac{C^{-1}\cdot\epsilon}{\deg X}\Bigr) \sum_{p\in \ti L_Y\cap\ti X\lalp} {\rho}_{i_0(p)}
+\sum_{p\in \ti L_Y\cap\ti X\lalp}\frac{ {\rho}_{i_0(p)}}{2}\\
&\leq& -\!\!\!\!\! \sum_{p\in \pi\upmo(\bx)\cap \ti X\lalp}\!\!\!\!\! \frac{{\rho}_{i_0(p)}}{2}
-\Bigl(1+\frac{C^{-1}\cdot\epsilon}{\deg X}\Bigr)
\sum_{p\in \ti N\lalp\setminus\ti L_Y}\frac{\rho_{i_0(p)}}{2}-\\
&&- \Bigl(1+\frac{C^{-1}\cdot\epsilon}{\deg X}\Bigr) \sum_{p\in \ti L_Y\cap\ti X\lalp} {\rho}_{i_0(p)}
+\sum_{p\in \ti L_Y\cap\ti X\lalp}\frac{ {\rho}_{i_0(p)}}{2}\ .
\end{eqnarray*}
Putting together, we obtain
\begin{eqnarray*}
\frac{ E\lalp(\lam,\rho))}{2}
&\leq&
\Bigl(1+\frac{C^{-1}\cdot \epsilon}{\deg X}\Bigr)\Bigl(\sum_{i\in \II\lalp^{\basic}} \delta\lalp(\ti s_i)\rho_{i}
-\sum_{p\in \ti L_Y\cap \ti X\lalp}\rho_{i_0(p)}-\\
&&-\sum_{p\in \ti N\lalp\setminus \ti L_Y}\frac{{\rho}_{i_0(p)}}{2}\Bigr)
+\sum_{p\in \ti L_Y\cap \ti X\lalp}\frac{\rho_{i_0(p)}}{2}-\sum_{p\in \pi\upmo(\bx)\cap\ti X\lalp}\frac{{\rho}_{i_0(p)}}{2}\ .
\end{eqnarray*}
Summing over $\alpha$ and applying \eqref{sum-L} prove the Lemma.
\end{proof}

%
%Recall $W_{Y\cap Y^{\complement}}$ is the linear subspace in $W$ spanned by $Y\cap \Ycomp$.
%We claim that if the slope semi-stable  inequality holds and $\deg X\ge N$ for an $N$ to be specified later,
%$\dim W_{Y\cap \Ycomp}=|Y\cap \Ycomp|$. ....

%Following the convention, we introduce $\ti N_Y=\pi\upmo(X_{\text{node}})\cap \ti Y$.
%:=W_Y\cap W_{Y^{\complement}}$.\black
The following inequality is crucial for the proof of the Proposition. %Theorem \ref{main}.

\begin{lemm}\label{dim-ineq}
For $1\leq k\leq m_0$, we have % the inequality
$$
 \sum_{\substack{\alpha> \bar{r}\\ i\in \II\lalp^{\basic}\cap[0,k)}}\delta\lalp(\ti s_i)\rho_{i}
- \sum_{\substack{ p\in \ti L_Y\\ i_0(p)<k }} \rho_{i_0(p)}- \sum_{\substack{p\in \ti N_Y\setminus \ti L_Y\\ i_0(p)<k } }
\frac{{\rho}_{i_0(p)}}{2}\
\leq \ \dim W_{Y}\cap W_{k}-\dim W_{Y\cap Y^{\complement}}\cap W_k,
$$
where $W_{Y\cap Y^{\complement}}$ is the linear subspace in $W$ spanned by $Y\cap \Ycomp$.
\end{lemm}

\begin{proof}
We prove the Lemma by induction on $k$. When $k=0$, then both sides of
the inequality are zero, and the inequality follows. Suppose the Lemma holds for
a $0\leq k < m_0$. Then the Lemma holds for $k+1$ if for the expressions
$$A_{k,1}:= \sum_{\substack{\alpha> \bar r, \\ k\in \II\lalp\prim}} \delta\lalp(\ti s_k) \rho_k,\quad
A_{k,2}:=\sum_{ \substack{p\in \ti L_Y,\\ i_0(p)=k}} \rho_{i_0(p)},\quad
A_{k,3}:=\sum_{\substack{p\in \ti N_Y\setminus \ti L_Y,\\ i_0(p)=k}} \frac{\rho_{i_0(p)}}{2}
$$
and
$$B_{k,1}:=\dim W_Y\cap W_{k+1}-\dim W_Y\cap W_k,\quad
B_{k,2}:=\dim W_{Y\cap Y^\complement}\cap W_{k+1}-\dim W_{Y\cap Y^\complement}\cap W_k,
$$
the following inequality
holds
\beq\label{A-B}
A_{k,1}-A_{k,2}-A_{k,3}\leq B_{k,1}-B_{k,2}.
\eeq

To study the left hand side of \eqref{A-B}, we introduce the set
\beq\label{R_k}
R_k=\{ p\in\ti Y\mid k\in \II_p\prim\}.
\eeq
By Proposition \ref{vir-ind} and \ref{vir-ind-2}, $R_k$ can take three possibilities
according to
\beq\label{value}
\sum_{\alpha>\bar r,\, k\in\II\lalp\prim} \delta\lalp(\ti s_k)
\eeq
taking values $0$, $1$ or $\ge 2$. Notice that if $A_{k,1}=0$, then $A_{k,1}-A_{k,2}-A_{k,3}\leq 0$. The Lemma holds
trivially in this case since the right hand side of \eqref{A-B} is non-negative.
So from now on, we will assume that $A_{k,1}\geq 1$, in particular, \eqref{value} is positive.

We first observe that since $\dim W_{k+1}-\dim W_{k}=1$, both $B_{k,1}$ and $B_{k,2}$ can only take values $0$ or $1$.
We now investigate the case when $B_{k,2}=1$.
\begin{clai}\label{B}
Suppose \eqref{value} is positive and  $B_{k,2}=1$. Then there is a $p\in R_k$ (cf.\eqref{R_k})such that
$i_0(p)=k$ and
\beq\label{q-1}
q=\pi(p)\in Y\cap \Ycomp\cap (\PP W_{k+1}-\PP W_k).
\eeq
\end{clai}
\begin{proof} %[Proof of Claim \ref{B}]
Suppose \eqref{value} is positive then there is a $p\in\inc(\ti s_{k})\cap \ti X\lalp$ with $\alpha>\bar{r}$ and
$k\in \II\prim\lalp$. Let $Z_k$ be the subscheme defined in \eqref{defZ} and $W_{Z_k+q}\supsetneq W_k$  be defined in \eqref{W-sigma}.
Then $W_{k+1}=W_{Z_k+q}$, since $\dim W_{k+1}=\dim W_k+1$.
Suppose $q=\pi(p)\not\in Y\cap Y\comp$ and $k\in \II\prim\lalp$ then by applying the argument parallel to Proposition \ref{vir-ind} and
\ref{vir-ind-2}, we deduce
\beq\label{W-B}
W_{Z_k+q}+W_{Y\cap Y\comp}\supsetneq W_{Z_k}+W_{Y\cap Y\comp}\ .
\eeq
 On the other hand, $B_{k,2}=1$ implies that
\begin{eqnarray*}
 \dim (W_{k}+W_{Y\cap Y\comp})
 &=&\dim W_k+\dim W_{Y\cap Y\comp}-\dim W_{k}\cap W_{Y\cap Y\comp}\\
 &=&\dim W_{k+1}+\dim W_{Y\cap Y\comp}-\dim W_{k+1}\cap W_{Y\cap Y\comp}\\
 &=&\dim (W_{k+1}+W_{Y\cap Y\comp})\ ,
\end{eqnarray*}
which means
$W_{k}+W_{Y\cap Y\comp}=W_{k+1}+W_{Y\cap Y\comp}$  contradicting to \eqref{W-B}. So we must have
$q\in Y\cap Y\comp$.

By definition, $q\in \PP W_{k+1}$ (cf. \eqref{W-i}) implies that $s_i(q)=0$ for $i\geq k+1$; $q\not\in\PP W_k$ implies
that not all $s_i(q)$, $k\le i\le m$, are zero. Combined, we have $s_k(q)\ne 0$. This implies $i_0(q)=k$.
As an easy consequence, this shows that $B_{k,2}=1$ forces $W_Y\cap W_{k+1}\ne W_Y\cap W_{k}$,
and hence $B_{k,1}=1$. In particular, the right hand side of \eqref{A-B} is non-negative. This proves the Claim.
\end{proof}

We complete our proof of Lemma \eqref{dim-ineq}.
When \eqref{value} takes value $1$, then $R_k$ consists of a single point, say $p\in\ti Y$.
In case $\pi(p)\in Y$ is a smooth point of $X$, $A_{k,1}=1$ and $A_{k,2}=A_{k,3}=0$.
We claim that $B_{k,1}=1$ and $B_{k,2}=0$.
Indeed, if $B_{k,1}=0$, then $\PP W_Y\cap \PP W_{k+1}=\PP W_Y\cap \PP W_k$, which is the same as
$Y\cap (s_k=\cdots=s_m=0)=Y\cap (s_{k+1}=\cdots=s_m=0)$ as subschemes of $Y$. But this contradicts to
$\sum_{\alpha> \bar r,\, k\in\II\lalp\prim} \delta\lalp(\ti s_k)=1$. Thus $B_{k,1}=1$. On the other hand,
if $B_{k,2}=1$, then Claim \ref{B} shows that $R_k\cap \ti Y$ contains an element in $\ti L_Y$,
contracting to our assumption that $R_k=\{p\}$ lies over a smooth point of $X$.

In case $p\in \ti L_Y$, then the previous paragraph shows that $A_{k,1}=B_{k,1}=1$, $A_{k,3}=0$.
For the values of $A_{k,2}$ and $B_{k,2}$, when $i_0(p)=k$,
then both $A_{k,2}=B_{k,2}=1$; when $i_0(p)\ne k$, then both $A_{k,2}=B_{k,2}=0$. Therefore,
\eqref{A-B} holds.

The last case is when $p\in \ti N_Y-\ti L_Y$. In this case,
since the point $p'$ in $\ti Y\cap \pi\upmo(\pi(p))$ other than $p$ is not contained in $R_k$,
either $i_0(p)\ne k$ or $i_0(p)=i_0(p')=k$ and $k\not\in \II_{p'}\prim$. In both cases, $A_{k,1}=B_{k,1}=1$, and $B_{k,2}=0$;
the inequality \eqref{A-B} holds.

Lastly, when \eqref{value} is bigger than $1$, by Proposition \ref{vir-ind} and \ref{vir-ind-2},
either $R_k=\{p_-,p_+\}$ such that $\pi(p_-)=\pi(p_+)$ is a node of $Y$, i.e. $p_\pm\in \ti N_Y$,
and $i_0(p_-)=i_0(p_+)=k$, or
$R_k=\{p_1,\cdots,p_l\}$ %; p_{1+},p_{1-}',\cdots, p_{t+}',p_{t-}'\}
so that $i_0(p_i)=k$ and $\{\pi(p_i)\}_{1\le i\le s}$ are distinct nodes of $X$.
In case $R_k=\{p_-,p_+\}$; since $p_\pm\in \ti N_Y\setminus\ti L_Y$, $A_{k,1}=2$, $A_{k,2}=B_{k,2}=0$, and $A_{k,3}=B_{k,1}=1$.
The inequality \eqref{A-B} holds in this case.

The other case is when $R_k=\{p_1,\cdots,p_l\}$. By reindexing, we may
assume $p_1,\cdots, p_{l_1}$ are in $\ti N_Y\setminus \ti L_Y$
 and $p_{l_1+1},\cdots,p_l$ are in $\ti L_Y$.
We let $p_i'\in \ti Y$ be such that $\pi\upmo(\pi(p_i))=\{p_i,p_i'\}$ for $i\le l_1$. Then $i_0(p_i')=k$ as well,
but $k\not\in \II_{p'_i}\prim$ because of Proposition \ref{vir-ind} and \ref{vir-ind-2}.
This in particular implies that the interior linking nodes $\ti N_Y\setminus \ti L_Y$ contributes once in $A_{k,1}$ but
twice in $A_{k,3}$ (, e.g. only $\rho_{i_0(p_i)}$ appears in $A_{k,1}$, but both $\rho_{i_0(p_i)}$ and $\rho_{i_0(p'_i)}$
appear in $A_{k,3}$). Therefore, $A_{k,1}=l$; $A_{k,2}=l-l_1$, and $A_{k,3}=2l_1/2=l_1$. Hence the left hand side of \eqref{A-B} is $0$.
This proves \eqref{A-B} in this case; hence for all cases. This proves the Lemma.
\end{proof}

We continue our proof of Proposition \ref{new-bdd}.
We  apply Lemma \ref{E-Y} and  Lemma \ref{dim-ineq} with $k=m_0$. Noticing $\rho_{i_0(p)}=0$ for $i_0(p)>m_0$, we obtain
\begin{eqnarray}\label{E_Y}
&&\frac{ E_{Y}(\lam,\rho)}{2}-\frac{\ell_Y}{2}=\sum_{\alpha=\bar{r}+1}^r\frac{E\lalp(\lam,\rho)}{2}-\frac{\ell_Y}{2}\, \leq\\
&\leq &\Bl1+\frac{C\upmo\epsilon}{\deg X}\Br\Bl\dim W_{Y}\cap W_{m_0}-\ell_Y\Br
-\frac{1}{2}\sum_{\pi(p)\in \bx\cap\pi(\ti\Lam)\cap{Y}}\!\!\rho_{i_0(p)}.\nonumber\\
&\leq&\Bl1+\frac{C\upmo\epsilon}{\deg X}\Br\Bl\dim W_{Y}\cap W_{m_0}-\ell_Y\Br
-\frac{1}{2}\sum_{p\in \ti S_{\reg}}\hat\rho_{i_0(p)}.\nonumber
\end{eqnarray}
{Here we used that for all $p'\in \pi(\ti S_{\text {reg}})-X\cap \pi(\ti \Lam)\cap Y$, $\rho_{i_0}(p')=0$.
And the last inequality holds since by the Definition of $\ti S_{\reg}$ and $\hat\rho_i$ (cf. \eqref{hat-rho}), we have
$\sum_{q\in \ti S_{\reg}}\hat{\rho}_{i_0(q)}\leq \sum_{\pi(p)\in \bx\cap\pi(\ti\Lam)\cap{Y}}\rho_{i_0(p)}$.}

Using $E_{Y\comp}(\lam,\rho)=2 \deg Y\comp$ and $\deg X-g=m$, we obtain
\begin{eqnarray*}
&& \frac{E_{X}(\lam,\rho)}{2}= \Bl\deg Y\comp +\frac{\ell_Y}{2}\Br+
\Bl \frac{E_{Y}(\lam,\rho)}{2} -\frac{\ell_Y}{2}\Br\nonumber\\
&\leq & \Bl\deg Y\comp +\frac{\ell_Y}{2}\Br+\Bl1+\frac{2C\upmo\epsilon}{m+1}\Br
\Bl m_0+1-\dim W_Y\Br-
\frac{1}{2}\sum_{p\in \ti S_{\reg}}\!\!\rho_{i_0(p)}\ ,\nonumber
\end{eqnarray*}
Here the last inequality follows from
\begin{eqnarray*}
\dim W_{m_0}
&\geq&
\dim(W_{m_0}\cap W_Y+W_{m_0}\cap W_{Y\comp})\\
&=&\dim W_{m_0}\cap W_Y+\dim W_{m_0}\cap W_{Y\comp}-\dim W_{m_0}\cap W_Y\cap W_{Y\comp}\\
&=&\dim W_{m_0}\cap W_Y+\dim W_{Y\comp}-\ell_Y\ .
\end{eqnarray*}

Now we consider the right hand side of \eqref{E_X-new} for $\rho$ chosen as in \eqref{10}, which gives
$\sum_{i=0}^m \rho_{i}=m_0+1$. Since by our assumption, the  embedding $X\sub \PP W $ is given by a complete linear
system of a very ample line bundle $\sO_X(1)$, using our choice of weights $\rho_i$ (cf. \eqref{rrr}),
$$\sum_{\alpha=1}^r\Bl \deg X_{\alpha}+\frac{\ell\lalp}{2}
-m\lalp-1\Br\cdot\rho_{\hbar\lalp}
=\deg Y\comp+\frac{\ell_Y}{2}-\dim W_{Y\comp}.$$

We claim that
$\sum_{i=0}^m\hat{\rho}_{i}=m_0+1-\dim W_{Y\comp}\ .$
Indeed, from our choice of $\rho$ and the definition of $\hat\rho$ (cf.  \eqref{hat-rho}), for any $0\leq i\leq m$,
$\hat \rho_i=1$ or $0$, and it is $0$ if and only if either $i>m_0$ of
there is an $X\lalp$ with $i\in \II\lalp$ (cf. \eqref{I-lalp}) such that $\rho_{\hbar\lalp}=1$,  that is,
$i\in \II_{Y\comp}=\bigcup_{X\lalp\sub Y\comp}\II\lalp$.
This proves
$$\sum_{i=0}^m\hat{\rho}_{i}=m_0+1-|\II_{Y\comp}|.
$$
Our claim will follow if once we prove
 $|\II_{Y\comp}|=\dim W_{Y\comp}$; but this follows from the criteria
\beq\label{jump1}
i\in \II_{Y\comp}\quad \text{\sl if and only if}\quad %\Longleftrightarrow
 \dim W_{i+1}\cap W_{X\lalp}-\dim W_i\cap W_{X\lalp}=1 \text{ for some }
X\lalp\sub Y\comp.
\eeq
To justify this criteria, we notice that  $ \dim W_{i+1}\cap W_{X\lalp}=\dim W_i\cap W_{X\lalp}$  for all
$X\lalp\sub Y\comp$ is equivalent to
$Y\comp\cap\{s_i=\cdots=s_m=0\}=Y\comp\cap\{s_{i+1}=\cdots=s_m=0\}$ as subschemes of $Y\comp$;
that is, $\inc(s_i)\cap Y\comp=\emptyset$.  Since $\lam$ is a staircase
$$i\not\in \II\lalp\text{ for all }X\lalp\sub Y\comp
\quad\text{if and only if}\quad %\Leftrightarrow
\inc(s_i)\cap Y\comp=\emptyset \text{  (cf. (\ref{I-lalp})) }.
$$
This proves \eqref{jump1}.

With those in hand, we obtain
\begin{eqnarray*}
&& \frac{E_{X}(\lam,\rho)}{2}= \Bl\deg Y\comp +\frac{\ell_Y}{2}\Br+
\Bl \frac{E_{Y}(\lam,\rho)}{2} -\frac{\ell_Y}{2}\Br\nonumber\\
&\leq & \Bl\deg Y\comp +\frac{\ell_Y}{2}\Br+\Bl1+\frac{2C\upmo\epsilon}{m+1}\Br
\Bl m_0+1-\dim W_{Y\comp}\Br-\sum_{q\in \ti S_{\reg}}\frac{\hat{\rho}_{i_0(q)}}{2}\\
&\leq & m_0+1-\sum_{q\in \ti S_{\reg}}\frac{\hat{\rho}_{i_0(q)}}{2}+\Bl\deg Y\comp +\frac{\ell_Y}{2}-
\dim W_{Y\comp}\Br+\frac{2C\upmo\epsilon}{m+1}\Bl m_0+1-\dim W_{Y\comp}\Br\\
&=&\sum_{i=0}^m \rho_{i}-\!\!\!\sum_{q\in \ti S_{\reg}}\frac{\hat\rho_{i_0(q)}}{2}
+\sum_{\alpha=1}^r\Bl \deg X\lalp+\frac{\ell\lalp}{2}
-m\lalp-1\Br\cdot\rho_{\hbar\lalp}+\frac{2C\upmo\epsilon}{ m+1}\sum_{i=0}^m\hat{\rho}_{i}\ .
\end{eqnarray*}
So the proof of Proposition is completed.
\end{proof}

We state and prove the main result of this section.
We introduce
\begin{equation}\label{T+m}
\hat\omega(\lam,\rho):=\frac{2\deg X}{m+1}\sum_{i=0}^m\rho_i-E_X(\lam,\rho)\and
\hat\omega_\ba(\lam)=\hat\omega(\lam)+\mu_\ba(\lam).
\eeq
where $E_X(\lam,\rho):=\sum_{\alpha=1}^r E\lalp(\rho)$.
By Theorem \ref{main-est}, we have
$\omega(\lam)\geq\hat\omega(\lam)$.

\begin{theo}\label{lower-bdd}
Let $(X,\sO_X(1),\bx,\ba)$ be a connected weighted pointed nodal curve that is slope stable.
Suppose $\omega_{X}(\bax)$ is ample. We let $1>\epsilon>0$ be such that
$$2(2C\upmo+1)\epsilon<\deg\omega_X(\bax)
$$
with $C$ given in Lemma \ref{deg-X-a}.
Then there exists an $M$ depending only on $\chi_\ba(X)$ %(cf. Theorem \ref{main})
and $\epsilon$ such that whenever $\deg X>M$, then for any staircase 1-PS $\lam$   we have
\beq\label{hat-rho-bdd}
\omega_\ba(\lam)=\omega(\lam)+\mu_\ba(\lam)
\geq \hat\omega(\lam)+\mu_\ba(\lam) %=\frac{\mu_{\ba}(\lam,\rho)}{2}+\frac{\deg X}{m+1}\sum_{i=0}^m\rho_i-\frac{E_X(\lam,\rho)}{2}
\geq \frac{2\cdot \epsilon}{m+1}\sum_{i=0}^m\hat\rho_i\ .
\eeq
\end{theo}
\begin{proof}
By Proposition \ref{new-bdd}, it suffices to prove
\begin{eqnarray}\label{coro}
&&\sum_{i=0}^m \rho_{i}-\!\!\!\sum_{q\in \ti S_{\reg}}\frac{\hat\rho_{i_0(q)}}{2}
+\sum_{\alpha=1}^r\Bl \deg X\lalp+\frac{\ell\lalp}{2}
-m\lalp-1\Br\cdot\rho_{\hbar\lalp}+\frac{(2C\upmo+1)\epsilon}{ m+1}\sum_{i=0}^m\hat{\rho}_{i}\leq \\
&%\qquad\qquad
\quad &\leq \frac{\deg X}{m+1}\sum_{i=0}^m\rho_i+\frac{\mu_{\ba}(\lam,\rho)}{2}\ .\nonumber
\end{eqnarray}
By linear programming, we only need to prove the above estimate for $\rho$ of the form \eqref{10}. We will break the verification
into several inequalities. First, we have
\beq\label{mu-ba}
\mu_{\ba}(\lam,\rho)=
%\sum_{j=1}^n a_j\frac{m_0+1}{m+1}-\sum_{x_j\in\PP W_{m_0}} a_j \\
\sum_{j=1}^n a_j\frac{m_0+1}{m+1}-\sum_{x_j\in Y\cap\PP W_{m_0}} a_j-\sum_{x_j\in Y\comp\cap\PP W_{m_0}} a_j.
\eeq
Here $x_j$ runs through all marked points of the curve. We claim  that
\beq\label{ti-s}
\sum_{q\in \ti S_{\reg}}\frac{\hat{\rho}_{i_0(q)}}{2}=\frac{|\bx\cap\pi(\ti\Lam)\cap{Y}\cap \PP W_{m_0}|}{2}\geq
\sum_{x_j\in Y\cap \PP W_{m_0} }\frac{a_j}{2}.
\eeq
To this purpose,
we first show that
\beq\label{supp-bx}
\bx\cap\pi(\ti \Lam)\cap Y\cap\PP W_{m_0}=\bx\cap Y\cap\PP W_{m_0}.
\eeq
Indeed, for any $x_i$ in $\bx$ that lies in
$Y\cap \PP W_{m_0}$, $s_{k}(x_j)=0$ for $k\geq m_0$. On the other hand, let $x_j\in X\lalp\sub Y$; since $\Ycomp$ is
the largest subcurve of $X$ contained in $\PP W_{m_0}$, for some $k\ge m_0$, $s_k|_{X\lalp}\ne 0$. Combined with $s_k(x_j)=0$, we conclude
$x_j\in \pi(\ti\Lam)$ (cf. Definition \ref{index}). In particular $\bx\cap Y\cap \PP W_{m_0}\sub \pi(\ti\Lam)$. This proves \eqref{supp-bx}.

Applying \eqref{supp-bx}, and using that %all the points $x_j$ in $\bx$ are smooth points of $X$, and
for any colliding subset $\{ x_{i_{1}},\cdots ,x_{i_{s}}\}$ (i.e. $x_{i_1}=\cdots=x_{i_s}$)
necessarily $a_{i_{1}}+\cdots +a_{i_{s}}\leq 1$, we obtain
\beq
\sum_{x_j\in Y\cap\PP W_{m_0}}\frac{ a_j}{2}-\frac{|\bx\cap\pi(\ti\Lam)\cap{Y}\cap \PP W_{m_0}|}{2}=
\sum_{x_j\in Y\cap\PP W_{m_0}}\frac{ a_j}{2}-\frac{|\bx\cap{Y}\cap \PP W_{m_0}|}{2}\,\leq 0,
\eeq
hence \eqref{ti-s}.

By putting \eqref{mu-ba} and \eqref{ti-s} together, we obtain
\beq\label{mu-rho}
-\!\!\!\sum_{q\in \ti S_{\reg}}\frac{\hat\rho_{i_0(q)}}{2}-\frac{\mu_{\ba}(\lam,\rho)}{2}
\leq-\frac{m_0+1}{m+1}\sum_{j=1}^n \frac{a_j}{2}+\sum_{x_j\in Y\comp\cap\PP W_{m_0}} \frac{a_j}{2} \ .
\eeq
On the other hand, for $\rho$ of the form in \eqref{10}, we have
\begin{eqnarray}\label{rho-only}
&&\sum_{i=0}^m \rho_{i}
+\sum_{\alpha=1}^r\bl \deg X\lalp+\frac{\ell\lalp}{2}
-m\lalp-1\br\cdot\rho_{\hbar\lalp}+\frac{(2C\upmo+1)\epsilon}{ m+1}\sum_{i=0}^m\hat{\rho}_{i}\\
&= & m_0+1+\bl\deg Y\comp +\frac{\ell_Y}{2}-\dim W_{Y\comp}\br+\frac{(2C\upmo+1)\epsilon}{m+1}\bl
m_0+1-\dim W_{Y\comp}\br\ .\nonumber
\end{eqnarray}
Pluging \eqref{rho-only} and \eqref{mu-rho} into \eqref{coro}, and using the slope condition
$$
\deg Y^{\complement} +\frac{\ell_Y}{2}+ \sum_{x_j\in Y\comp\cap\PP W_{m_0}}
\frac{a_j}{2}\leq  \frac{\deg X+\sum_{j=1}^n \frac{a_j}{2}}{m+1}\dim W_{Y\comp},
$$
we obtain
\begin{eqnarray*}
&&\!\!\! -\frac{\mu_{\ba}(\lam,\rho)}{2}+\frac{E_X(\lam,\rho)}{2}+ \frac{((2C\upmo+1)\epsilon}{m+1}\sum_{i=0}^m\hat\rho_i\\
&\leq&\!\!\! m_0+1+\Bl\deg Y\comp +\frac{\ell_Y}{2}+\!\!\!\!\!\!\!\sum_{x_j\in Y\comp\cap \PP W_{m_0} }\!\!\!\!\!\frac{a_j}{2}-
\dim W_{Y\comp}\Br-\frac{m_0+1}{m+1}\sum_{i=1}^n \frac{a_i}{2}+\nonumber\\
&&\!\!\!+\frac{(2C\upmo+1)\epsilon}{m+1}\Bl m_0+1-\dim W_{Y\comp}\Br\nonumber \\
&=&\!\!\! \frac{\deg Y\comp +\frac{\ell_Y}{2}+\sum_{x_j\in Y\comp\cap \PP W_{m_0} }\frac{a_j}{2}}{\dim W_{Y\comp}}\dim W_{Y\comp}-\frac{m_0+1}{m+1}\sum_{i=1}^n \frac{a_i}{2}+\nonumber\\
&&\!\!\!+\Bl1+\frac{(2C\upmo+1)\epsilon}{m+1}\Br\Bl m_0+1-\dim W_{Y\comp}\Br \nonumber\\
&\leq &\!\!\! \frac{\deg X+\sum_{j=1}^n\frac{a_j}{2}}{m+1}\dim W_{Y\comp}\!+\!\Bl1+\!\!\frac{(2C\upmo+1)\epsilon}{m+1}\Br\!\Bl m_0+1-\dim W_{Y\comp}\Br\!\!-\!\!\frac{m_0+1}{m+1}\sum_{i=1}^n \frac{a_i}{2}\nonumber\\
%\end{eqnarray*}
%\begin{eqnarray*}
&\leq &\!\!\! \frac{\deg X+\sum_{j=1}^n\frac{a_j}{2}}{m+1}\Bl\dim W_{Y\comp}+ m_0+1-\dim W_{Y\comp}\Br-
\frac{m_0+1}{m+1}\sum_{i=1}^n \frac{a_i}{2}\nonumber\\
&=&\!\!\! \frac{\deg X}{m+1}\cdot (m_0+1)= \frac{\deg X}{m+1}\sum_{i=0}^m\rho_i\ ,\nonumber\\
\nonumber
\end{eqnarray*}
where we have use the assumption
$2(2C\upmo+1)\epsilon<\deg\omega_X(\bax)$
to conclude
$$ \frac{\deg X+\sum_{j=1}^n\frac{a_j}{2}}{m+1}>1+\frac{(2C\upmo+1)\epsilon}{m+1}$$ in the fourth inequality.
This completes the proof.
\end{proof}

%As a consequence, we now give the
\begin{proof}[Proof of Theorem \ref{main}]
Since $\hat{\rho_i}\geq 0$, the sufficiency follows from Theorem \ref{lower-bdd}.
We now prove the other direction.
Let $Y\sub X$ be any proper subcurve; let $W_Y\sub W$ be the linear subspace spanned by $Y$,
and let $m_0+1=\dim W_Y$.
We choose a 1-PS
$\lam=\mathrm{diag}[ t^{\rho _{0}},\cdots ,t^{\rho_{m}}] \cdot t^{-\rho_{\text{ave}}}$
such that the corresponding filtration $\{W_i\}_{i=0}^m$ satisfies $W_{m_0+1}=W_Y$; we choose
the weights $\{\rho_i\}$ be as in \eqref{10}.
Then
$$\mu_\ba=\sum_{j=1}^{n}a_{j}\Bl \frac{m_0+1}{m+1}\Br-\sum_{x_j\in \PP W_Y}^{n}a_{j}.
$$
Thus by Corollary \ref{sum-Delta} (cf. \cite[Prop 5.5]{Mum}),
$e(\ti \sI)/2=\deg Y+\ell_Y/2$; hence
\begin{eqnarray*}
0\leq\frac{\hat{\omega}+\mu_\ba}{2}
&=&\frac{m_0+1}{m+1}\cdot\deg X-(\deg Y+\frac{\ell_Y}{2})
+ \frac{m_0+1}{m+1}\sum_{j=1}^{n}\frac{a_{j}}{2}-\sum_{x_j\in \PP W_Y}^{n}\frac{a_{j}}{2}\\
&=&(m_0+1)\Bl \frac{\deg X+\sum_{j=1}^{n}\frac{a_{j}}{2}}{m+1}-
\frac{\deg Y+\frac{\ell_Y}{2}+\sum_{x_j\in Y}\frac{a_{j}}{2}}{m_0+1}\Br,
\end{eqnarray*}
which is \eqref{basic1}. This completes the proof of the Theorem.
\end{proof}

\section{Re-construction of the moduli of weighted pointed curves}\label{hilb-to-chow}
\def\Pm{{\PP^m}}
\def\Gamx{\Gamma_{\!X}}
\def\fB{\mathfrak B}
\def\bcM{\overline{\cM}}
\def\bcK{\overline{\cK}}
\def\bcP{\overline{\cP}}
\def\quotient{\overline{\cK}\!\sslash \!G}
\def\bcK{\overline{\cK}}
\def\q{\mathcal{Q}}
\def\st{\text{st}}
\def\moduli{\bcM_{g,\ba}}
\def\cX{{\mathcal X}}
\let\para=\parallel

In this section, we use GIT quotient of Hilbert scheme to construct the moduli of  weighted pointed stable curves, first introduced
and constructed by Hassett \cite{Hass} using different method.

\begin{defi}\label{weight-ss}
A {weighted pointed semi-stable} curve is a weighted pointed curve $(X,\bx,\ba)$ such that
% consisting of a connected nodal curve $X$ with $n$-marked points
%$\bx=(x_1,\ldots,x_n)$ of weights  $\ba=(a_1,\ldots,a_n)$, $a_i\in \QQ^n_+$, such that
\begin{enumerate}
%\item $\bx$ are away from the nodes of $X$,
%and the total weights of those $x_i$ that lie over a single point $p\in X$ does not exceed one;
\item $\omega_X(\bax)$ is numerically non-negative;
\item the total degree $2\chi_{\ba}(X)=\deg \omega_X(\bax)$ is positive;
\item \label{E-line} for any smooth  subcurve $E\sub X$ such that $\deg\omega_X(\bax)|_E=0$,
necessarily $E\cap \bx=\emptyset$ and $E\cong \PP^1$.
%\footnote{\red This excludes the existence of rational components with only one linking nodes and several marked points.}
\end{enumerate}

We call $E\sub X$ satisfying (3) {exceptional components}. We say $(X,\bx, \ba)$ is {weighted pointed stable} if
it does not contain exceptional components.
\end{defi}

We fix integers $n$, $g$ and weights $\ba\in \QQ_+^n$ satisfying $\chi_{\ba}(X)>0$; for a large integer $k$
such that $k\cdot a_i\in\ZZ$ for all $i$, we let $d=(|\ba| +2g-2)\cdot k$, and form
\beq\label{spe}
P(t)=d\cdot t+1-g\in \ZZ[t], \quad\text{and set}\ \, m=P(1).
\eeq
We denote by
$\Hilb_{\PP^m}^{P}$  the Hilbert scheme of subschemes of $\PP^m$ of Hilbert polynomial $P$;
we define $\cH$ %\Hilb_\Pm^{P,n}$
be the fine moduli scheme of families of data
$$%\Hilb_\Pm^{P,n}
\cH=_{\text{set}}\{(X,\iota, \bx)\mid [\iota: X\to \Pm]\in \Hilb_\Pm^P, \ \bx=(x_1,\cdots, x_n)\in X^n\}.
$$
%Here we gave a set description of $\cH$; it is straightforward to show that the
Using that Hilbert schemes are projective, we see that $\cH$ exists and is projective. We denote by
\beq\label{univ}
(\pi_\cH, \varphi): \sX\lra \cH\times\Pm,\quad \mathfrak x_i: \cH\to \sX
\eeq
the universal family of $\cH$.

%where $\ss_i$ such that $\pi_\cH\circ \ss_i=\text{id}$ are sections of the $n$-marked points of the
%universal curve $\pi_\cH: \sX\to \cH$, and $\varphi: \sX\to\Pm$ is the universal embedding.

We introduce a parallel space for the Chow variety. We let $\Chow_\Pm^d$ be the Chow variety of degree $d$
dimension one effective cycles in $\PP^m$. For any such cycle $Z$, we denote by
$\Chow(Z)\in \mathrm{Div}^{d,d}[(\PP^m)\dual\times (\PP^m)\dual]$
% \Chow_\Pm^d$ %\sub \mathrm{Div}^{d,d}[ ( (\PP^m)\dual ) ^2]$$
its associated Chow point (cf. Section \ref{intro}).  We define
$$
\cC:=\{(Z,\bx) \in \Chow_\Pm^d\times (\PP^m)^n
\mid  \bx=(x_1,\cdots,x_n) \in \text{supp}(Z)^n\}.
$$
By Chow Theorem, $\cC$ is projective. Using the Chow coordinate, we obtain an injective morphism
\beq\label{cH}
\cC \mapright{\sub}\mathrm{Div}^{d,d}[(\PP^m)\dual\times (\PP^m)\dual  ]\times (\PP^m)^n\ .
\eeq
Like before (cf. Section \ref{intro}), we endow it with an ample $\QQ$-line bundle
$\sO_{\cC}(1,\ba)$ (depending on the weights $\ba$); %=\sO_{\mathrm{Div}^{d,d}[ ( (\PP^m)\dual ) ^2]\times (\PP^m)^n}(1,\ba)|_{\cC}$
the line bundle is canonically linearized by the diagonal action of
$$G\defeq SL(m+1)
$$
on \eqref{cH}.
%The definitions of various notions of stability of $\xi\in \cC$ are understood in the sense of  Definition \ref{chow-ss}.  Let
We let $\cC^{ss}\sub\cC$
be the (open) set of {\em semi-stable} points with respect to the ${G}$ linearization on $\sO_\cC(1,\ba)$.

For any one-dimensional subscheme $X\sub \PP^m$, we denote by $[X]$ its associated one-dimensional cycle.
By sending $(X,\iota,\bx) \in \cH$ to $([X], \bx)\in\cC$,
% the Chow point of its underlying $n$-pointed cycle $\Chow(X,\bx)=(\Chow(X),\bx)\in\cC$,
we obtain the ${G}$-linear  {\em Hilbert-Chow} morphism (cf. \cite[Section 5.4]{MFK})
$$\Phi:\cH\longrightarrow\cC.
$$

\begin{lemm}\label{weight-nodal}
For fixed $g$, $n$ and $\ba$ satisfying $\chi_{\ba}(X)>0$ ,
there is an integer $M$ depending on $g$, $n$ and $\ba$ so that for $d\ge M$,  $\Phi\upmo( \cC^{ss})$ consists exactly of those $(X,\iota,\bx)\in\cH$ so that
the associated data $(X, \iota^*\sO_{\PP^m} (1), \bx, \ba)$ is a slope semi-stable weighted nodal curve.
% (cf. Definition \ref{weight-pt-curve}).
\end{lemm}

\begin{proof}
By an argument parallel  to  \cite[Prop. 5.5]{Mum}, one proves that there is an $M$ depending only on
$\chi_\ba(X)$ such that for $d\geq M$, $\Chow(X,\bx)\in \cC^{ss}$ implies that $X$ is a nodal
curve and the inclusion $\iota: X\rightarrow \PP^m$  is given by a complete linear system.

We now show that any $(X,\bx,\ba)\in \Phi\upmo(\cC^{ss})$ is  a weighted pointed nodal curve as defined in the beginning of the paper.
We first check that the weighted points are away from the nodes, and
the total weight at any point is no more than one.

Let $p\in X$; we choose the 1-PS $\lam$ as in the Example \ref{weighted-pt}; the $\lam$-weight for
$\Chow(X,\bx)$  is (cf.\eqref{omega-mu})
$$
\omega(\lam)+\mu_\ba(\lam)
=\frac{2\deg X}{m+1}-\epsilon_p+ \frac{1}{m+1}\sum_{j=1}^na_j-\sum_{x_j=p}a_j=
2-\epsilon_p+\frac{2\chi_\ba(X)}{m+1}-\sum_{x_j=p}a_j,
$$
where $\epsilon_p=2$ if $p$ is a node and $1$ otherwise.  Since $\Chow(X,\bx)$ is semistable,
we must have $0\le \omega(\lam)+\mu_\ba(\lam)$.
Now we choose $M'$ so that $m+1=M'+1-g(X)>\frac{2\chi_\ba(X)}{\min_{a_j>0}\{a_i\}}$;
then
$0\le \omega(\lam)+\mu_\ba(\lam)$
implies that the weighted points are away from the nodes, and the total weight of marked points at $p$
does not exceed one.

Finally, the slope semi-stability of $(X,\bx,\ba)$ follows from the necessity part of Theorem \ref{main}; the condition
\eqref{E-line} of Definition \ref{weight-ss} follows from the inequality \eqref{basic} and Lemma \ref{deg-X-a}.
This proves that all $(X,\bx,\ba)\in \Phi\upmo(\cC^{xx})$ are slope-semi-stable weighted nodal curves. The other
direction is straightforward, and will be omitted.
\end{proof}

We define
$$\cH^{ss}=\Phi\upmo(\cC^{ss})\sub\cH.
$$

\begin{coro}
For $d\ge M$ specified in Lemma \ref{weight-nodal}, the restriction
$$\Phi^{ss}\defeq \Phi|_{\cH^{ss}}: {\cH^{ss}}\to \cC^{ss}
$$
is an isomorphism
\end{coro}

\begin{proof}
First, since $\cH$ is proper, $\Phi^{ss}$ is surjective.
%By Lemma \ref{weight-nodal},  $\Chow(X,\bx)\in \cC^{ss}$ implies that $X$ is a weighted pointed nodal curve embedded by complete linear system, hence the image of $(X,\iota,\bx)\in \cH$ under $\Phi$ is $\Chow(X,\bx)$. This implies that $\Phi|_{\cH^{ss}}$ is surjective.
For the injectivity of $\Phi^{ss}$, suppose there are $(X,\iota,\bx)$
and $(X',\iota',\bx')\in \cH^{ss}$ such that $\Phi(X,\iota, \bx)=\Phi(X',\iota',\bx')\in \CC^{ss}$,
then by Lemma \ref{weight-nodal}, both $X$ and $X'$ are nodal subcurves of $\PP^m$.
Since $\Phi(X,\iota, \bx)=\Phi(X',\iota',\bx')\in \CC^{ss}$, the cycles $[X]=[X']$ and $\bx=\bx'\sub \PP^m$;
since both $X$ and $X'$ are nodal, we have $X=X'$.
This proves $(X,\iota,\bx)=(X',\iota',\bx')$; thus $\Phi^{ss}$ is bijective.
Finally, $\Phi^{ss}$ is an isomorphism since both $\cH^{ss}$ and $\cC^{ss}$ are smooth.
\end{proof}

\def\Gamx{\Gamma_{\!X}}
\def\fB{\mathfrak B}
\def\bcM{\overline{\cM}}
\def\bq{\mathbf q}

To construct the moduli of weighted pointed curves, taking the $k$ specified before \eqref{spe}, we form
$$\cK=\{(X,\iota,\bx)\in\cH\mid X\ \text{smooth weighted pointed curves}, \iota\sta\sO_{\PP^m}(1)\cong\omega_X(\ba\cdot\bx)^{\otimes k}\}.
$$
It is locally closed, and is smooth. Since $X$ in $(X, \iota,\bx)\in \cK$ are smooth, applying Theorem \ref{main}, we conclude that
$\Phi(\cK)\sub \cC^{ss}$, thus $\cK\sub \cH^{ss}$.
Let $\overline{\cK}\sub\bH^{ss}$  be the closure of $\cK$ in $\bH^{ss}$.
Because $\Phi^{ss}$ is an isomorphism, and $\cC$ is projective, the GIT quotient $\cH^{ss}/\!\!/ G\cong \cC^{ss}/\!\!/ G$
exists and is projective. Because $\overline\cK$ is closed in $\cH^{ss}$, the GIT quotient
\beq\label{bq}
\bq:\overline{\cK}\longrightarrow\quotient
\eeq
exits and is projective.

\begin{theo}\label{moduli}
The coarse moduli space $\bcM_{g,\ba}$ of stable genus $g$, $\ba$-weighted nodal curves constructed by Hassett
is canonically isomorphic to the GIT quotient $\overline{\cK}\sslash G$.
\end{theo}

The main technical part of the proof is to analyze the closed points of $\quotient$. We have the following
preliminary results.

For any $(X,\iota,\bx)\in \bcK$, since the associated weighted pointed curve $(X,\bx,\ba)$ is semistable, we can form a
new weighted pointed curve by contracting all of its exceptional components %of$E\sub (X,\bx,\ba)$; namely, those components
(cf. Definition \ref{weight-ss}). We denote the resulting curve by
\beq\label{stabilization}
(X^\st,\bx^\st,\ba),
\eeq
and call it the {\em stablization of} $(X,\bx,\ba)$.
Since the marked points never lie on the contracted components, the stabilization
produces a weighted pointed nodal curve of the same genus.
Further, the stabilization also applies to families of semistable weighted curves. Thus applying this to
the restriction to $\bcK$ of the universal family of $\cH$, we obtain a family of
weighted pointed stable curves on $\bcK$.
Since $\bcM_{g,\ba}$ is the coarse moduli space of stable weighted pointed nodal curve, we obtain a morphism
\beq\label{Psi}
\overline \Psi: \bcK \lra \bcM_{g,\ba}.
\eeq
As this morphism is $G$-equivariant with $G$ acting trivially on $\bcM_{g,\ba}$, it descends to a morphism
\beq\label{st-Psi}
\psi: \quotient \lra \bcM_{g,\ba}.
\eeq
We will prove Theorem \ref{moduli} by proving that $\psi$ is an isomorphism.

%
%Applying deformation of curves with marked points, one sees that
%both $\quotient$ and $\moduli$ have quotient singularities. Thus $\Psi$ will be an isomorphism
%if it is both injective and surjective.

\subsection{Surjectivity}

Let $(X,\bx, \ba)$ be a weighted pointed stable curve. We endow it the
polarization $\sO_X(1)=\omega_X(\ba\cdot\bx)^{\otimes k}$
together with the embedding $\iota: X\to \PP H^0(\sO_X(1))\dual$.
When $X$ is smooth, $(X,\iota, \bx,\ba)$ lies in $\cK$; when $X$ is singular,
this may not necessarily hold.
Our solution is to replace $\omega_X(\ba\cdot\bx)^{\otimes k}$ by its twist, to be
defined momentarily.

Given $(X,\bx)$, we choose a smoothing
$\pi: \cX\to T$
over a pointed curve $0\in T$ such that
$\cX$ is smooth and $\cX_0=\cX\times_T 0\cong X$.
By shrinking $T$ if necessary, we can extend the $n$-marked points of $X$ to sections $\fx_i: T\to \cX$
so that, denoting $\fx=(\fx_1,\ldots, \fx_n)$,
$(\cX, \fx, \ba)$ form a flat family of  weighted pointed stable curves. Let $X_1,\ldots, X_r$ be the
irreducible components of $X$. The following Proposition gives the surjectivity of $\psi$.

\begin{prop}\label{6.6}
Let $(X,\bx, \ba)$ be a weighted pointed stable curve, and let $(\cX,\fx,\ba)$ be the $T$-family
constructed.
Then there exist non-negative integers $\{b\lalp\}_{\alpha=1}^r$ so that after letting
\beq\label{eq6.6}
\sO_{\cX}(1)=\omega_{\cX/T}(\ba\cdot\bs)^{\otimes k}\otimes_{\sO_{\cX}}\sO_{\cX}(\sum b\lalp X\lalp),
\eeq
$(\cX, \sO_\cX(1), \bs)$ forms a family of slope semistable weighted pointed nodal curves.
\end{prop}

The Proposition was essentially proved by Caporaso in \cite{Cap}. Since we need to use the
same technique to prove the injectivity, we recall the notation used to prove this Proposition.
The remainder part of this subsection essentially follows \cite{Cap}.

%We let $X$ be a nodal curve, having
%$r$ irreducible componentes, indexed as $X_1,\ldots, X_r$.
For any line bundle $\sL$ on $X$, we denote $\delta\lalp(\sL)=\deg\sL|_{X\lalp}$.
%Using the fixed indexing of irreducible components $X_i$ of $X$, we abbreviate
%$\delta_i(\sL)=\delta_{X_i}(\sL)$.
We define the associated lattice point of $\sL$ be
$$\vec\delta(\sL)\defeq (\delta_1(\sL),\ldots,\delta_r(\sL))\in \ZZ^{\oplus r}.
%\quad \vec\delta(\sL)=(\delta_1(\sL),\ldots,\delta_r(\sL)).
$$
We call $\vec\delta(\sL)$ the numerical class of $\sL$.

We next introduce a subgroup $\Gamma_X\sub \ZZ^{\oplus r}$. %generated by $\vec\ell\lalp$ as follows.
We let
\beq\label{ell_ab}
\ell_{\alpha,\beta}=\ell_{\alpha,\beta}(X)=|X\lalp\cap X\lbe|\quad \text{if}\ \alpha\ne \beta;
\and \ell_{\alpha,\alpha}=\ell_{\alpha,\alpha}(X)= -|X\lalp\cap\ X\lalp\comp|.
\eeq
We define
$\vec\ell_\alpha=\vec\ell_\alpha(X)=(\ell_{\alpha,1}(X),\ell_{\alpha,2}(X),\cdots,\ell_{\alpha,r}(X))$.
%,\ \vec\ell_2=(\ell_{2,1},\ell_{2,2},\cdots,\ell_{2,r}),\
%\cdots\ \vec\ell_r=(\ell_{r,1},\ell_{r,2},\cdots,\ell_{r,r});
%$$
We define $\Gamma_X\sub \ZZ^{\oplus r}$ be the subgroup generated by $\vec\ell_1,\ldots,\vec\ell_r$.

\begin{rema}
%Let $X\cong\cX_0\sub \cX$ be as in Proposition \ref{6.6}, and
Let $\sL=\omega_X(\ba\cdot\bx)^{\otimes k}$. Since
$\cX$ is smooth, for
the invertible sheaf $\sO_{\cX}(1)$ stated in \eqref{eq6.6} %Proposition \ref{6.6}
depending on the integers $b_1,\ldots, b_r$,
we have
$$\vec\delta(\sO_{\cX}(1)|_X)=\vec\delta(\sL)+\sum_{\alpha=1}^r b\lalp \vec\ell\lalp.
$$
\end{rema}

This says that any two choices of $\sO_{\cX}(1)$ restricted to the central fiber have equivalent
numerical classes modulo $\Gamma_X$. This motivates the definition

\begin{defi} We define the {\em degree class group} of $X$ be the quotient $\ZZ^{\oplus r}/\Gamx$.
\end{defi}

We introduce one more notation. For any vector $\vec v=(v_1,\ldots,v_r)\in \ZZ^{\oplus r}$ and any
subcurve $Y\sub X$, mimicimg the notion of degree, we define
$$\deg_Y \vec v=\sum_{X\lalp\sub Y} v\lalp.
$$
In particular, $\deg_X \vec v=\sum v\lalp$.

Let $(X, \bx, \ba)$ be a weighted pointed nodal curve and $\sL$ a line bundle on $X$ of total degree $d$.
For any subcurve $Y\sub X$, we
introduce the {\em extremes} of $Y$ (depending on $d$) be
\beq\label{m_Y}
\bM^\pm_Y:=\frac{\deg_Y\omega_X(\ba\cdot\bx)}{\deg\omega_X(\ba\cdot\bx)}\Bl d
+\sum_{j=1}^n\frac{a_{j}}{2}\Br-\sum_{x_{j}\in Y}\frac{a_{j}}{2}\pm \frac{\ell_Y}{2}.
\eeq

Proposition \ref{main-c} is reformulated as

\begin{lemm} \label{basic-m1}
Suppose $\chi_{\ba}(X)>0$, and let $d>M$, where $M$ is defined in Corollary \ref{basic=basic1}. Then
a weighted pointed nodal curve $(X,\bx, \ba)$ with a numerical effective line bundle $\sL$ on $X$ of $\deg\sL=d$
is slope semi-stable (cf. Proposition \ref{main-c}) if and only if
\beq\label{basic-m}
\deg_Y\sL  \in[\bM^-_Y,\bM^+_Y] \ \text { for any subcurve } Y\sub X.
\eeq
%It is stable if $[M^-_Y,M^+_Y]$ is replaced by $(M^-_Y,M^+_Y)$.
\end{lemm}

We quote the basic properties of extremes.

\begin{lemm}\label{m-M}Let $Y$, $Y_1$ and $Y_2$ be surcurves of $X$. We have\\
%\begin{enumerate}
(1) $\bM^+_Y-\bM^-_Y=\ell_Y$, and $ \bM^-_{Y\comp}+\bM^+_Y=v_{Y\comp}+v_Y=d$;\\
(2) suppose $E\sub X$ is an {\em exceptional} component such that $|E \cap Y|=1$, then
$\bM^\pm_{E\cup Y}=\bM^\pm_Y$;\\
(3) suppose $Y_1$ and $Y_2$ have no common component, then
% $\ell_{Y_1,Y_2}=|Y_1\cap Y_2| $ and
$\bM^\pm_{Y_1\cup Y_2}\pm |Y_1\cap Y_2|=\bM^\pm_{Y_1}+\bM^\pm_{Y_2}$.
% \text{ and }\ M_{Y_1\cup Y_2}=M_{Y_1}+M_{Y_2}-\ell_{Y_1,Y_2}$.
%\item\label{4} suppose $\vec \delta(\sO_X(1))\in\fB^d_{X,\bx,\ba} $, then every chain of exceptional components has
%total degree $\leq 1$;
%\item Suppose $\pi: (\bar X,\bar \bx)\rightarrow (X,\bx)$ is a blow-up, then
%then $\vec\delta(\pi^*\sO_X(1))\in \fB^d_{\bar X,\bar\bx,\ba}$ if and only if $\vec\delta(\sO_{X}(1))\in\fB^d_{X,\bx,\ba}$.
%\end{enumerate}
\end{lemm}

\begin{proof}
The proof is a direct check, and will be omitted.
\end{proof}

%The Lemma motivates the following definition.
Let $\ZZ_{\ge 0}^{\oplus r}$ be those $\vec v=(v_i)\in \ZZ^{\oplus r}$ so that $v_i\ge 0$. We define
$$%\beq%\label{fB}
\fB^d_{X,\bx,\ba}=\{\vec v\in \ZZ_{\ge 0}^{\oplus r}\ \big|\ \deg_X\vec v=d,\  \vec v  \text{ satisfies \eqref{basic-m}
with $\deg_Y\sL$ replaced by $ \deg_Y \vec v$}\}. % \text{ and } d\lalp\geq 0,\forall 1\leq \alpha\leq r\}.
$$

\begin{prop} \label{exist-d}
Let $(X,\bx,\ba)$ be a weighted pointed semi-stable (cf. Definition \ref{weight-ss})
curve and $d$ sufficiently large.
Then for any $\vec v\in\ZZ^{\oplus r}$ satisfying
$ \deg_X \vec v=d$, we have
$$(\vec v+\Gamx)\cap \fB^d_{X,\bx,\ba}\ne\emptyset.
$$
\end{prop}
\begin{proof}
The proof is parallel to that of \cite[Prop. 4.1]{Cap}, and will be omitted.
\end{proof}

\begin{lemm}\label{d>0}
Let $(X,\bx,\ba)$ be a weighted pointed nodal curve satisfying $\chi_\ba(X)>0$,
then there is a constant $K$ depends only on the genus $g$, $\chi_\ba(X)$ and $\ba$
such that if $d\geq K $, then $\bM^-_Y>0$ for any connected subcurve $Y\sub X$ satisfying
$\deg_Y\omega_X(\ba\cdot\bx)>0$.
\end{lemm}

\begin{proof}
Let $Y\sub X$ be a connected subcurve such that $\deg_Y\omega_X(\ba\cdot\bx)>0$.
Since the expression $\bM^-_Y$ only involves the nodes $L_Y=Y\cap Y\comp$, without lose of
generality we can assume that $X$ consists of two smooth irreducible components $Y$ and $Y\comp$.
Then $\ell_Y\le g+1$. Thus to prove that for large $d$ we have $\bM_Y^->0$ whenever $\deg_Y \omega_X(\ba\cdot\bx)>0$,
it suffices to show that
\beq\label{inf-Y}
\inf\{ \deg_Y \omega_X(\ba\cdot\bx)\mid \deg_Y \omega_X(\ba\cdot\bx)>0\}
\eeq
is bounded below by a positive constant depending only on $\ba$ and $g$. But this is true because
\eqref{inf-Y} is bounded below by
$$\kappa=\inf\Bigl\{\sum_{i\in I}a_i-l\, \Big|\, l\in \ZZ,\ I\sub\{1,\cdots,n\},\ \sum_{i\in I}a_i-l>0\Bigr\}.
$$
Since $\ba$ is fixed, $\kappa$ is positive.
Since $d=k\cdot \deg \omega_X(\ba\cdot\bx)$, for large $k$, which is the same as for large $d$, we have
$\bM_Y^->0$ whenever $\deg_Y \omega_X(\ba\cdot\bx)>0$.
This proves the Lemma.
\end{proof}

\begin{proof}[Proof of Proposition \ref{6.6}]
By Proposition \ref{exist-d}, there are $\{b\lalp\}$'s such that for the
$\sO_\cX(1)$ given in \eqref{eq6.6} and $\sL=\sO_{\cX}(1))|_{\cX_0}$,
$\vec\delta(\sL)$ satisfies \eqref{basic-m}.
To show that $(X,\sL,\bx,\ba)$ is a polarized slope semi-stable curve,
we need to show that $\sL$ is ample. Since $\vec\delta(\sL)$
satisfies \eqref{basic-m}, $\sL$ is ample if  $\bM_{X\lalp}^->0$ for any
component $X\lalp\sub X$; but this is precisely Lemma \ref{d>0} because $(X,\bx,\ba)$ is weighted pointed
implies that $\omega_X(\bax)$ is positive.
\end{proof}

\subsection{Injectivity}\label{uni-b}
In this subsection, we use the separatedness of $\quotient$ to prove that $\psi$ in \eqref{st-Psi} is injective.
%Before we proceed, let us introduce some
%\begin{defi}\label{except-set}
% An irreducible component $E$ of a pointed nodal curve $(X,\bx)$ is an {\em exceptional component} if $\deg_E\omega_X=0$ and $E\cap \bx=\emptyset$, otherwise we call it {\em non-exceptional}. A subcurve $\bE\sub X$ is called a {\em maximal exceptional chain} if it is a connected chain of exceptional components and $L_{\bE}=\bE\cap\bE\comp$ is contained in two non-exceptional components.  Finally,  we define the {\em exceptional set }  $E_X$ of $X$ to be  the union of all exceptional components.
%\end{defi}

\begin{defi}\label{blow-up}
Let $(X,\bx,\ba)$ be a weighted pointed nodal curve (cf. Definition \ref{weight-pt-curve}).
%An irreducible component $E\sub X$ is called an {\em exceptional component} if $E\cap \bx=\emptyset$ and
% $\deg_E \omega_X=0$.
We say a weighted pointed nodal curve $(\bar X,\bar \bx, \ba)$ is a {\em blow-up} of $(X,\bx,\ba)$
if there is a morphism $\pi: \bar X\rightarrow X$ that is derived by contracting some of the exceptional components of
$(\bar X,\bar \bx, \ba)$.
\end{defi}

(Recall that exceptional component is defined in Definition \eqref{weight-ss}.) Suppose
$(\bar X,\bar \bx, \ba)$ is a blow-up of $(X,\bx,\ba)$, then
 $(\bar X^\st,\bar\bx^\st,\ba)=(X^\st,\bx^\st,\ba)$ (cf. \eqref{stabilization}).

Since the restriction of $\psi$ to $\cK\!\sslash\!{G}$ is an isomorphism, $\psi$ is a birational  morphism.
By Zariski's Main theorem and the properness of $\quotient$, the injectivity follows from

\begin{lemm}\label{inj}
$\psi\upmo(\psi(\xi))$ is zero dimensional for each $\xi\in\quotient$. %\setminus( \cK\!\sslash\!{G})$.
\end{lemm}

\begin{proof}Let $\xi\in\quotient\setminus( \cK\!\sslash\!{G})$, and let $\psi(\xi)=(X,\bx,\ba)\in \overline{\cM}_{g,\ba}$ be
the associated weighted pointed stable curve. We describe the set
$\Theta_\xi=\bq\upmo\bl\psi\upmo(\psi(\xi))\br\sub \bcK$, where $\bq$ is defined in \eqref{bq}.

For any $\eta=(\bar X,\iota,\bar\bx)\in\Theta_{\xi}\sub \bcK$, there is a smooth affine curve $\phi: 0\in T\rightarrow \bcK$
so that the pull back of the universal family of $\overline \cK$, say $\pi:(\cX,\sL,\bs)\rightarrow T$,
contains $(\bar X,\iota^* \sO_{\PP^m}(1),\bar \bx)$
as its central fiber and $\phi(T\setminus\{0\})\sub \cK$, and that the total space $\cX$ is smooth.

By Lemma \ref{weight-nodal}, the central fiber $(\bar X,\bar\bx,\ba)$
is weighted pointed semi-stable (cf. Definition \ref{weight-ss}) and is a blow-up of $(X,\bx,\ba)$ (cf. Definition \ref{blow-up}).
Since $\cX$ is smooth, there are integers $\{b_{\alpha}\}$ indexed by the irreducible components $\bar X\lalp$ of $\bar X$,
such that if we view $\bar X\lalp$ as divisor in $\cX$, then
$$\iota^*\sO_{\PP^m}(1)=\omega_{\bar X/T}(\ba\cdot\bx)^{\otimes k}(\Sigma_{\alpha=1}^{\bar r} b\lalp \bar X\lalp).
$$
%where $\sO_{\bar X}(\Sigma_{\alpha=1}^{\bar r}b\lalp \bar X\lalp)$ is the line bundle on $\bar X$ so that
%its restriction to $\bar X\lalp$ is isomorphic to $\sO_{\bar X\lalp}(\sum_{\beta\ne \alpha}(\bar b\lbe-\bar b\lalp)\bar X\lalp\cap\bar X\lbe)$.
%then
%$(\bar X,\bx)$ is a contraction of $(\ti X,\ti\bx)$, that is, the image of $\bar X$ using the complete linear series
%$$f: \bar X\lra \PP H^0(\bar X, \sL)\dual,
%$$
%and $\sO_{\bar X}(1)=\sL$.

Since the collection of blow-ups of $X$ coupled with integers $\{b\lalp\}_{\alpha=1}^{\bar r}$
is a discrete set, the choices of $(\bar X,\sL,\bar \bx)$ are discrete.
Thus $\{(\bar X, \iota\sta\sO_{\PP^m}(1),\bar \bx)\mid (\bar X, \iota, \bar \bx)\in\Theta_\xi\}$
is discrete.
Finally, any two $(\bar X,\iota, \bar\bx)$ with isomorphic $(\bar X,  \iota\sta\sO_{\PP^m}(1),\bar \bx)$
lie in the same $G$-orbit. Thus
%once the choice of $(\bar X,\sO_{\bar X}(1),\bar \bx)$ is given, $\eta$ is parameterized by a $G$-orbit.
$\Theta_\xi$ consists of a discrete collection of $G$-orbits. Hence $\psi\upmo(\psi(\xi))$ is discrete.
\end{proof}

%\begin{rema}
%Notice that the proof above is purely local, we actually do not need the existence of
%$\overline{\cM}_{g,\ba}$. To see that suppose $\xi_i=(X_i,\iota_i,\bx_i), i=1,2\in \bcK$
%such that they both $(X_i,\bx_i)$ are blow-up of a weighted pointed stable curve $(X,\bx)$ and
%$\q(\xi_1)\ne\q(\xi_2)\in \quotient$. Let $U_i,i=1,2$ be the analytic neighborhood of $\q(\xi_i)$.
%Again we have birational morphism between $U_i\cap(\cK\!\sslash\!{G})$, which play the role of $\overline{\cM}_{g,\ba}$.
%Then the Lemma above together with Zariski's Main theorem and properness of $\quotient$ imply that $\q(\xi_1)=\q(\xi_2)$.

We remark that this proof uses the existence of the coarse moduli space $\bcM_{g,\ba}$ constructed by Hassett.

\subsection{The coarse moduli space}
\def\cP{\mathcal P}
\def\bcP{\widetilde{\cP}}
\def\sm{{\text{sm}}}
\def\bGamma{\overline\Gamma}
\def\bn{{\mathbf n}}

We prove that $\quotient$ is a coarse moduli space of weighted pointed stable curves, thus proving that
$\psi$ is an isomorphism.

\begin{prop}\label{coarse}
Let $T$ be any scheme and $(\cX,\fx,\ba)$ be a $T$-family of  weighted pointed stable  curves. Then there is a
unique morphism $f: T\to \quotient$, canonical under base changes, such that
for any closed point $c\in T$, the image $\psi(f(c))\in\overline\cM_{g,\ba}$ is the closed point
associated to the  weighted pointed stable curve $(\cX,\fx, \ba)|_c$.
\end{prop}

We define a subscheme $\bcP\sub \cH$:
$$\bcP=\bigl\{(X,\iota, \bx)\in \cH\mid (X, \bx, \ba) \text{ weighted pointed stable curves},\
\omega_X(\bax)^{\otimes k} \cong \iota\sta\sO_\Pm(1)\bigr\}
$$
A direct check shows that $\bcP$ is a smooth, locally closed, and $G$-invariant subscheme of $\cH$.
We let $\cP\sub \bcP$ be the open subset of
$(X,\iota,\bx)$ such that $X$ are smooth. By definition, we have $\cP=\cK$.

\begin{lemm}
The composition $F: \cP\to \bcK\to \quotient$ extends to a unique morphism $\widetilde F: \bcP\to\quotient$.
\end{lemm}

\begin{proof}
Applying deformation theory of nodal curves, we know that $\cP$ is dense in $\bcP$. Let
$\Gamma\sub \cP\times\quotient$ be the graph of the morphism $F$ stated in the Lemma;
we let
$$\bGamma\sub \bcP\times\quotient
$$
be the closure of $\Gamma$. Let $p: \bGamma\to \bcP$ be the projection. We claim that $p$ is bijective.
Indeed, given $\xi=(X,\iota,\bx)\in \bar \cP$, %by the proof of the surjectivity of $\psi$, we can find a family
we let $(\cX, \sO_{\cX}(1), \fx)$ be the family given by Proposition \ref{6.6}, which shows that
$\xi\in p(\overline\Gamma)$. This proves that $p$ is surjective. On the other hand, repeating the
proof of Lemma \ref{inj}, we see that $p$ is one-to-one. This proves that $p$ is bijective.

Next, we claim that
$p$ is an isomorphism. Since $\bcP$ is smooth, $\cP\sub \bcP$ is dense, and $\Gamma$ is isomorphic
to $\cP$, we conclude that $\bGamma$ is reduced.
Then since $p:\bGamma\to\bcP$ is birational, a diffeomorphism and $\bcP$ is smooth,
$p$ must be \'etale. Thus $p$ is an isomorphim. Finally, by composing the isomorphism $p\upmo$ with the projection to the second factor of $\bcP\to\quotient$, we obtain the desired extension $\widetilde F$ of $F$.
\end{proof}

\begin{proof}[Proof of Proposition \ref{coarse}]
We cover $T$ by a collection of  affine open $\{T_a\}_{a\in A}$. Let $\pi_a: \cX_a\to T_a$ with sections $\fx_{a, i}: T_a\to\cX_a$ be the
restriction of $\fx_{i}$ to $T_a$ of the family on $T$. By fixing a trivialization
$\pi_{a\ast}\omega_{\cX_a/T_a}(\ba\cdot \fx_a)^{\otimes k}\cong \sO_{T_a}^{\oplus m+1}$, we obtain
morphisms $f_a: T_a\to \bcP$. Composed with the morphism $\widetilde F$ constructed in the previous Lemma, we obtain
$\widetilde F\circ f_a: T_a\to \quotient$.

Since the choice of picking the trivializations does not alter the morphism $\widetilde F\circ f_a$, this collection
$\{\widetilde F\circ f_a\}_{a\in A}$ patches to a
morphism $T\to \quotient$. This proves the first part of Proposition \ref{coarse}.

Finally, that $\psi(f(c))$ is the point associated to the weighted pointed curve $(\cX,\fx,\ba)|_c$ follows from the construction.
\end{proof}

\begin{proof}[Proof of Theorem \ref{moduli}] It follows from Proposition \ref{6.6}, \ref{coarse}, and Lemma \ref{inj}.
\end{proof}

%\section{Members in $\cC^{ss}$}
\vsp

For completeness, we describe without proof the geometry of poly-stable\footnote{Recall a point $\xi\in \cC^{ss}$ is
poly-stable if the $G$-orbit $G\cdot \xi$ is closed in $\cC^{ss}$.}
points in $\cC^{ss}$.
%Here is the description of poly-stable polarized weighted pointed curves.

%\begin{theo}\label{poly}
\begin{defi}[\cite{Cap} when $\bx=\emptyset$]
We call $(X,\sO_X(1),\bx,\ba)$ {\em extremal} if for any proper subcurve $Y\sub X$
satisfying $\vec\delta_Y(\sO_X(1))=\bM^-_Y $ (cf. \eqref{basic-m}),
we have $L_Y=Y\cap Y\comp\sub E_X$, where $E_X$ is the union of degree one rational curves in $X\sub \PP W$.
\end{defi}

Let $(X,\sO_X(1),\bx,\ba)$ be a slope semi-stable weighted pointed nodal curve  such that $\deg X>M$ with
$M$ given in Theorem \ref{main}.  Then the Chow point of $(X,\sO_X(1),\bx)$ is poly-stable with
respect to the polarization $\sO_{\Xi}(1,\ba)$ if and only if  it is extremal. (This in case $\bx=\emptyset$ was
proves by Caporaso in \cite{Cap}.)

\section{$K$-stability of nodal curve}\label{K-stable}
\def\fS{\mathfrak S}
\def\lbalp{_{\bar\alpha}}

In this section, we give another application of Theorem \ref{main},  which was motivated by a question of Yuji Odaka on studying the $K$-stability of a polarized nodal curve. 

\begin{theo} \label{K-stab}
A polarized connected nodal curve $(X,\sO_X(1))$ is $K$-stable if and only if $\sO_X(1)$ is numerically
equivalent to a multiple of  $\omega_X$.
\end{theo}

We comment that Odaka has proved the $K$-stability for nodal curve $X$ polarized by $\sO_X(1)=\omega_X^{\otimes k}$ for some
$k\in \NN$ \cite{Od}. His proof uses birational geometry and a weight formula proved by himself and by the
second named author independently.  He also informed us that he was able to generalize his method to
prove the above theorem.

\vsp
We recall the notion of $K$-stability of polarized varieties.

\begin{defi}[{\cite[Sect. 3]{RT}}]\label{test}
A {\em test configuration} for a polarized scheme $(X,\sO_X(1))$ consists of
a $\CC\sta$-equivariant flat projective morphism $\pi:\cX\to \Ao$,
where $\CC\sta$ acts on $\Ao$ via the usual action, and a $\CC\sta$-linearized
relative very ample line bundle $\sL$ of $\pi$,
such that for any $t\ne 0\in\Ao$, $(\cX_t,\sL_t)\cong (X,\sO_X(1))$. (Here $\sL_t=\sL|_{\cX_t}$.)
We call such test configuration $(\cX,\sL)$ a product test configuration if $\cX\cong X\times \Ao$;
we call it a trivial test configuration if in addition to that it is a product test configuration,
the line bundle $\sL$ is a pull back from $X$ and the $\CC\sta$-action is the product action that acts trivially on $X$.
\end{defi}

For notational simplicity, from now on we restrict ourselves to when $(X,\sO_X(1))$ is a polarized
nodal curve. Given a text configuration $(\cX,\sL)$, we
$w(l)$ be the weight of the induced $\CC\sta$-action on $\pi_*\sL^{\otimes l}|_{0}$;
%=H^0(\sL^l)/tH^0(\sL^l)$ by  $w(l)$.
$w(l)=a_2 l^2+a_1 l+a_0$ is a degree $2$ ($=\dim X+1$) polynomial in $l$.
We then form the quotient
%$P(l)\defeq \chi(\sO_X(l))=b_1l+b_0$, we have
$$\frac{w(l)}{l\cdot\chi(\sO_X(l))}=e_0+e_{-1}l^{-1}+\ldots.
$$

\begin{defi} \label{DF}
We define the Donaldson-Futaki invariant of a test configuration $(\cX,\sL)$ of $(X,\sO_X(1))$ be
$$\mathrm{DF}(\cX,\sL)=e_{-1}=-\frac{a_{n+1}b_{n-1}-a_n\cdot b_n}{b_n^2};
$$
%it is clear that it is unchanged after the replacing $\sO_X(1)$ by $\sO_X(k)$.A
the polarized nodal curve $(X,\sO_X(1))$ is $K$-stable if $\mathrm{DF}(\cX,\sL)<0$ for any nontrivial test configuration
$(\cX,\sL)$ of $(X,\sO_X(1))$.
\end{defi}

For $(X,\sO_X(1))$, and letting $W\dual =H^0(\sO_X(l))$ with $X\sub \PP W$ the tautological embedding,
then given any 1-PS subgroup $\lam$ of $\Aut \PP W$, the
$\CC\sta$-orbit of $X$ in $\PP W\times\Ao$ via the diagonal $\CC\sta$ action produces
a test configuration of $(X, \sO_X(1))$; we denote such test configuration by $(\cX_\lam,\sL_\lam)$.

Conversely, given $(X,\sO_X(1))$, any test configuration of
$(X,\sO_X(1))$ can be derived from a 1-PS of $\Aut\PP W$ (cf.  \cite[Prop. 3.7]{RT}).
Thus to prove the $K$-stability of $(X,\sO_X(1))$, it suffices to show that the Donaldson-Futaki invariant
$\text{DF}(\cX_\lam,\sL_\lam)<0$ for all 1-PS $\lam$ of $\Aut\PP W$.

%In the following, we let $(X,\sO_X(1))$ be a polarized nodal curve.
%%By replace $\sO_X(1)$ by $\sO_X(l)$, we may assume $l=1$.
%We set $P(1)=m+1=\dim H^0(\sO_X(1))$; we let $X\sub \PP (H^0(\sO_X(1))^\vee)$ be the tautological embedding.
%To prove that $(X,\sO_X(1))$ is $K$-stable, we suffice to show that for any
%1-PS $\lam$ of $PGL(m+1)$, the Donaldson-Futaki invariant
%$\mathrm{DF}(\cX_\lam,\sL_\lam)<0$.
\vsp

Our starting point is to relate $\mathrm{DF}(\cX_\lam,\sL_\lam)$ with the Chow weights of $(X,\sO_X(l))$.
We pick a $\lam$-diagonalizing basis
%\beq\lab{basis-s}
$\bolds=\left\{ s_0,\cdots,s_m\right\}$ of $W\dual$;
namely, under its dual bases the action $\lambda$ is given by the following (1-PS of $GL(W\dual)$):
\beq
\lambda(t):=
\mathrm{diag}[t^{\rho _0},\cdots,t^{\rho_m}], %\subset GL(m+1)
\quad
\rho_0\geq\rho_1\ge\cdots\ge \rho_m=0,\quad \rho_i\in \ZZ.
\eeq
Here by shifting all $\rho_i$ by a common integer $a$, % does not alter its action on $\PP W$,
we can choose $\rho_m=0$.

Our next step is to construct a 1-PS of $W_l\dual=H^0(\sO_X(l))$ that is most close to $\lam$.
First, since $(X,\sO_X(1))$ is $K$-stable is equivalent to that of $(X,\sO_X(k))$, $k>0$. THus without
lose of generality, we can assume $\sO_X(1)$ is sufficiently ample so that
\beq\label{phi-l}
\phi_l: S^l W\dual\lra W_l\dual=H^0(\sO_X(l)).
\eeq
For convenience, for multi-indices $I=(i_0,\cdots, i_m)$, we denote $s^I=s_1^{i_0}\cdots s_m^{i_m}$;
then $s^I$ has total degree $|I|=\sum i_j$. For the weights $\rho=(\rho_0,\cdots,\rho_m)$, we
denote $\rho(I)=\sum \rho_j \cdot i_j$, which is the weight of $s^I$ under the induced $\lam$ action on
$S^l W\dual$.

We let $\fS_l$ be the set of monomials in $S^l W\dual$.
We order $\fS_l$ as follows:
{\sl $s^I\succ s^{I'}$ when either $\rho(I)<\rho(I')$, or
when $\rho(I)=\rho(I')$ and there is a $0\leq j_0\leq m$
such that $i_j=i'_j$ for all $j>j_0$ and $i_{j_0}>i'_{j_0}$.}
Thus, $s_m^l$ (resp. $s_0^l$) is the largest (resp. least) element in $\fS_l$. Further,
$s^I\succ s^{I'}$ if and only if $s^J\cdot s^I\succ s^J\cdot s^{I'}$ for any non-trivial monomial $s^J$,
and vice versa.

We pick a basis of $W_l\dual$, which will be a diagonalizing basis of the $\lam_l$ we will construct momentarily.
Let $m_l+1=\dim W_l\dual$. We set $s_{l,m_l}=s_m^l$, with weight $\varrho_{l,m_l}=\rho_m^l$. Suppose for
an integer $0\le k<m_l$, we have picked $s_{l,k+1},\cdots,s_{l,m_l}$ and their weights $\varrho_{l,j}$, we let
$s_{l,k}=s^{I_k}$ be the largest element in
$$\bigl\{ s^I\in \fS_l\mid \phi_l(s^I)\not\in \phi_l(\Theta_{l,k+1})\bigr\},
$$
where $\Theta_{l,k+1}=\phi_l\bl\CC \{s_{l,k+1},\cdots,s_{l,m_l}\}\br$, and $\CC\{\cdot\}$
is the $\CC$-linear span of elements in $\{\cdot\}$.
We let $\rho(I_k)$, which is the weight of $s^{I_k}$ in $S^l W\dual$ under $\lam^{\otimes l}$.
%s_{\vec i}\in \fS_{l,k}\mid \phi_l\bl \CC\cdot \{ s_{\vec j}\mid s_{\vec j}\succ s_{\vec i}\}\br
%\ne  \phi_l\bl \CC\cdot \{ s_{\vec j}\mid s_{\vec j}\succeq s_{\vec i}\}\br\bigr\}.
%$$

We let $s_{l,k}=\phi_l(s^{I_k})$. Then $s_{l,0},\cdots, s_{l,m_l}$ form a basis $W_l\dual$.
By setting
\beq\label{varrho}
\lam_l(\sigma)\cdot s_{l,k}=\sigma^{\varrho_{l,k}} s_{l,k}
\eeq
we obtain a 1-PS of $\Aut \PP W_l$ of diagonalizing basis $\{s_{l,0},\cdots,s_{l,m_l}\}$.

Following the discussion in Section \ref{newton-poly}, we define, (for $q\in \ti X$
and $\ti s_{l,k}$ the lift of $s_{l,k}$ to the normalization $\ti X$ of $X$, )
\beq\label{hbar-alp-l}
\hbar\lalp(\lam_l)=\min\{i\mid \ti s_{l,i+1}|_{\ti X\lalp}=0\}
,\quad  \hbar(\lam_l,q)=\max\{i\mid v(\ti s_{l,i},q)\ne \infty\},
\eeq
and
\beq\label{lam-l}
%\Lam^{[l]}\lalp=\{ q\in X\lalp\mid s^{[l]}_{\hbar^{[l]}\lalp}(q)=0\}; \ \
%\Lam=\cup_{\alpha=1}^r \Lam^{[l]}\lalp;
\ti\Lam\lalp(\lam_l)=\{p\in\ti X\lalp\mid \ti s_{l,\hbar\lalp(\lam_l)}(p)=0\}, \quad  \ti\Lam(\lam_l)=\cup_{\alpha=1}^r \ti\Lam\lalp(\lam_l).
\eeq
(Here $r$ is the number of irreducible components of $X$.)

\begin{lemm}\label{w-l} For the 1-PS $\lam_l$ constructed, we have
$\ti\Lam(\lam_l)=\ti\Lam(\lam)$. Further, for each $q\in \ti\Lam(\lam)$, and for $w(\ti\sI,q)$ defined in \text{\eqref{hbar}}, we have
$v(\ti s_{l,\hbar(\lam,q)},q)=w(\ti\sI^l,q)= l\cdot w(\ti\sI,q)= l\cdot v(\ti s_{\hbar(q)},q)$, and hence
$e( \sI(\lam_l))= \nlc\, \chi ( \sO_{X\times \Ao}(k)/\sI(\lam)^{kl})$.

Let $\lam_l'$ be
the staircase 1-PS obtained from $\lam_l$ by applying Proposition \ref{stair-case}, then
\begin{enumerate}
\item
the support  $\ti \Lam(\lam_l)$ is the same as $\ti\Lam(\lam_l')$; for each $q\in \ti\Lam(\lam_l)$,
$w(\ti\sI(\lam_l),q)=w(\ti\sI(\lam_l'),q)$;
\item
for each $q\in \ti\Lam(\lam_l)$, $\Delta_q(\lam_l)=l\cdot \Delta_q(\lam)\sub\Delta_q(\lam'_l)$.
\end{enumerate}
\end{lemm}

\begin{proof} Because of our choice of  $\lam_l$, we have the middle identity (the first the the third is by the definition)
$$
(s_0,\cdots, s_m)= \sI(\lam_l)=\sI (\lam)^l=(t^{\rho _0}s_0,\cdots,t^{\rho_m}s_m)^l \sub \sO_{X\times \Ao}(l) .
$$
This prove the first part of the Lemma. The second part,
follows from the construction of staircase in Proposition \ref{stair-case}, and from Lemma \ref{newton}.
\end{proof}

We have the following useful Lemma, relating $\text{DF}(\cX_\lam,\sL_\lam)$ to the weights $\omega(l)$ of
$\lam$.

%and let $X\sub \PP H^0(\sO_X(l))^\vee $ be the embedding given by the basis $\{s^{[l]}_i\}$ we constructed above.
%Then we have the following

\begin{lemm}[{\cite[Sect. 2.3]{Do}, \cite[Thm. 3.9]{RT}}]\label{asymp}
Let $X\sub \PP W\dual$ and $\lam$ a 1-PS be as before.  Then
%$$
%\lim_{l\to \infty}\frac{\omega(\lam_l)}{l}=\lim_{l\to\infty}\frac{1}{l}\Bl \frac{2l\deg X\cdot \sum_{i=0}^{m_l}\varrho_{l,i}}{l\cdot \deg X+1-g_X}
%-e\bl \sI(\lam_l)\br\Br=-\frac{1}{ b_1}\mathrm{DF}(\cX_\lam,\sL_\lam)<\infty\ .
%$$
$$
\lim_{l\to \infty}l\upmo\cdot \omega(\lam_l)
%=\lim_{l\to\infty}\frac{1}{l}\Bl \frac{2l\deg X\cdot \sum_{i=0}^{m_l}\varrho_{l,i}}{l\cdot \deg X+1-g_X}-e\bl \sI(\lam_l)\br\Br
=-b_1\upmo\cdot \mathrm{DF}(\cX_\lam,\sL_\lam)<\infty\ .
$$
\end{lemm}

Thus to prove $\text{DF}(\cX_\lam,\sL_\lam)<0$, it suffices to show that
\beq\label{req}
\lim_{l\to \infty}l\upmo\cdot \omega(\lam_l)>0.
%\lim_{l\to\infty}\frac{1}{l}\Bl \frac{2l\deg X\cdot \sum_{i=0}^{m_l}\varrho_{l,i}}{l\cdot \deg X+1-g_X}-e( \sI(\lam_l))\Br>0.
\eeq

%This together with the theory we developed in Section \ref{newton-poly} imply that
%$\Delta_{q}(\lam_l)=l\cdot \Delta_q(\lam)$
%and
%\beq\label{tip-Delta}
%|\Delta_{q}(\lam_l)\cap[0,l\cdot w(\ti\sI,q)]\times[l\cdot\rho_{\hbar(q)},\infty]|=l^2\cdot |\Delta_q\cap[0,w(\ti\sI,q)]\times[\rho_{\hbar(q)},\infty]|\
%\eeq
%for each $q\in \ti \Lam(\lam_l)=\ti\Lam$.

%We now consider the case when $e(\sI(\lam))>0$ and $X$ is irreducible.
Let $\lam_l'$ be the staircase constructed from $\lam_l$ using Proposition \ref{stair-case},
of the same weights $\varrho_{l,i}$. We let $\hat\varrho_{l,i}$ be the shifted weights according to
the rule \eqref{hat-rho} applied to $\lam_l'$; namely,
$\hat{\varrho}_{l,i}=\min_{\alpha}\{\varrho_i-\varrho_{\hbar\lbe(\lam_l')}\mid i\in \II\lbe(\lam_l')\}$.

\begin{proof}[Proof of Theorem \ref{K-stab}]
Suppose $X$ is a stable (nodal) curve and
$\sO_X(1)$ is numerically proportional to $\omega_X$, then $(X,\sO_X(1))$ is slope stable.  We will show in this
case that for any 1-PS $\lam\sub SL(W)$ we have
$\text{DF}(\cX_\lam,\sL_\lam)<0$. %To do that, we divide into two cases: either $e(\sI(\lam))=0$ or $e(\sl(\lam))>0$.

The first case to study is when $e(\sI(\lam))=0$.
By Section \ref{newton-poly} we know that there is a $0<i_0<m$ such that $\varrho_{i_0}=0$ and
$\bigcap_{k\geq i_0}\{s_k=0\}=\emptyset$.
Stoppa proved that in this case either the test configuration $(\cX_{\lam},\sL)$ induced by $\lam$ is trivial
(cf. Definition \ref{test}) or $\mathrm{DF}(\lam)<0$ \cite[page 1405-1406]{St}. This settles this case.

The other case is when $e(\sI(\lam))>0$.
Applying Theorem \ref{lower-bdd} (since $(X,\sO_X(1))$ is slope stable), and applying Proposition \ref{stair-case},
we can find an $\epsilon>0$ so that
\beq\label{111}
l\upmo\cdot \omega(\lam_l)\geq l\upmo\cdot \omega(\lam_l')
%\frac{1}{l}\Bl \frac{2l\deg X\cdot \sum_{i=0}^{m_l}\varrho_{l,i}}{l\cdot \deg X+1-g_X}-e( \sI(\lam_l))\Br\geq
\geq\frac{1}{\deg X+l\upmo(1-g_X)}\cdot \frac{\epsilon}{l^2}\cdot {\sum_{i=0}^{m_l}\hat{\varrho}_{l,i}}.
\eeq
Thus $\text{DF}(\cX_\lam,\sL_\lam)<0$ follows from Lemma \ref{l2}, which we will prove shortly.
This proves the {\em if} part, once we prove Lemma \ref{l2}.

For the other direction, suppose $(X,\sO_X(1))$ is $K$-stable, we show that
$(X,\sO_X(1))$ is slope stable.
%{\red Since $K$-stability is independent on $l$, without loss of generality we may assume $l=1$. }
Suppose $(X,\sO_X(1))$ is not slope stable, %Letting $\sL=\sO_X(l)$,
then there is a subcurve $Y\sub X$ destabilizing the polarized curve $(X, \sO_X(1))$, namely,
\beq\label{dest}
\frac{\deg_Y\omega _X}{\deg \omega_X}\cdot\deg X-\deg Y-\frac{\ell_Y}{2}\ge 0.
\eeq
Let $H^0(\sO_X(1)|_Y)\dual \sub W=H^0(\sO_X(1))\dual$, which is the linear subspace spanned by $Y$;
let $m_0+1=\dim H^0(\sO_X(1)|_Y)$.
We choose a two-weight 1-PS $\lam$ as in the proof of Theorem \ref{main} (at the end of Section \ref{proof})
so that $\lam$ acts of weight $1$ on $H^0(\sO_X(1)|_Y)\dual\sub H^0(\sO_X(1))\dual$ and acts
of weight $0$ on its linear complement; we form its associated test configuration $(\cX_\lam,\sL_\lam)$.
We now evaluate
$$\mathrm{DF}(\cX_\lam,\sL_\lam)=\lim_{l\to\infty} -\frac{\omega(\lam_l)}{l}=
\lim_{l\to\infty}\frac{1}{l}\cdot\Bl \frac{2l \deg X\sum_{i=0}^{m_l}\varrho_{l,i}}{l\deg X+1-g}- e(\sI(\lam_l))\Br.
$$

First, the central fiber $\cX_0$ is the union $Y\cup E\cup Y\comp$, where $E$ consists of
$\ell_Y$ number of lines that inserted in the linking nodes $L_Y$ of $X$;
the total space $H^0(\sL_\lam^{\otimes l}|_{\cX_0})$ has a decomposition
$$H^0(\sL_\lam^{\otimes l}|_{\cX_0})
\cong H^0(\sO_X(l)|_Y)\oplus H^0(\sO_E(l)(-Y\cap E))\oplus H^0(\sO_X(l)|_{Y\comp}(-E\cap Y\comp));
$$
elements in $H^0(\sO_X(l)|_Y)$ have weights $l$; elements in $H^0(\sO_E(l)(-Y\cap E))$ have weights
$l-1, \cdots, 0$ for each copy in $E$; elements in $H^0(\sO_X(l)|_{Y\comp}(-E\cap Y\comp))$ have weights $0$.
Thus
$$\sum\varrho_{l,i}=h^0(\sO_X(l)|_Y)\cdot l+\ell_Y\cdot \frac{l(l-1)}{2}=
 (\deg Y+\frac{\ell_Y}{2})\cdot l^2+(1-g_Y-\frac{\ell_Y}{2})\cdot l.
$$
Using that $\sum_{i=0}^{m_l}\varrho_{l,i}=e(\sI(\lam))\cdot \frac{l^2}{2}+\text{lower order term}$, and simplifying, we obtain
\beq\label{DF-2wt}
\mathrm{DF}(\cX_\lam,\sL_\lam)=\frac{g-1}{\deg X}\Bl\frac{\deg_Y\omega _X}{\deg \omega_X}\cdot\deg X-\deg Y-\frac{\ell_Y}{2}\Br.
\eeq
Since $Y\sub X$ is destabilizing, by \eqref{dest} we have $\mathrm{DF}(\cX_\lam,\sL_\lam)>0$, violating
that $(X,\sO_X(1))$ is $K$-stable.  This proves that $(X,\sO_X(1))$ is slope stable.

On the other hand, since $K$-stability
of $(X,\sO_X(l))$ is independent of $l>0$, $(X,\sO_X(1))$ is $K$-stable implies that
$(X,\sO_X(l))$ is $K$-stable for $l>0$, thus $(X,\sO_X(l))$ is slope stable.

Once we know that $(X,\sO_X(l))$ is slope stable for large $l$, an easy argument shows that
$\sO_X(l)$ satisfies \eqref{basic} for all large $l$,
which is possible
only when $\sO_X(1)$ is numerically proportional to $\omega_X$.
This proves the Theorem.
\end{proof}

\begin{lemm}\label{l2}
Let the notation be as stated. Suppose further that
$e(\sI(\lam))>0$. Then\,\footnote{Here by $l=l_k\to \infty$ we mean that by passing to a subsequence we assume that the limit $\lim_{l_k\to\infty}$
does exist, and is finite.}
$$\lim_{l=l_k\to\infty}\frac{1}{l^2}\cdot {\sum_{i=0}^{m_l}\hat{\varrho}_{l,i}}>0.
$$
\end{lemm}

\begin{proof} We comment that in case $X$ is irreducible, the positivity is immediate. Indeed,
applying \cite[Prop. 2.11 ]{Mum}, we have
\beq\label{rho-l}
%\sum_{i=0}^{m_l}\hat\varrho_{l,i}=
\sum_{i=0}^{m_l}\varrho_{l,i}=e(\sI(\lam))\cdot \frac{l^2}{2}+a_1\cdot l+a_2, \quad a_i \text{ depending only on }\lam.
%\footnote{\red Notice that $e(\sI(\lam_l))=l\cdot e(\sI(\lam))$, (\ref{rho-l}) immediately implies Lemma \ref{asymp}. }
\eeq
Suppose $X$ is irreducible, we have $\varrho_{l,i}=\hat\varrho_{l,i}$. Therefore,
$$\lim_{l=l_k\to\infty}\frac{1}{l^2}\cdot {\sum_{i=0}^{m_l}\hat{\varrho}_{l,i}}=
\lim_{l=l_k\to\infty}\frac{1}{l^2}\cdot {\sum_{i=0}^{m_l}{\varrho}_{l,i}}=\frac{e(\sI(\lam))}{2}>0.
$$

We now prove the general case. We claim that there is a $1\le \beta\le r$ and $q\in \ti X\lbe$ so that
\beq\label{no-sq}
|\Delta_q(\lam)|-\rho_{\hbar\lbe(\lam)}\cdot w(\ti\sI(\lam),q)>0.
\eeq
Suppose for any $q\in\ti X\lbe$ the inequality \eqref{no-sq} does not hold. Since the $\ge$ always hold, we will have
that $\varrho_i=\varrho_{\hbar\lalp}$ for every $i\in \II\lalp$. Since $e(\sI(\lam))>0$,
we must have an $\alpha> 1$ such that  $\varrho_{\hbar_{\alpha}}>0$.
Since $X$ is connected, we can find a pair $\alpha\ne \beta$ so that $X\lalp\cap X\lbe\ne \emptyset$,
and $\rho_{\hbar\lalp(\lam)}> \rho_{\hbar\lbe(\lam)}=0$.

We next let $q\in\ti X\lbe$ be a lift of a node in $X\lalp\cap X\lbe$. We show that
the pair $(\beta,q)$ satisfies the inequality \eqref{no-sq}. Let $\pi: \ti X\to X$ be the projection. Since $\pi(q)\in X\lalp$,
we have $\ti s_{j}(q)=0$ for all $j>\hbar\lalp(\lam)$; since $\rho_{\hbar\lalp(\lam)}> \rho_{\hbar\lbe(\lam)}=0$, we have
$i_0(q)\le \hbar\lalp(\lam)$. Thus $\rho_{i_0(q)}\geq\rho_{\hbar\lalp(\lam)}>0$, and then $\Delta_q(\lam)$ is two dimensional.
Since $\rho_{\hbar\lbe(\lam)}=0$, this contradicts to initial assumption that \eqref{no-sq} never holds. This proves the claim.

Let $(\beta,q)$ be a pair satisfying \eqref{no-sq}. We next establish the following two inequalities
\beq\label{in-1}
\sum_{i}\hat{\varrho}_{l,i}
\geq
\frac{1}{2}\bl |\Delta_q(\lam_l')|-\varrho_{l,\hbar\lbe(\lam_l')}\cdot w(\ti\sI(\lam_l'),q)\br
%+\hat{\varrho}_{i_0(p)}>\frac{1}{2}\Bl |\Delta_q(\lam_l')|-\varrho_{\hbar\lalp(\lam_l')}\cdot w(\ti\sI(\lam_l'),q)\Br\ .
\eeq
and
\beq\label{in-2}
\lim_{l=l_k\to \infty}l^{-2}\cdot \bl |\Delta_q(\lam_l')|-\varrho_{l,\hbar\lbe(\lam_l')}\cdot w(\ti \sI(\lam_l'),q)\br
\geq |\Delta_q(\lam)|-\varrho_{l,\hbar\lbe(\lam)}\cdot w(\ti\sI(\lam),q).
\eeq

%we have $\ti s_{\hbar_{\alpha+1}}|_{\ti X\lbe}\ne 0$. Thus
%$w(\ti \sE_{\hbar_{\alpha'}+1},q)>0$ (cf. \eqref{w} for the definition of $\ti \sE_{j}$).
%
%
%for each $p\in Y_1\cap Y_2$, hence $\hbar_{\alpha'}\in \II_{Y_1}$ and $\varrho_{\hbar_{\alpha'}}>0=\varrho_{\hbar_{Y_1}}$.  As a consequence,  we have
%$\varrho_{\hbar_{\alpha'}}>0$ and $w(\ti \sE_{\hbar_{\alpha'}+1},p)>0$. These  imply
%$$|\Delta_p(\lam)|-\varrho_{\hbar_{X\lalp}}(\lam)\cdot w(\ti\sI(\lam),p)=|\Delta_p(\lam)|-\varrho_{\hbar_{Y_1}}(\lam)\cdot w(\ti\sI(\lam),p)=|\Delta_p|> 0$$
%contradicting to the assumption we made in the beginning of the proof.  So the proof is completed.
%
%
% We then let $Y_1=\cup_{\alpha< \alpha'} X\lalp$ and  $Y_2=\cup_{\alpha\geq \alpha'} X\lalp$.
%We pick a $\beta<\alpha'$ and a %$\ti s_{\hbar_{\alpha
%
%Let  us denote $\varrho_{\hbar_{Y_1}}=\varrho_{\hbar_{\alpha'}}>\varrho_{\hbar_{Y_2}}=\varrho_{\alpha'}=0$. Since $Y_1\cap Y_2\sub Y_1\sub \{s_{\hbar_{\alpha'}+1}=\cdots=s_m=0\}$, this implies that

%By reindexing $X\lalp$ we assume $B_{\hbar\lalp}$ is increasing and $\bar\alpha$ is the largest
%$\alpha$ so that $B_{\hbar_{\bar\alpha}}=0$.
%Then $\varrho_{\hbar_2(\lam'_l)}$ increases at least linearly in $l$.
%By assumption $\sO_X(1)$ is numerically equivalent to $\omega_X$, which is positive; hence $X$ has no non-trivial
%automorphism. Thus $\lam$  do not leave $X$ invariant,
%\begin{clai}\label{not-rect}

We prove the inequality \eqref{in-1}.
Following the notation introduced in Section \ref{main-estimate}, we have
$$\sum_{i=0}^{m_{l}}\hat{\varrho}_{l,i}
\geq \sum_{i\in \II\lalp(\lam_l')}\hat{\varrho}_{l,i}
\geq \sum_{i\in \II_q\prim(\lam_l')}\hat{\varrho}_{l,i},
$$
where $\II\lbe(\lam_l')$ is the set of indices for $\ti X\lbe$, and $\II_q\prim(\lam_l')$ is the set of primary indices
for $q\in \ti X\lbe$, both with respect to the staircase $\lam_l'$.

By Proposition \ref{vir-ind} and \ref{vir-ind-2}, we know that for $i_0(q)\ne i\in \II_q\prim(\lam_l')$, we have
$\hat{\varrho}_{l,i}=\varrho_{l,i}-\varrho_{l,\hbar\lbe(\lam_l')}$. (Note $\hat{\varrho}_{l,i_0(q)}=\varrho_{l,i_0(q)}-
\varrho_{l,\hbar_{\alpha'}(\lam_l')}$ possibly for some $\alpha'\ne\beta$.)

By the proof of  Lemma \ref{trapezoid}, we have
\begin{eqnarray*}
\sum_{i\in \II_q\prim(\lam_l')}\hat{\varrho}_{l,i}-\hat{\varrho}_{l,i_0(p)}
&\geq&|\Delta_p\prim(\lam_l')\cap([1,w\prim(p,\lam_l')]\times \RR)|-\varrho_{l,\hbar\lbe(\lam_l')}\cdot (w\prim(p,\lam_l')-1).
\end{eqnarray*}
Following \eqref{Wp}, we continue to denote
$\bar\jmath_p(\lam_l')=\max\{i\in \II_p\prim(\lam_l')\}$ and
$w\prim(p,\lam_l')=w(\ti\sE_{\bar\jmath_p(\lam_l')+1}(\lam_l'),p)$.
By the boundness result from Corollary \ref{bdd-tail}, for sufficiently large $l$,
the effects to the shape of $\Delta_q(\lam_l')$ from the {\em secondary indices} $\II_q(\lam_l')\setminus\II_q\prim(\lam_l')$
is marginal, thus for large $l$ we have
%\begin{eqnarray*}
%&&|\Delta_q\prim(\lam_l')\cap([1,w\prim(q,\lam_l')]\times \RR)|-\varrho_{l,\hbar\lbe(\lam_l')}\cdot (w\prim(q,\lam_l')-1)\\
%&\geq&
%\frac{1}{2}\Bl |\Delta_q(\lam_l')|-\varrho_{l,\hbar\lbe(\lam_l')}\cdot w(\ti\sI(\lam_l'),q)\Br ,
%\end{eqnarray*}
$$\big|\Delta_q\prim(\lam_l')\cap([1,w\prim(q,\lam_l')]\times \RR)\big|-\varrho_{l,\hbar\lbe(\lam_l')}\cdot \bl w\prim(q,\lam_l')-1\br
\geq
\frac{1}{2}\Bl |\Delta_q(\lam_l')|-\varrho_{l,\hbar\lbe(\lam_l')}\cdot w(\ti\sI(\lam_l'),q)\Br.
$$
Combined, and adding $\hat{\varrho}_{i_0(p)}>0$, we obtain
$$
\sum_{i\in \II_q\prim(\lam_l')}\hat{\varrho}_{l,i}
\geq
\frac{1}{2}\Bl |\Delta_q(\lam_l')|-\varrho_{l,\hbar\lbe(\lam_l')}\cdot w(\ti\sI(\lam_l'),q)\Br.
%+>\frac{1}{2}\Bl |\Delta_q(\lam_l')|-\varrho_{\hbar\lbe(\lam_l')}\cdot w(\ti\sI(\lam_l'),q)\Br\ .
$$
This proves \eqref{in-1}.

%In case, $\varrho_{\hbar\lalp(\lam_l')}=\varrho_{\hbar\lalp(\lam_l)}=0$.  These together with  Lemma \ref{w-l} imply
%\beq\label{delta>0}
%\frac{1}{l^2}\sum_{i=0}^{m_{l}}\hat{\varrho}_{l,i}\geq
%\frac{1}{l^2}|\Delta_q(\lam_l')|\geq\frac{1}{l^2}|\Delta_q(\lam_l)|=
% |\Delta_q(\lam)|>0.
%\eeq
%Hence our proof is completed for this case.

Before we move to \eqref{in-2}, we claim that
\beq\label{eq-3}
A\lbe:=\lim_{l=l_k\to \infty}\frac{\varrho_{l,\hbar\lbe(\lam_l')}-\varrho_{l,\hbar\lbe(\lam_l)}}{l}= 0.
\eeq
Suppose not, say $A\lbe>0$, (it is non-negative,) then for
$l=l_k$ large,
$$\varrho_{l,\hbar\lbe(\lam_{l}')}-\varrho_{l,\hbar\lbe(\lam_{l}')}\geq  \frac{1}{2}\cdot l\cdot A\lbe;
$$
by examining the geometry of $\Delta_q(\lam_{l})\sub \Delta_q(\lam_{l}')$, we obtain
$$|\Delta_q(\lam_{l}')|-|\Delta_q(\lam_{l})|\geq \frac{A\lbe}{2} \cdot \frac{l\cdot w(\ti \sI(\lam),q)}{h_{\Delta_q(\lam)}-\varrho_{l,\hbar\lbe(\lam)}}\cdot l
\defeq C\cdot l^2 >0,
$$
where $h_{\Delta_q(\lam)}$ is the height of  $\Delta_q(\lam)$.
This implies
$$l\upmo\cdot \omega(\lam_l)=l\upmo\cdot \omega(\lam_l')+l\upmo\cdot \bl e(\sI(\lam_{l}')-e(\sI(\lam_{l})\br
\geq l\upmo\cdot \bl e(\sI(\lam_{l}')-e(\sI(\lam_{l})\br,
$$
where we have used Theorem \ref{main} to deduce $\omega(\lam_l')\geq 0$.

%$$\geq \frac{1}{l}\cdot (e(\sI(\lam_{l}')-e(\sI(\lam_{l}))
%$$\frac{1}{l}\Bl \frac{2l\deg X\cdot \sum_{i=0}^{m_{l}}\varrho_{i,l}}{l\cdot \deg X+1-g_X}-e( \sI(\lam_{l}))\Br\\
%=\frac{1}{l}\Bl \frac{2l\deg X\cdot \sum_{i=0}^{m_{l}}\varrho_{i,l}}{l\cdot \deg X+1-g_X}-
%e( \sI(\lam_{l}'))\Br+\frac{1}{l}\cdot (e(\sI(\lam_{l}')-e(\sI(\lam_{l}))\\
%$$
%$$\geq \frac{1}{l}\cdot (e(\sI(\lam_{l}')-e(\sI(\lam_{l})),
%$$
%where we have use Theorem \ref{main} to guarantee the non-negativity of the first term in the second line to obtain the inequality.
By Corollary \ref{sum-Delta} and our construction of staircase using Proposition \ref{stair-case}, we deduce
$$
l\upmo\cdot  (e(\sI(\lam_{l}')-e(\sI(\lam_{l}))
\geq l\upmo\cdot  (|\Delta_q(\lam_{l}')|-|\Delta_q(\lam_{l})|)>C\cdot l\ .
$$
This is impossible since Lemma \ref{asymp} implies that the left-hand-side remains  bounded as $l=l_k\to \infty$.
So we must have  $A\lbe=0$. This proves the claim.

We prove inequality \eqref{in-2}. Because $A\lbe=0$, $|\Delta_q(\lam_l')|\ge |\Delta_q(\lam_l)|$, and by Lemma \ref{w-l}, we obtain
\begin{eqnarray*}
&&|\Delta_q(\lam_l')|-\varrho_{l,\hbar\lbe(\lam_l')}\cdot w(\ti \sI(\lam_l'),q)\\
&=&|\Delta_q(\lam_l')|-\varrho_{l,\hbar\lbe(\lam_l)}\cdot w(\ti \sI(\lam_l),q)+\varrho_{l,\hbar\lbe(\lam_l)}\cdot w(\ti \sI(\lam_l),q)-\varrho_{l,\hbar\lbe(\lam_l')}
\cdot w(\ti \sI(\lam_l'),q)\\
&\geq & |\Delta_q(\lam_l)|-\varrho_{l,\hbar\lbe(\lam_l)}\cdot w(\ti \sI(\lam_l),q)+(\varrho_{l,\hbar\lbe(\lam_l)}-\varrho_{l,\hbar\lbe(\lam_l')})\cdot
w(\ti \sI(\lam_l),q)\\
&=& l^2\!\!\cdot(|\Delta_q(\lam)|-\varrho_{l,\hbar\lbe(\lam)}\!\cdot\! w(\ti\sI(\lam),q))+l^2\cdot \frac{\varrho_{l,\hbar\lbe(\lam_l)}-
\varrho_{l,\hbar\lbe(\lam_l')}}{l}\cdot w(\ti \sI(\lam),q).
\end{eqnarray*}
Taking limit as $l=l_k\to\infty$, and using $A\lbe=0$, we obtain \eqref{in-2}.

Finally, by \eqref{rho-l}, and that $0\le\hat\varrho_{l,i}\le \varrho_{l,i}$, we
conclude that the limit in the Lemma
%$\lim_{l=l_k\to\infty}\frac{1}{l^2}\cdot {\sum_{i=0}^{m_l}\hat{\varrho}_{l,i}}$
is finite; thus the limit is finite and positive; this proves the Lemma.
\end{proof}

%Once again by Section \ref{main-estimate},  Lemma \ref{w-l} and the inequality above we deduce
%$$ %\beq\label{delta1>0}
%\lim_{l=l_k\to \infty}\frac{1}{l^2}\cdot \sum_{i=0}^{m_{l}}\hat{\varrho}_{l,i}\geq
%\lim_{l=l_k\to \infty}\frac{1}{l^2}\bl |\Delta_q(\lam_l')|-\varrho_{l,\hbar\lbe(\lam_l')}\cdot w(\ti \sI(\lam_l'),q)\br
%\geq |\Delta_q(\lam)|-\varrho_{l,\hbalbe(\lam)}\cdot w(\ti\sI(\lam),q)>0,
%$$ %\eeq
%which is exactly what we want.

\begin{rema}
%It is easy to see that  if $X$ is irreducible then the assumption that $\sO_X(1)$ is numerically equivalent to $k\omega$ is unnecessary. Further more, by
Following \cite[Sect. 3]{Mum}, (or the recent work of Odaka \cite{Od},) we know that a $K$-stable polarized curve
has at worst nodal singularity. Thus results like Theorem \ref{K-stab}
show that $K$-stability compactifies the moduli of smooth curves (of $g\ge 2$).
As $K$-stability is an analytic version of GIT stable via a CM-line bundle defined by Paul and Tian \cite{PT},
which is  a multiple of
$\lam^{\otimes 12}\otimes \delta^{-1}$  for moduli of curves  (cf. \cite[Thm. 5.10]{Mum}),
Theorem \ref{K-stab} can be viewed as comparing compactifications via two versions of GIT stability, one via finite dimensional
embedding and one via analysis. Generalizing this to high dimensional canonically polarized varieties
remains a challenge. Lately, Yuji Odaka (cf. \cite{Od}) has made important progress along this direction. 
\end{rema}

\end{document}